\documentclass{amsart}

\usepackage{amsfonts,amssymb,verbatim,amsmath,amsthm,latexsym,textcomp,amscd, hyperref}
\usepackage{epsfig}
\usepackage{amsmath,amsthm,graphics}
\usepackage{graphicx}
\usepackage[dvipsnames]{xcolor}
\usepackage{tikz,tikz-cd,faktor}
\usepackage{color}
\usepackage{url}
\usepackage{placeins}
\usepackage{stackrel}

\let\Oldsection\section
\renewcommand{\section}{\FloatBarrier\Oldsection}

\let\Oldsubsection\subsection
\renewcommand{\subsection}{\FloatBarrier\Oldsubsection}

\let\Oldsubsubsection\subsubsection
\renewcommand{\subsubsection}{\FloatBarrier\Oldsubsubsection}

\begin{document}

\newtheorem{theorem}{Theorem}[section]
\newtheorem{prop}[theorem]{Proposition}
\newtheorem{lemma}[theorem]{Lemma}
\newtheorem{cor}[theorem]{Corollary}
\newtheorem{cond}[theorem]{Condition}
\newtheorem{ing}[theorem]{Ingredients}
\newtheorem{conj}[theorem]{Conjecture}
\newtheorem{claim}[theorem]{Claim}
\newtheorem{constr}[theorem]{Construction}
\newtheorem{rem}[theorem]{Remark}
\newtheorem{scheme}[theorem]{Scheme}
\newtheorem{assume}[theorem]{Standing Assumption}
\newtheorem{problem}{Problem}
\newtheorem*{question}{Question}

\newtheorem*{theorem*}{Theorem}
\newtheorem*{modf}{Modification for arbitrary $n$}
\newtheorem{qn}[theorem]{Question}
\newtheorem{condn}[theorem]{Condition}

\theoremstyle{definition}
\newtheorem{defn}[theorem]{Definition}
\newtheorem{eg}[theorem]{Example}
\newtheorem{rmk}[theorem]{Remark}

\theoremstyle{definition}
\newtheorem{thmx}{Theorem}
\renewcommand{\thethmx}{\Alph{thmx}} 

\newcommand{\map}{\rightarrow}
\newcommand{\pwfm}{\textrm{piecewise Fuchsian Markov\,}}
\newcommand{\cfm}{\textrm{completely folding map\,}}
\newcommand{\hdm}{\textrm{higher degree map without folding\,}}
\newcommand{\boundary}{\partial}
\newcommand{\C}{{\mathbb C}}
\newcommand{\da}{{\mathbf d}A}
\newcommand{\hA}{{\widehat A}}
\newcommand{\disk}{{\mathbb D}}
\newcommand{\integers}{{\mathbb Z}}
\newcommand{\natls}{{\mathbb N}}
\newcommand{\ratls}{{\mathbb Q}}
\newcommand{\reals}{{\mathbb R}}
\newcommand{\proj}{{\mathbb P}}
\newcommand{\lhp}{{\mathbb L}}
\newcommand{\tr}{{\operatorname{Tread}}}
\newcommand{\rs}{{\operatorname{Riser}}}
\newcommand{\tube}{{\mathbb T}}
\newcommand{\bS}{{\mathbb S}}
\newcommand{\goa}{{GO_A}}
\newcommand{\goah}{{\hhat{GO_A}}}
\newcommand{\boa}{{BO_A}}
\newcommand{\dbS}{{\mathbf d}{\mathbb S}}
\newcommand\AAA{{\mathcal A}}
\newcommand\BB{{\mathcal B}}
\newcommand\CC{{\mathcal C}}
\newcommand\ccd{{{\mathcal C}_\Delta}}
\newcommand\DD{{\mathcal D}}
\newcommand\EE{{\mathcal E}}
\newcommand\FF{{\mathcal F}}
\newcommand\GG{{\mathcal G}}
\newcommand\HH{{\mathcal H}}
\newcommand\II{{\mathcal I}}
\newcommand\JJ{{\mathcal J}}
\newcommand\KK{{\mathcal K}}
\newcommand\LL{{\mathcal L}}
\newcommand\MM{{\mathcal M}}
\newcommand\NN{{\mathcal N}}
\newcommand\OO{{\mathcal O}}
\newcommand\PP{{\mathcal P}}
\newcommand\QQ{{\mathcal Q}}
\newcommand\RR{{\mathcal R}}
\newcommand\SSS{{\mathcal S}}
\newcommand\TT{{\mathcal T}}
\newcommand\ttt{{\mathcal T}_T}
\newcommand\tT{{\widetilde T}}
\newcommand\UU{{\mathcal U}}
\newcommand\VV{{\mathcal V}}
\newcommand\WW{{\mathcal W}}
\newcommand\XX{{\mathcal X}}
\newcommand\YY{{\mathcal Y}}
\newcommand\ZZ{{\mathcal Z}}
\newcommand\CH{{\CC\HH}}
\newcommand\TC{{\TT\CC}}
\newcommand\EXH{{ \EE (X, \HH )}}
\newcommand\GXH{{ \GG (X, \HH )}}
\newcommand\GYH{{ \GG (Y, \HH )}}
\newcommand\PEX{{\PP\EE  (X, \HH , \GG , \LL )}}
\newcommand\MF{{\MM\FF}}
\newcommand\PMF{{\PP\kern-2pt\MM\FF}}
\newcommand\ML{{\MM\LL}}
\newcommand\mr{{\RR_\MM}}
\newcommand\tmr{{\til{\RR_\MM}}}
\newcommand\PML{{\PP\kern-2pt\MM\LL}}
\newcommand\GL{{\GG\LL}}
\newcommand\Pol{{\mathcal P}}
\newcommand\half{{\textstyle{\frac12}}}
\newcommand\Half{{\frac12}}
\newcommand\Mod{\operatorname{Mod}}
\newcommand\Area{\operatorname{Area}}
\newcommand\ep{\epsilon}
\newcommand\hhat{\widehat}
\newcommand\Proj{{\mathbf P}}
\newcommand\U{{\mathbf U}}
 \newcommand\Hyp{{\mathbb H}}
\newcommand\D{{\mathbb D}}
\newcommand\Z{{\mathbb Z}}
\newcommand\R{{\mathbb R}}
\newcommand\s{{\Sigma}}
\renewcommand\P{{\mathbb P}}
\newcommand\Q{{\mathbb Q}}
\newcommand\E{{\mathbb E}}
\newcommand\til{\widetilde}
\newcommand\length{\operatorname{length}}
\newcommand\BU{\operatorname{BU}}
\newcommand\gesim{\succ}
\newcommand\lesim{\prec}
\newcommand\simle{\lesim}
\newcommand\simge{\gesim}
\newcommand{\simmult}{\asymp}
\newcommand{\simadd}{\mathrel{\overset{\text{\tiny $+$}}{\sim}}}
\newcommand{\ssm}{\setminus}
\newcommand{\diam}{\operatorname{diam}}
\newcommand{\pair}[1]{\langle #1\rangle}
\newcommand{\T}{{\mathbf T}}
\newcommand{\inj}{\operatorname{inj}}
\newcommand{\pleat}{\operatorname{\mathbf{pleat}}}
\newcommand{\short}{\operatorname{\mathbf{short}}}
\newcommand{\vertices}{\operatorname{vert}}
\newcommand{\collar}{\operatorname{\mathbf{collar}}}
\newcommand{\bcollar}{\operatorname{\overline{\mathbf{collar}}}}
\newcommand{\I}{{\mathbf I}}
\newcommand{\tprec}{\prec_t}
\newcommand{\fprec}{\prec_f}
\newcommand{\bprec}{\prec_b}
\newcommand{\pprec}{\prec_p}
\newcommand{\ppreceq}{\preceq_p}
\newcommand{\sprec}{\prec_s}
\newcommand{\cpreceq}{\preceq_c}
\newcommand{\cprec}{\prec_c}
\newcommand{\topprec}{\prec_{\rm top}}
\newcommand{\Topprec}{\prec_{\rm TOP}}
\newcommand{\fsub}{\mathrel{\scriptstyle\searrow}}
\newcommand{\bsub}{\mathrel{\scriptstyle\swarrow}}
\newcommand{\fsubd}{\mathrel{{\scriptstyle\searrow}\kern-1ex^d\kern0.5ex}}
\newcommand{\bsubd}{\mathrel{{\scriptstyle\swarrow}\kern-1.6ex^d\kern0.8ex}}
\newcommand{\fsubeq}{\mathrel{\raise-.7ex\hbox{$\overset{\searrow}{=}$}}}
\newcommand{\bsubeq}{\mathrel{\raise-.7ex\hbox{$\overset{\swarrow}{=}$}}}
\newcommand{\tw}{\operatorname{tw}}
\newcommand{\base}{\operatorname{base}}
\newcommand{\trans}{\operatorname{trans}}
\newcommand{\Int}{\operatorname{int}}
\newcommand{\rest}{|_}
\newcommand{\bbar}{\overline}
\newcommand{\UML}{\operatorname{\UU\MM\LL}}
\renewcommand{\d}{\operatorname{dia}}
\newcommand{\hs}{{\operatorname{hs}}}
\newcommand{\vT}{{V(T)}}
\newcommand{\EL}{\mathcal{EL}}
\newcommand{\tsum}{\sideset{}{'}\sum}
\newcommand{\tsh}[1]{\left\{\kern-.9ex\left\{#1\right\}\kern-.9ex\right\}}
\newcommand{\Tsh}[2]{\tsh{#2}_{#1}}
\newcommand{\qeq}{\mathrel{\approx}}
\newcommand{\Qeq}[1]{\mathrel{\approx_{#1}}}
\newcommand{\qle}{\lesssim}
\newcommand{\Qle}[1]{\mathrel{\lesssim_{#1}}}
\newcommand{\simp}{\operatorname{simp}}
\newcommand{\vsucc}{\operatorname{succ}}
\newcommand{\vpred}{\operatorname{pred}}
\newcommand\fhalf[1]{\overrightarrow {#1}}
\newcommand\bhalf[1]{\overleftarrow {#1}}
\newcommand\sleft{_{\text{left}}}
\newcommand\sright{_{\text{right}}}
\newcommand\sbtop{_{\text{top}}}
\newcommand\sbot{_{\text{bot}}}
\newcommand\sll{_{\mathbf l}}
\newcommand\srr{_{\mathbf r}}
\newcommand\geod{\operatorname{\mathbf g}}
\newcommand\mtorus[1]{\boundary U(#1)}
\newcommand\A{\mathbf A}
\newcommand\Aleft[1]{\A\sleft(#1)}
\newcommand\Aright[1]{\A\sright(#1)}
\newcommand\Atop[1]{\A\sbtop(#1)}
\newcommand\Abot[1]{\A\sbot(#1)}
\newcommand\boundvert{{\boundary_{||}}}
\newcommand\storus[1]{U(#1)}
\newcommand\Momega{\omega_M}
\newcommand\nomega{\omega_\nu}
\newcommand\twist{\operatorname{tw}}
\newcommand\modl{M_\nu}
\newcommand\MT{{\mathbb T}}
\newcommand\dw{{d_{weld}}}
\newcommand\dt{{d_{te}}}
\newcommand\Teich{{\operatorname{Teich}}}
\renewcommand{\Re}{\operatorname{Re}}
\renewcommand{\Im}{\operatorname{Im}}
\newcommand{\mc}{\mathcal}
\newcommand{\ccs}{{\CC(S)}}
\newcommand{\mtdw}{{({M_T},\dw)}}
\newcommand{\tmtdw}{{(\til{M_T},\dw)}}
\newcommand{\tmldw}{{(\til{M_l},\dw)}}
\newcommand{\mtdt}{{({M_T},\dt)}}
\newcommand{\tmtdt}{{(\til{M_T},\dt)}}
\newcommand{\tmldt}{{(\til{M_l},\dt)}}
\newcommand{\trvw}{{\tr_{vw}}}
\newcommand{\ttrvw}{{\til{\tr_{vw}}}}
\newcommand{\but}{{\BU(T)}}
\newcommand{\ilkv}{{i(lk(v))}}
\newcommand{\pslc}{{\mathrm{PSL}_2 (\mathbb{C})}}
\newcommand{\tttt}{{\til{\ttt}}}
\newcommand{\bcomment}[1]{\textcolor{magenta}{#1}}
\newcommand{\jfm}[1]{\marginpar{#1\quad -jfm}}
\newcommand{\mate}{\bot \!\! \! {\bot}} 

\makeatletter
\DeclareFontFamily{U}{tipa}{}
\DeclareFontShape{U}{tipa}{m}{n}{<->tipa10}{}
\newcommand{\arc@char}{{\usefont{U}{tipa}{m}{n}\symbol{62}}}%

\newcommand{\arc}[1]{\mathpalette\arc@arc{#1}}

\newcommand{\arc@arc}[2]{%
  \sbox0{$\m@th#1#2$}%
  \vbox{
    \hbox{\resizebox{\wd0}{\height}{\arc@char}}
    \nointerlineskip
    \box0
  }%
}
\makeatother

\makeatletter
\providecommand{\bigsqcap}{%
  \mathop{%
    \mathpalette\@updown\bigsqcup
  }%
}
\newcommand*{\@updown}[2]{%
  \rotatebox[origin=c]{180}{$\m@th#1#2$}%
}
\makeatother

\newcommand{\leftrarrows}{\mathrel{\raise.75ex\hbox{\oalign{%
  $\scriptstyle\leftarrow$\cr
  \vrule width0pt height.5ex$\hfil\scriptstyle\relbar$\cr}}}}
\newcommand{\lrightarrows}{\mathrel{\raise.75ex\hbox{\oalign{%
  $\scriptstyle\relbar$\hfil\cr
  $\scriptstyle\vrule width0pt height.5ex\smash\rightarrow$\cr}}}}
\newcommand{\Rrelbar}{\mathrel{\raise.75ex\hbox{\oalign{%
  $\scriptstyle\relbar$\cr
  \vrule width0pt height.5ex$\scriptstyle\relbar$}}}}
\newcommand{\longleftrightarrows}{\leftrarrows\joinrel\Rrelbar\joinrel\lrightarrows}

\makeatletter
\def\leftrightarrowsfill@{\arrowfill@\leftrarrows\Rrelbar\lrightarrows}
\newcommand{\xleftrightarrows}[2][]{\ext@arrow 3399\leftrightarrowsfill@{#1}{#2}}
\makeatother

\title[Combining Rational maps and Kleinian groups]{Combining rational maps and Kleinian groups via orbit equivalence}

\author{Mahan Mj}

\address{School of Mathematics, Tata Institute of Fundamental Research, Mumbai-400005, India}
\email{mahan@math.tifr.res.in, mahan.mj@gmail.com}

\author{Sabyasachi Mukherjee}

\address{School of Mathematics, Tata Institute of Fundamental Research, Mumbai-400005, India}
\email{sabya@math.tifr.res.in, mukherjee.sabya86@gmail.com}



\thanks{Both authors were  supported by  the Department of Atomic Energy, Government of India, under project no.12-R\&D-TFR-5.01-0500 as also  by an endowment of the Infosys Foundation.
MM was also supported in part by   a DST JC Bose Fellowship. SM was supported in part by SERB research project grant  SRG/2020/000018. } 
\date{\today}

\begin{abstract}  We develop a new orbit equivalence framework for holomorphically 
	mating the dynamics of complex polynomials with that of Kleinian surface groups. We show that the only torsion-free Fuchsian groups that can be thus mated are punctured sphere groups. We describe a new class of maps that
	are orbit equivalent to Fuchsian punctured sphere  groups. We call these \emph{higher Bowen-Series} maps. The existence of this class ensures  that the Teichm\"uller space of matings has one component corresponding to Bowen-Series maps and one corresponding to higher Bowen-Series maps. We also show that, unlike in higher dimensions, topological orbit equivalence rigidity fails for Fuchsian groups acting on the circle.
	We also classify the collection of Kleinian Bers boundary groups that are mateable in our framework.
\end{abstract}

\maketitle

\setcounter{tocdepth}{1}

\section{Introduction}\label{sec-intro}

Various connections and philosophical analogies exist between two branches of conformal dynamics; namely, rational dynamics on the Riemann sphere and actions of Kleinian groups. Fatou had already observed similarities between limit sets of Kleinian groups and Julia sets of rational maps in the 1920s, which led him to conjecture the following \cite[p. 22]{Fatou29}:
\begin{qn}\label{qn-fatou}
	L'analogie remarque{\'e} entre les ensembles de points limites des groupes Kleineens et ceux qui sont constitu{\'e}s par les fronti{\`e}res des r{\'e}gions de convergence des it{\'e}r{\'e}es d'une fonction rationnelle ne parait d'ailleurs pas fortuite et il serait probablement possible d'en faire la synt{\'e}se dans une th{\'e}orie g{\'e}n{\'e}rale des groupes discontinus des substitutions algr{\'e}briques.
\end{qn}

In the 1980s, Sullivan discovered deep connections between the iteration theory of rational functions and the theory of Kleinian groups. This became known as the \emph{Sullivan Dictionary} \cite[p. 405]{sullivan-dict}. Since then, several efforts to draw direct connections between these two branches of conformal dynamics have been made (see for example \cite{bullett-penrose, ctm-classification, lyubich-minsky, pilgrim}) and, in particular, lines to the dictionary have been added by McMullen \cite{ctm-renorm}.

An explicit framework for mating the modular group with rational maps has been developed by Bullett-Penrose-Lomonaco \cite{bullett-penrose,BL20,BL22}; it extends to Kleinian groups abstractly isomorphic to the modular group \cite{BH00,BH07}, and more generally to all Kleinian groups which are free products of two finite cyclic groups \cite{B00}. In the approach in these articles the matings are holomorphic correspondences, multi-valued algebraic functions defined by polynomial relations between two variables. In this paper we adopt a completely different approach to construct matings of Kleinian groups and rational maps that are single-valued analytic functions of a single complex variable.

The dynamical plane of a rational map $R$ admits a natural invariant partition: the \emph{Fatou set} $\mathcal{F}(R)$ (the largest open set on which the iterates of the map form a normal family) and the \emph{Julia set} $\mathcal{J}(R)$ (the complement of the Fatou set). On the other hand, the dynamical plane of a Kleinian group $\Gamma$ can be divided into two group-invariant subsets: the \emph{domain of discontinuity} $\Omega(\Gamma)$ (the largest open set on which the group acts properly discontinuously) and the \emph{limit set} $\Lambda(\Gamma)$ (the complement of the domain of discontinuity). 

We now reformulate Fatou's question as a  {\bf mating framework.}
\smallskip

\noindent {\emph{The group side of the dictionary-Bers slice closure groups:}}  Fix a Fuchsian group $\Gamma_0$ (of the first kind).  The \emph{Bers slice} $\mathcal{B}(\Gamma_0)$ is the space of quasi-Fuchsian simultaneous uniformizations 
of pairs $(\Gamma_0,\Gamma)$, where $\Gamma$ ranges over all Fuchsian groups isomorphic to $\Gamma_0$. (Here we require the isomorphisms to induce homeomorphisms between the quotient surfaces $\D/\Gamma$ and $\D/\Gamma_0$.) The induced
representations of (the abstract group underlying) $\Gamma_0$ into $\textrm{PSL}_2(\C)$ 
are, in fact, induced by quasiconformal homeomorphisms of $\widehat{\C}$ that are conformal on $\widehat{\C}\setminus\overline{\D}$. This gives $\mathcal{B}(\Gamma_0)$ a complex-analytic structure.
The Bers slice $\mathcal{B}(\Gamma_0)$ is thus a complex-analytic realization of the \emph{Teichm{\"u}ller space} $\textrm{Teich}(\Gamma_0)$. Recall that $\textrm{Teich}(\Gamma_0)$ is the space of marked hyperbolic
structures on the topological surface $\Sigma$ underlying $\D/\Gamma_0$. Equivalently, it is the space of equivalence classes of
discrete, faithful representations of $\pi_1(\D/\Gamma_0)$ into $\textrm{Aut}(\D)\cong\textrm{PSL}_2(\mathbb{R})$ such that boundary components of $\Sigma$ go to parabolics. Here, two representations are said to be equivalent, if they are conjugate in $\textrm{Aut}(\D)$.
For ease of exposition, we shall from now on identify 
$\pi_1(\D/\Gamma_0)$ with the base group $\Gamma_0$ and
$\Sigma$ with the resulting base hyperbolic surface.
An important property of the Bers slice is its pre-compactness in the discreteness locus of $\Gamma_0$, the space of all equivalence classes of  discrete, faithful representations into $\textrm{PSL}_2(\C)$ equipped with the topology of algebraic convergence. Here two representations are regarded as equivalent if they are conjugate in $\textrm{PSL}_2(\C)$.
The boundary 
$\partial \mathcal{B}(\Gamma_0)$ of the Bers slice $\mathcal{B}(\Gamma_0)$ in the discreteness locus is called the \emph{Bers boundary} of $\Gamma_0$.

	By the Bers density theorem \cite{minsky-elc2} due to Brock-Canary-Minsky, a finitely generated, non-elementary Kleinian group admits a simply connected invariant component in its domain of discontinuity if and only if it lies in the Bers slice closure of a Fuchsian lattice (in the classical literature, such groups were called $B$-groups, see \cite{maskit}).
	
	\smallskip

	\noindent  
{\emph{The rational dynamics side of the dictionary-Bers rational maps:}} 
 In the complex dynamics world, the analogous objects are rational maps with a simply connected completely invariant Fatou component. We shall henceforth call a rational map with the above property a {\bf Bers rational map}. It follows from the standard classification of Fatou components of rational maps that a completely invariant Fatou component of a Bers rational map is the basin of attraction of a (super)attracting or parabolic fixed point. 
	
	Prototypical examples of Bers rational maps are given by complex polynomials with connected Julia sets.
A polynomial is said to be \emph{hyperbolic} if each of its critical points converges to an attracting cycle under forward iteration. The set of all hyperbolic polynomials (of a given degree) is open in the parameter space. A connected component of degree $d$ hyperbolic polynomials is called a \emph{hyperbolic component} in the parameter space of degree $d$ polynomials. The hyperbolic component of degree $d$ polynomials containing the map $z^d$ is called the \emph{principal hyperbolic component}, and is denoted by $\mathcal{H}_d$. The Julia set of each map in $\mathcal{H}_d$ is a quasicircle, and the Julia dynamics of such a map is quasisymmetrically conjugate to the action of $z^d$ on $\bS^1$. Thus, principal hyperbolic components can be thought of as an analog of Bers slices in the polynomial dynamics world.

	In this paper, we set up a framework (See Sections \ref{mateable_gen_subsec}  and \ref{mat_def_subsec} for details) that allows one to potentially combine the actions of Kleinian groups in Bers slice closures with the dynamics of Bers rational maps, and focus on a reformulation of Fatou's Question \ref{qn-fatou} in the following  case.

\begin{qn}\label{problem_1} 
	Are there `dynamically natural' homeomorphisms between Julia sets of Bers rational maps and limit sets of Kleinian groups in the closure of a Bers slice? Further, can such a dynamically natural homeomorphism be used to complex analytically combine (or mate) the dynamics
	of a  Bers rational map with that of a Kleinian group along the lines of Douady
	and Hubbard \cite{douady-mating}? What is the parameter space of such matings?
\end{qn}

Theorem \ref{thm_mating_intro} below gives an affirmative answer to Question \ref{problem_1} for Bers
slice groups. Theorem \ref{thm_moduli_components_intro} concerns  the parameter space of such matings. Theorems \ref{thm_boundary_groups_1_intro} and \ref{thm_boundary_groups_2_intro} answer
Question \ref{problem_1} for groups on the Bers boundary in terms of existence of matings and their moduli.

	\emph{All Kleinian groups in this paper will be finitely generated non-elementary, unless mentioned otherwise.}
There are three natural  combination theorems that provide the background and the inspiration for Question \ref{problem_1}:\\
\noindent 1)  Bers Simultaneous Uniformization Theorem \cite{Bers60}\\
\noindent 2)  Thurston's Double Limit Theorem \cite{thurston-hypstr2}\\
\noindent 3)  Mating of complex polynomials \cite{douady-mating,shishikura-mating, shishikura-tan-mating,petersen-mayer-mating,hubbard-mating}.

The key point in Question \ref{problem_1} is the term `dynamically natural'. Indeed, the Julia set of any polynomial in the hyperbolic component containing $z^d$  is homeomorphic to the limit set of any group in a Bers slice (since they are both quasicircles). However, such a homeomorphism typically has no dynamical significance. The first obvious obstacle to formalizing the notion of a dynamically natural homeomorphism between a limit set and a Julia set is that the dynamics of a polynomial cannot be conjugated to that of a group. 

The first step to address the inherent mismatch between the two worlds of Kleinian groups and rational maps will be to associate a single map $A_\Gamma$ to a group $\Gamma$ on the Bers slice closure (of a given Fuchsian group $\Gamma_0$) such that $A_\Gamma$ is \emph{orbit equivalent} to the group $\Gamma$ on its limit set; i.e.,  $A_\Gamma$ is a continuous map such that the grand orbits of $A_\Gamma$ agree with the orbits of $\Gamma$ acting on $\Lambda(\Gamma)$. Orbit equivalence between $A_\Gamma$ and $\Gamma$ is possibly the weakest property that one can require $A_\Gamma$ to satisfy to address  meaningfully the mateability Question \ref{problem_1}. In fact, one needs to impose further regularity and compatibility conditions on $A_\Gamma$ for the purpose of mateability. These conditions are introduced in Sections~\ref{mat_fuch_subsec} and ~\ref{mateable_gen_subsec}, where such maps $A_\Gamma$ are termed \emph{mateable maps}. 
 The next step is to formulate a precise notion of mateability between a Bers rational map $R$ and a continuous self-map $A_\Gamma$ on the limit set of a Bers slice closure group $\Gamma$ (such that $\Gamma$ and $A_\Gamma$ are orbit equivalent) along the lines of Thurston's Double Limit Theorem and Douady-Hubbard theory of polynomial mating. This is done in Section \ref{mat_def_subsec}.

As an aside, we note that
	the action of $z^d$ on $\bS^1$ preserves the Lebesgue measure, and gives a measurable dynamical system of type $III_{\log d}$; whereas the action of a finite co-volume Fuchsian group on $\bS^1$ preserves the Lebesgue measure class, giving rise to a measurable dynamical system of type $III_1$ \cite{spatzier}. (We thank Caroline Series for explaining this to us.) Hence, topological,  not  measurable orbit equivalence, is the right framework for   Question \ref{problem_1}.

Most of the rest of the paper addresses Question \ref{problem_1} by exploring the question of existence and moduli of mateable maps. We first explicitly describe families of mateable examples and show that Bowen-Series maps of  punctured spheres are mateable (Section \ref{bs_sec}). We need to focus here on topological properties of Bowen-Series maps.
In Section \ref{sec-fold}, we describe a  new class of mateable maps that we term \emph{higher  Bowen-Series maps}. 
 Roughly, a higher Bowen-Series map of a Fuchsian group is obtained by `gluing together' several Bowen-Series maps of the same Fuchsian group with overlapping fundamental domains (see Proposition~\ref{hbs_alternative_prop} and Corollary~\ref{second_iterate_hbs_cor}).

Our first mateability theorem can now be  stated as follows (see Theorems~\ref{moduli_interior_mating_thm} and~\ref{thm-cfdmateable} for the precise formulations).

\begin{thmx}\label{thm_mating_intro}
	Bowen-Series maps and higher Bowen-Series maps of Fuchsian groups uniformizing punctured spheres (possibly with one/two orbifold points of order two) can be conformally mated with polynomials lying in principal hyperbolic components (of appropriate degree).
\end{thmx}

Existence of parabolic fixed points prevents us from using classical conformal welding techniques to construct the desired conformal matings. We use instead the theory of \emph{David homeomorphisms} \cite{David88}.
The examples in \cite{LLMM1, LMM2, LMMN} may  be viewed as anti-holomorphic precursors of the holomorphic mating construction of Theorem~\ref{thm_mating_intro}. We should, however, point out that there is no natural anti-holomorphic analog of higher Bowen-Series maps for  reflection groups.

After obtaining explicit examples of mateable maps, we turn to the problem of identifying their moduli. 
Surprisingly, it turns out that the only Fuchsian groups $\Gamma$ that are mateable correspond (in the torsion-free case) to spheres with punctures (see Proposition~\ref{prop-goah-dichotomy}). 
 It turns out that there are at least two distinct components, one corresponding to the Bowen-Series map
and one to the higher  Bowen-Series map (Theorem \ref{thm-moduli-eoe}):

\begin{thmx}\label{thm_moduli_components_intro}
	The moduli space of matings of the fundamental group
	of a punctured sphere $S_{0,k}$ ($k \geq 3$) with polynomials in $\mathcal{H}_d$ has at least two components corresponding to $d = 1-2 \chi(S_{0,k})$ and $d =  (1-\chi(S_{0,k}))^2$.
\end{thmx}

\begin{rmk}\label{rmk-clarintro}
Lest we mislead the reader, it is worth pointing out that the two different components Theorem \ref{thm_moduli_components_intro} refers to may equivalently be thought of as  two different moduli spaces:
\begin{enumerate}
\item the space of matings between the Bowen-Series map of a punctured sphere group (with a particular choice of a fundamental domain) and
polynomials in a particular $\mathcal{H}_d$, 
\item the space of matings between a higher Bowen-Series map associated to the same group, 
and polynomials in $\mathcal{H}_d$ for a different $d$.
\end{enumerate}
\end{rmk}

Sections \ref{bs_sec}, \ref{sec-fold}, \ref{sec-combchar} and \ref{sec-pattern} deal with groups in the  Bers slice. 
In Section \ref{limit_julia_homeo_sec}, we investigate which Kleinian groups on the boundary of a Bers slice $\mathcal{B}(\Gamma_0)$ can be mated with polynomials. 
The following theorem explicates topological obstructions to mating (see Lemmas~\ref{pwfm_invariant_lami_lem},~\ref{inv_lami_pwfm}, Proposition~\ref{b_s_inv_simp_closed_geod_lem}, and Theorem~\ref{first_return_conf_model_thm}).

\begin{thmx}\label{thm_boundary_groups_1_intro} 
	Let $\Gamma_0$ be a punctured sphere Fuchsian group. Then, there are only finitely many quasiconformal conjugacy classes of groups $\Gamma\in\partial\mathcal{B}(\Gamma_0)$ for which the Cannon-Thurston map of $\Gamma$ semi-conjugates the Bowen-Series map of $\Gamma_0$ to a self-map of $\Lambda(\Gamma)$ that is orbit equivalent to $\Gamma$ (we call this map the \emph{Bowen-Series map} of $\Gamma$). These Kleinian groups arise out of pinching finitely many disjoint, simple, closed curves (on the surface $\D/\Gamma_0$) out of an explicit finite list. In particular, all such groups $\Gamma$ are geometrically finite. Moreover, there exist groups $\Gamma$ of the above type for which the first return map of the Bowen-Series map to certain circles contained in $\Lambda(\Gamma)$ is a  higher Bowen-Series map.
	
	On the contrary, simply degenerate Kleinian groups \emph{do not} admit such orbit equivalent maps. 
\end{thmx}

Once we have a collection of Bers boundary groups equipped with Bowen-Series maps (that are orbit equivalent to the corresponding groups) at our disposal, we proceed to answer Question~\ref{problem_1} affirmatively for such groups.

\begin{thmx}\label{thm_boundary_groups_2_intro}
	Let $\Gamma_0$ be a Fuchsian group uniformizing $S_{0,k}$, and $\Gamma\in\partial\mathcal{B}(\Gamma_0)$ be a group  that admits a Bowen-Series map. Then the following hold.
	\begin{enumerate}
		\item There exists a complex polynomial $P_\Gamma$ (of degree $1-2 \chi(S_{0,k})=2k-3$) such that the action of the Bowen-Series map of $\Gamma$ on its limit set is topologically conjugate to the action of $P_\Gamma$ on its Julia set.
		
		\item The canonical extension of the Bowen-Series map of $\Gamma$ can be conformally mated with polynomials lying in the principal hyperbolic component $\mathcal{H}_{2k-3}$.
	\end{enumerate}
\end{thmx}

We refer the reader to Theorems~\ref{julia_limit_dyn_equiv_thm} and~\ref{bers_bdry_mating_thm} for precise statements.

\begin{rmk}
We point out that while Theorems~\ref{thm_mating_intro} and~\ref{thm_boundary_groups_2_intro} guarantee the existence of conformal matings, they do not characterize the class of holomorphic maps arising in this process. An explicit description of conformal matings between Bowen-Series maps of punctured sphere groups and polynomials in principal hyperbolic components would be useful for a finer algebraic/analytic study of their parameter spaces. This will be taken up in a subsequent paper.
\end{rmk}

Section \ref{sec-coe} provides an application to the problem of topological orbit equivalence rigidity (see \cite{fisherwhyte_gafa}
for instance for very general positive results).
As a fallout of the methods of this paper, we obtain the unexpected conclusion that
topological orbit equivalence rigidity  fails for Fuchsian groups (Theorem \ref{thm-goe-fuchsian}).

\begin{thmx}\label{thm_goe_fuchsian_intro}
For	any two punctured sphere Fuchsian groups $\Gamma_1$, $\Gamma_2$, the actions of $\Gamma_1, \Gamma_2$ on $\bS^1$  are
		topologically orbit equivalent. 
\end{thmx}

\noindent {\bf Acknowledgments:} 
The authors thank Kingshook Biswas, David Fisher, Etienne Ghys, Sam Kim, Yusheng Luo, and Katie Mann for helpful correspondence. The stronger version of Lemma \ref{lemma-referee} in the present paper as opposed to the corresponding one in an earlier draft was kindly supplied to us by an anonymous referee. We also thank the referees for other helpful comments.
Special thanks are due to Caroline Series for explaining some important aspects of orbit equivalence to us. SM gratefully acknowledges the collaboration with Seung-Yeop Lee, Mikhail Lyubich, and Nikolai Makarov on matings of the ideal triangle reflection group with anti-holomorphic polynomials in \cite{LLMM1}. 
This led us to the idea that the Bowen-Series map could be put to use to mate rational maps and Kleinian groups.

There are two  major intellectual debts that we ought to mention at the outset. The first is to the paper \cite{bowen-series} of Bowen and Series that provides the starting point of a map orbit equivalent to a group, though they work in the framework of measurable orbit equivalence as opposed to the  framework of topological orbit equivalence that we use in this paper. 
The other is the work of Bullett-Penrose-Lomonaco \cite{bullett-penrose,BL20} that provides a correspondence-theoretic framework for the mating of the modular group with quadratic polynomials in terms of algebraic correspondences.

\section{Mateable maps and a mateability framework}\label{sec-mat}
We denote the group of all conformal automorphisms of the unit disk $\disk$ by $\textrm{Aut}({\disk})$. 
The aim of this section is to associate \emph{mateable} maps $A_\Gamma:\Lambda(\Gamma)\rightarrow\Lambda(\Gamma)$ to groups $\Gamma$ in the Bers slice closure of a Fuchsian group $\Gamma_0$, and to lay out precisely what it means in this paper to say that $A_\Gamma$ is topologically/conformally mateable with a Bers rational map $R$.

\subsection{Mateable maps for Fuchsian groups}\label{mat_fuch_subsec}

\subsubsection{Orbit equivalence}

A basic problem that arises in trying to make sense of what it means to mate a polynomial $P$
with
a Fuchsian group $\Gamma$ is purely algebraic in nature. On one side of the picture we have the semigroup
$\langle P \rangle$
generated by $P$, while
on the other side we have a non-commutative group
$\Gamma$ generated by more than one element. Thus, we need
to replace $\Gamma$ by a single map $A$ that captures at least the topological dynamics of $\Gamma$. This leads us to the notion of orbit equivalence.

\begin{defn}\label{defn-goe-A}
	Let $A:\bS^1\to \bS^1$. The \emph{grand orbit} of a point $x\in\bS^1$ under $A$ is defined as 
	$
	\mathrm{GO}_A(x):=\{ x'\in\bS^1: A^{\circ m}(x)=A^{\circ n}(x'),\ \textrm{for\ some}\ m, n\geq 0\}.
	$
	We say that two continuous maps
	$A_i:\mathbb{S}^1\to\mathbb{S}^1$ ($i=1,2$) are \emph{topologically orbit equivalent}  if there exists a homeomorphism $\phi: \bS^1 \to \bS^1$ such that for every $x\in \bS^1$, $\phi(\mathrm{GO}_{A_1}(x)) = \mathrm{GO}_{A_2}(\phi(x))$. The homeomorphism $\phi$ is called a \emph{topological orbit equivalence} between $A_1$ and $A_2$.
\end{defn}

\begin{defn}\label{orbit_equiv_def}
	Let $\Gamma$ be a Fuchsian group with limit set equal to $\Lambda \subset \mathbb{S}^1$. We say that a continuous map
	$A:\mathbb{S}^1\to\mathbb{S}^1$ is \emph{orbit equivalent} to $\Gamma$ on the limit set if for every $x\in\Lambda$, the $\Gamma$-orbit of $x$ is equal to the grand orbit of $x$ under $A$; i.e., $\Gamma\cdot x= \mathrm{GO}_A(x)$.
\end{defn}

The notion of orbit equivalence turns out to be  flexible enough for our purposes, as illustrated  by Theorem \ref{thm_goe_fuchsian_intro} and Lemma \ref{lemma-referee}.

\subsubsection{Regularity}

The classical theorem on mating purely in the context of Fuchsian and Kleinian groups is the Bers' simultaneous uniformization theorem. Roughly speaking, this theorem asserts that two Fuchsian groups, acting on two copies of $\disk$, can be conformally combined to obtain a quasi-Fuchsian group (acting on the Riemann sphere $\widehat{\C}$). 

Let $A:\bS^1 \to \bS^1$ be a continuous map orbit equivalent to a Fuchsian group $\Gamma$. In the spirit of the Bers' simultaneous uniformization theorem, to conformally mate $A$ with complex polynomials, we would like to augment $A$ to a `conformal dynamical system'. More precisely, we want to extend $A$ to a complex analytic map $\widehat{A}$ defined on a subset of $\disk$. 

Now suppose that there exists a complex-analytic map $\widehat{A}$ defined on the open set $\{z: r<\vert z\vert<1\}$ (for some $r\in\left(0,1\right)$) that continuously extends $A:\bS^1\to\bS^1$.  Then by the Schwarz reflection principle,  $\widehat{A}$ can be extended to a complex analytic map on an annular neighborhood of $\bS^1$. Orbit equivalence between $A$ and $\Gamma$ combined with the identity principle for complex analytic maps now implies that $A$ must equal a single M{\"o}bius map in $\Gamma$ (compare the proof of Lemma~\ref{pwa_is_pwm_lem}). But this would force $\Gamma$ to be generated by a single M{\"o}bius map, making it evident that such an extension $\widehat{A}$ is too much to ask for. To tackle this problem, we allow $A$ to be a continuous piecewise real analytic map of $\mathbb{S}^1$, so that each piece of $A$ admits a complex analytic extension to a neighborhood of its domain of definition.

\begin{defn}\label{pwm_def}
	We say that a map $A:\mathbb{S}^1\to\mathbb{S}^1$ is \emph{piecewise M{\"o}bius} if there exist $k\in\mathbb{N}$, closed arcs $I_j\subset\mathbb{S}^1$, and $g_j\in\textrm{Aut}({\disk})$, $j\in\{1,\cdots, k\}$,  such that
	\begin{enumerate}
		\item $\displaystyle\mathbb{S}^1=\bigcup_{j=1}^k I_j,$
		
		\item $\Int{I_m}\cap\Int{I_n}=\emptyset$ for $m\neq n$, and
		
		\item $A\vert_{I_j}=g_j$.
	\end{enumerate}
	
	\noindent A piecewise M{\"o}bius map is called \emph{piecewise Fuchsian} if $g_1,\cdots, g_k$ generate a Fuchsian group, which we denote by $\Gamma_A$.
	
	\noindent	 In the above definition, if the maps $g_j$ are assumed only to be complex-analytic in some small neighborhoods of $\Int{I_j}$ (without requiring them to be M{\"o}bius), then $f$ is said to be \emph{piecewise analytic}. 
\end{defn}

Formally speaking, a piecewise M{\"o}bius/analytic map $A$ is a pair $\left(\{g_j\}_{j=1}^k, \{I_j\}_{j=1}^k\right)$; i.e., the partition of $\bS^1$ into the closed arcs $\{I_j\}$ is a part of the definition of $A$.
The maps $g_j$ will be called the \emph{pieces} of $A$. We shall occasionally refer to the domains $I_j$ of $g_j$ also as  \emph{pieces} of $A$ when there is no scope for confusion.

\begin{lemma}\label{pwa_is_pwm_lem}
	Let $A:\mathbb{S}^1\to\mathbb{S}^1$ be a piecewise analytic map that is orbit equivalent to a finitely generated Fuchsian group $\Gamma$. Then, $A$ is piecewise Fuchsian, and the pieces of $A$ form a generating set for $\Gamma$.
\end{lemma}
\begin{proof}
	Let $k\in\mathbb{N}$, $I_j\subset\mathbb{S}^1$, $j\in\{1,\cdots, k\}$, be as in Definition~\ref{pwm_def}. 
	
	Fix $j\in\{1,\cdots, k\}$. Since $A$ is orbit equivalent to $\Gamma$, for each $x\in I_j$, there exists some $g\in\Gamma$ such that $A(x)=g(x)$. It follows from countability of $\Gamma$ and uncountability of $I_j$ that $A$ agrees with some $g_j\in\Gamma$ on an uncountable subset of $I_j$. The identity theorem for holomorphic functions now implies that $A\equiv g_j$ on $I_j$. This proves that $A$ is piecewise M{\"o}bius. Moreover, the M{\"o}bius maps that define $A$ belong to the Fuchsian group $\Gamma$, so $A$ is piecewise Fuchsian.
	
	Clearly, the grand orbit of any $x\in\mathbb{S}^1$ under $A$ is contained in the $\Gamma'$-orbit of $x$, where $\Gamma'$ is the subgroup of $\Gamma$ generated by $g_1,\cdots, g_k$. Therefore, orbit equivalence of $A$ and $\Gamma$ implies that $\Gamma\cdot x=\Gamma'\cdot x$, for each $x\in\mathbb{S}^1$. Hence, $\Gamma'=\Gamma$ (for instance, by choosing $x\in\mathbb{S}^1$ with trivial stabilizer in both $\Gamma$ and $\Gamma'$).
\end{proof}

Suppose that $x_1, \cdots, x_k$ are a
cyclically ordered collection of $k$ points on $\bS^1$ defining the pieces $I_j=[x_j,x_{j+1}]$ of $A$ ($j+1$ taken modulo $k$).
\begin{defn}\label{def-minimal}
	We shall say that $A$ is minimal, if the decomposition of $\bS^1$ given by $x_1, \cdots, x_k$ is minimal;
	i.e., there does not exist $i$ and $h \in \Gamma_A$ such that
	\begin{enumerate}
		\item $A|_{[x_i,x_{i+1}]} = h|_{[x_i,x_{i+1}]}$, and
		\item $A|_{[x_{i-1},x_{i}]} = h|_{[x_{i-1},x_{i}]}$.
	\end{enumerate}
	(Thus, there are no superfluous break-points in a minimal $A$.)
\end{defn}

We now define a canonical extension of a piecewise M{\"o}bius map. 

\begin{defn}\label{def-canonicalextension}
	Let $A$ be a continuous 
	piecewise M\"obius map on the circle. Let $\D$ denote the unit disk.
	Let $I_1, \cdots, I_k$
	be a circularly ordered family of intervals with disjoint interiors such that
	\begin{enumerate}
		\item $I_j \cap I_{j+1} = \{x_{j+1}\}$ (the indices being taken mod $k$).
		\item $A |_{I_j} = g_j$.
	\end{enumerate}
	Let $\gamma_j$ be the semi-circular arc in $\D$ between $x_{j}, x_{j+1}$ meeting $\bS^1$ at right angles at $x_{j}, x_{j+1}$, and let $\mathcal{D}_j \subset \bbar{\D}$ be the closed region bounded by $I_j$ and $\gamma_j$. Then 
	$\widehat{A}$, the \emph{canonical extension of $A$ in $\overline{\D}$} 
	is defined on $\cup_j \mathcal{D}_j$ as $\widehat A = g_j$ on $\mathcal{D}_j$.
	
	Next, denote the (full) Euclidean circle containing  $\gamma_j$ by $C_j$, and the open round disk bounded by $C_j$ by $D_j$.
	Then $\widehat{A}$, as defined above, admits a natural extension 
	$$
	\widehat{A}:\bigcup_{j=1}^k \overline{D_j}\to\widehat{\mathbb{C}},\ z\mapsto g_j(z),\ \textrm{if}\ z\in\overline{D_j}.
	$$
	We shall refer to this further extension  as the  \emph{canonical extension of $A$ in $\widehat{\C}$}.
\end{defn}
It will be clear from the context whether we are taking the 
canonical extension of $A$ in $\overline{\D}$ or $\widehat{\C}$, and we shall often omit
mentioning this explicitly.

\begin{defn}\label{def-canonicalextension-domain}
	We let $\mathcal{D}=\cup_j \mathcal{D}_j$ and call $\mathcal{D}$ the \emph{canonical domain of definition} of $\widehat A$ (in $\overline{\D}$).
	Let $R = \D \setminus \mathcal{D}$. We refer to $R$ as the 
	\emph{fundamental domain} of $A$, and also as the fundamental domain of $\widehat{A}$.
	Each bi-infinite geodesic contained in the boundary $\partial R$ will be called an \emph{edge} of $R$.
	The ideal endpoints of $R$ will be called the \emph{vertices} of $R$. The set of vertices of $R$ will be denoted by $S$.
	A pair of adjacent vertices  in $S$ (joined by an edge) will also occasionally be referred to as
	an  \emph{edge} of $R$.
	We shall refer to pairs of non-adjacent points in $S$ (the ideal endpoints of $R$), or equivalently the bi-infinite geodesic joining them in $R$ as a \emph{diagonal} of $R$.
\end{defn}
Note that $\widehat A : \mathcal{D} \to \overline{\D}$.

\subsubsection{Expansivity and degree}

Consider the (simplest) degree $d$ polynomial $z^d$ ($d>1$). It preserves the unit disk $\disk$ and the unit circle $\bS^1$ (which is the \emph{Julia set} of $z^d$). We are interested in conformal matings between $z^d$ and $\widehat{A}$ (where $A:\bS^1\to\bS^1$ is a piecewise Fuchsian map). The first step in this direction is to `topologically glue' the two dynamical systems $z^d:\overline{\disk}\to\overline{\disk}$ and $\widehat{A}:\Omega\to\overline{\disk}$ along the unit circle to obtain a topological dynamical system defined on a subset of $\widehat{\C}$. Clearly, the `welding homeomorphism' used to identify the two copies of $\bS^1$ must respect the dynamics of $A$ and $z^d$; i.e., it must conjugate $z^d\vert_{\bS^1}$ to $A$ (thus, in particular, $A$ must be a covering of $\bS^1$ of degree at least two). According to \cite{coven-reddy}, this is equivalent to saying that $A$ is an expansive covering of $\mathbb{S}^1$ of degree at least two (more generally, two homotopic expansive maps of compact manifolds are topologically conjugate).

\begin{defn}\label{expansive_def}
	A continuous map $f\colon \mathbb{S}^1\to \mathbb{S}^1$ is called \emph{expansive} if there exists a constant $\delta>0$ such that for any $a,b\in \mathbb{S}^1$ with $a\neq b$ we have $d(f^{\circ n}(a), f^{\circ n}(b))>\delta$ for some $n\in \mathbb{N}$.
\end{defn}

\begin{defn}\label{para_hyp_def}
	Let $A$ be a piecewise M{\"o}bius, expansive circle covering having $x_1,\cdots, x_k$ as the break-points of its piecewise definition. Further, let $x_j$ be a periodic point of period $n$ under $A$. We say that $x_j$ is \textit{parabolic on the right} (resp., \textit{on the left}) if $(A^{\circ n})'(x_j^+)=1$ (resp., $(A^{\circ n})'(x_j^-)=1$). Likewise, $x_j$ is \textit{hyperbolic on the right} (respectively, \textit{on the left}) if $(A^{\circ n})'(x_j^+)>1$ (resp., $(A^{\circ n})'(x_j^-)>1$). 
	
	We say that $x_j$ is \textit{symmetrically parabolic} (respectively, \textit{symmetrically hyperbolic}) if $(A^{\circ n})'(x_j^+)=(A^{\circ n})'(x_j^-)=1$ (respectively, if $(A^{\circ n})'(x_j^+)=(A^{\circ n})'(x_j^-)>1$). 
	
	$x_j$ is called \textit{asymmetrically hyperbolic} if it is hyperbolic on both sides, but $(A^{\circ n})'(x_j^+)\neq (A^{\circ n})'(x_j^-)$.
	
	Finally, $x_j$ is said to be a \textit{periodic point of mixed type} if it is hyperbolic on one side, but parabolic on the other.
\end{defn}

\begin{lemma}\label{no_mixed}
	Let $A:\mathbb{S}^1\to\mathbb{S}^1$ be a piecewise Fuchsian expansive covering map having $x_1,\cdots, x_k$ as the break-points of its piecewise definition. Further, let $x_j$ be a periodic point of $A$. Then, $x_j$ is not of mixed type.
\end{lemma}
\begin{proof}
	That $x_j$ is periodic under $A$ implies that there exist elements $h_1, h_2$ in the Fuchsian group $\Gamma_A$ (generated by the M{\"o}bius maps that define $A$) such that $h_1(x_j)=h_2(x_j)=x_j$. But $x_j$ cannot be fixed by a hyperbolic element and a parabolic element both of which lie in a Fuchsian group. The result follows.
\end{proof}

\subsubsection{Markov property}

The polynomial
map $z \mapsto z^d$, restricted to the unit circle $\mathbb{S}^1$, admits a Markov partition. Since $A$ is required to be topologically 
conjugate to $z \mapsto z^d$ (for some $d\geq 2$), $A: \bS^1 \to \bS^1$ must also admit
a Markov partition. There are two issues we now need to address:
\begin{enumerate}
	\item Non-uniqueness of the Markov partition for $z \mapsto z^d$.
	\item Potential incompatibility between the pieces (intervals of definition) of $A$ and the Markov partition.
\end{enumerate}
The following definition addresses these issues by declaring that the Markov partition for $A$ is compatible with
the pieces of $A$:

\begin{defn}\label{def-pwfm}
	We call $A: \bS^1 \to \bS^1$ a \emph{\pwfm} map if it is a piecewise Fuchsian expansive covering map (of degree at least two) such that the pieces (intervals of definition) of $A$ in $\bS^1$ give a Markov partition for $A: \bS^1 \to \bS^1$.
\end{defn} 

\subsubsection{Mateable maps}

\begin{defn}\label{def-mateable}
	A \pwfm map $A: \bS^1 \to \bS^1$ is said to be \emph{mateable} if
	
	\begin{enumerate}
		\item\label{cond_orbit_equiv} $A$ is orbit equivalent to the Fuchsian group $\Gamma_A$ generated by its pieces, and
		
		\item\label{cond-asymhyp} none of the periodic break-points of $A$ is asymmetrically hyperbolic.
	\end{enumerate} 
\end{defn}

\begin{lemma}\label{lem-gammalattice}
	If $A$ is mateable, then $\Gamma_A$ is a lattice (or equivalently, a finitely generated Fuchsian group of the first kind).
\end{lemma}

\begin{proof}
	Note that the maps $A$ and $z^d$ are topologically conjugate on $\bS^1$. Clearly, the topological conjugacy carries grand orbits of $A$ to those of $z^d$. Hence, all grand orbits of $A$
	are dense in $\bS^1$. By Condition \ref{cond_orbit_equiv} of Definition \ref{def-mateable}, $\Gamma_A$-orbits are dense in $\bS^1$. Hence the limit set of $\Gamma_A$ is $\bS^1$; i.e., $\Gamma_A$ is a Fuchsian group of the first kind. Since $\Gamma_A$ is generated by the pieces of $A$, the Lemma follows.
\end{proof}

\subsection{Mateable maps for Bers slice closure groups}\label{mateable_gen_subsec}

We  recall the notion of Cannon-Thurston maps \cite{CTpub}.

\begin{theorem}\label{ctsurf}\cite{mahan-split,mahan-elct,mahan-red,mahan-kl}
	Let $\rho (\pi_1({\Sigma})) =\Gamma \subset   \pslc$ be a  Kleinian surface  group (possibly with accidental parabolics). Let $M=\Hyp^3/\Gamma$ and $i: {\Sigma} \to M$ be an embedding inducing a homotopy equivalence. Let $\til{i} : \Hyp^2 \to \Hyp^3$ denote a lift of $i$ between universal covers of ${\Sigma},M$. Let $\bbar{\mathbb{H}^2}, \bbar{\mathbb{H}^3}$ denote the compactifications of $\Hyp^2, \Hyp^3$
	respectively. Then a continuous extension $\hhat{i}: \bbar{\mathbb{H}^2} \to \bbar{\mathbb{H}^3}$ exists.

	Let $\partial i: \bS^1 \to \bS^2$ denote the restriction of $\hhat{i}$ to the ideal boundaries. For $\Gamma$ a Bers boundary group, $p\neq q \in \bS^1$, $\partial i(p) = \partial i(q)$ if and only if $p, q$ are the ideal endpoints of a leaf of an ending lamination or ideal endpoints of a complementary ideal polygon of an ending lamination.
\end{theorem}

A continuous extension $\hhat{i}$ as in Theorem~\ref{ctsurf} is called a \emph{Cannon-Thurston map}.

For the remainder of this subsection, let us fix a finitely generated Fuchsian group $\Gamma_0$ of the first kind and a mateable map $A_{\Gamma_0}:\bS^1\to\bS^1$ whose pieces generate $\Gamma_0$ (thus by definition, $A_{\Gamma_0}$ is orbit equivalent to $\Gamma_0$). We denote the canonical extension of $A_{\Gamma_0}$ by $\widehat{A}_{\Gamma_0}:\mathcal{D}_{\Gamma_0}\to\overline{\D}$ where $\mathcal{D}_{\Gamma_0}$ is the canonical domain of definition of $\widehat{A}_{\Gamma_0}$ in $\overline{\D}$ (see Definition~\ref{def-canonicalextension}).

For any $\Gamma\in\overline{\mathcal{B}(\Gamma_0)}$, we set $K(\Gamma):=\widehat{\C}\setminus\Omega_\infty(\Gamma)$ where $\Omega_\infty(\Gamma)$ is the simply connected invariant component of $\Omega(\Gamma)$ on which the $\Gamma$-action is conformally conjugate to the $\Gamma_0$-action on $\widehat{\C}\setminus\overline{\D}$ (cf. \cite[\S 8]{bers-boundary}). 

Let $\Gamma$ be a Kleinian group in the Bers slice $\mathcal{B}(\Gamma_0)$. Then the representation $\Gamma_0\to\Gamma$ is induced by an equivariant quasiconformal homeomorphism $\phi_\Gamma$ (which is unique up to M{\"o}bius maps); more precisely, the group isomorphism is given by $g\mapsto \phi_\Gamma\circ g \circ\phi_\Gamma^{-1}$. Note that $K(\Gamma)=\phi_\Gamma(\overline{\D})$. We further set $\mathcal{D}_\Gamma:=\phi_\Gamma(\mathcal{D}_{\Gamma_0})$. We call the map
$$
A_\Gamma:=\phi_\Gamma\circ A_{\Gamma_0}\circ\phi_\Gamma^{-1}:\Lambda(\Gamma)\to\Lambda(\Gamma)
$$
the \emph{mateable map associated to $\Gamma\in \mathcal{B}(\Gamma_0)$ compatible with $A_{\Gamma_0}$}, and define its \emph{canonical extension} to be
$$
\widehat{A}_\Gamma:=\phi_\Gamma\circ \widehat{A}_{\Gamma_0}\circ\phi_\Gamma^{-1}:\mathcal{D}_\Gamma\to K(\Gamma).
$$ 
Thus, a mateable map orbit equivalent to the Fuchsian group $\Gamma_0$ automatically determines compatible mateable maps for all groups in the Bers slice of $\Gamma_0$. By construction, such an $A_\Gamma$ is a piecewise M{\"o}bius, Markov covering map of $\Lambda(\Gamma)$ that is orbit equivalent to $\Gamma$.

We now turn our attention to groups on the boundary of $\mathcal{B}(\Gamma_0)$. Let $\Gamma\in\partial\mathcal{B}(\Gamma_0)$ with associated group isomorphism $\rho_\Gamma:\Gamma_0\rightarrow\Gamma$. We say that $\Gamma$ admits a \emph{mateable map compatible with $A_{\Gamma_0}$} if the Cannon-Thurston map $\phi_\Gamma:\Lambda(\Gamma_0)=\bS^1\rightarrow \Lambda(\Gamma)$ semi-conjugates $A_{\Gamma_0}$ to a continuous self-map of $\Lambda(\Gamma)$, and denote this self-map (if it exists) by $A_\Gamma$. By the equivariance property of Cannon-Thurston maps, such an $A_\Gamma$ is a piecewise M{\"o}bius map of $\Lambda(\Gamma)$ that is orbit equivalent to $\Gamma$ (see Lemma~\ref{pwfm_invariant_lami_lem}). As $A_\Gamma$ is piecewise M{\"o}bius, we can extend it complex-analytically to a set $\mathcal{D}_\Gamma$ such that $\Lambda(\Gamma)\subset\mathcal{D}_\Gamma\subset K(\Gamma)$, and $\mathcal{D}_\Gamma$ contains relatively open neighborhoods of all the pieces of $A_\Gamma$ possibly after a further finite subdivision (cf. Subsection~\ref{bs_extension_bdry_subsec}).
We call this extended map $\widehat{A}_\Gamma$, and note that the domain of definition of this extension is not canonical.

\subsection{A framework for topological and conformal mating}\label{mat_def_subsec}

Let $k:= \deg{A_{\Gamma_0}:\bS^1\to\bS^1}$.
Consider the pair of dynamical systems:
\begin{enumerate}
	\item a mateable map $\widehat{A}_\Gamma:\mathcal{D}_\Gamma\to K(\Gamma)$ associated to $\Gamma\in\overline{\mathcal{B}(\Gamma_0)}$ compatible with $A_{\Gamma_0}$, and
	
	\item a Bers rational map $R$ of degree $k$ with locally connected Julia set.
\end{enumerate}

We set $\mathcal{K}(R):=\widehat{\C}\setminus\mathcal{B}(R)$ where $\mathcal{B}(R)$ is a marked simply connected completely invariant Fatou component, and call it the \emph{filled Julia set} of $R$.

We now state what it means to topologically (respectively, conformally) combine the above two dynamical systems. Our definition conforms to the Douady-Hubbard convention of polynomial mating.

If $\phi_R:\D\rightarrow\mathcal{B}(R)$ is a Riemann uniformization, then $\phi_R^{-1}\circ R\circ\phi_R:\D\to\D$ is a degree $k$ Blaschke product $B$. As $\mathcal{J}(R)$ is locally connected, $\phi_R$ extends continuously to $\bS^1$ to yield a semi-conjugacy between $B$ and $R$. Moreover, $B$ has a (super-)attracting or parabolic fixed point in $\overline{\D}$, the Julia set of $B$ is $\bS^1$, and $B\vert_{\bS^1}$ is an expansive covering of degree $k$. Thus, there exists a homeomorphism $\eta:\bS^1\to\bS^1$ that conjugates $B$ to $A_{\Gamma_0}$.

We first define the topological mating of the above pair of dynamical systems.
We consider the disjoint union $\mathcal{K}(R)\sqcup K(\Gamma)$ and the map 
\begin{center}
	$R\sqcup A_{\Gamma}: \mathcal{K}(R)\sqcup \mathcal{D}_\Gamma\to \mathcal{K}(R)\sqcup K(\Gamma),$\\
	$R\sqcup A_{\Gamma}\vert_{\mathcal{K}(R)}=R,\quad  R\sqcup A_\Gamma\vert_{\mathcal{D}_\Gamma}=\widehat{A}_\Gamma.$
\end{center}
Let $\sim$ be the equivalence relation on $\mathcal{K}(R)\sqcup K(\Gamma)$ generated by 
\begin{equation}
\phi_{R}(z)\sim \phi_{\Gamma}(\eta(\overline{z})),\ \textrm{for all}\ z\in\mathbb{S}^1.
\label{conf_mat_equiv_rel}
\end{equation} 
It is easy to check that $\sim$ is $R\sqcup A_\Gamma$-invariant, and hence it descends to a continuous map $R\mate A_\Gamma$ to the quotient $\mathcal{K}(R)\mate K(\Gamma):=\left(\mathcal{K}(R)\sqcup K(\Gamma)\right)/\sim$ (see \cite[\S 4.1]{petersen-mayer-mating} for details of this construction in the polynomial mating context). We call the map $R\mate A_\Gamma$ the \emph{topological mating} of the Bers rational map $R$ and the mateable map $A_\Gamma$.
Moreover, if $\mathcal{K}(R)\mate K(\Gamma)$ is homeomorphic to a $2$-sphere, we say that the topological mating is \emph{Moore-unobstructed}. (We refer the reader to \cite[Theorem~2.12]{petersen-mayer-mating} for the statement of Moore's theorem, which provides a general sufficient condition for the quotient of $\mathbb{S}^2$ under an equivalence relation to be a topological $2$-sphere.) We say that $R$ and $A_\Gamma$ are \emph{conformally mateable} if their topological mating is Moore-unobstructed, and if the topological $2$-sphere $\mathcal{K}(R)\mate K(\Gamma)$ admits a complex structure that turns the topological mating $R\mate A_\Gamma$ into a holomorphic map.

The following equivalent definition of conformal mating of $R$ and $A_\Gamma$ is often more useful in practice (cf. \cite[\S 4.7]{petersen-mayer-mating}).

\begin{defn}\label{conf_mating_def}
	A continuous map $F\colon \mathrm{Dom}(F)\subsetneq\widehat{\C}\to\widehat{\C}$
	is called a \emph{conformal mating} of $A_{\Gamma}$ and $R$ if $F$ is complex-analytic in the interior of $\mathrm{Dom}(F)$, and there exist continuous maps 
	$$
	\mathfrak{X}_R:\mathcal{K}(R)\to\widehat{\C}\ \textrm{and}\ \mathfrak{X}_\Gamma: K(\Gamma)\to\widehat{\C}
	$$
	conformal on $\Int{\mathcal{K}(R)}, \Int{K(\Gamma)}$ (respectively), satisfying
	\begin{enumerate}
		\item\label{topo_cond} $\mathfrak{X}_R\left(\mathcal{K}(R)\right)\cup \mathfrak{X}_\Gamma\left(K(\Gamma)\right) = \widehat{\C}$,
		
		\item\label{dom_cond} $\mathrm{Dom}(F)\supset \mathfrak{X}_R(\mathcal{K}(R))\cup\mathfrak{X}_\Gamma(\mathcal{D}_\Gamma)$,
		
		\item $\mathfrak{X}_R\circ R(w) = F\circ \mathfrak{X}_R(w),\quad \mathrm{for}\ w\in\mathcal{K}(R)$,
		
		\item $\mathfrak{X}_\Gamma\circ \widehat{A}_{\Gamma}(w) = F\circ \mathfrak{X}_\Gamma(w),\quad \mathrm{for}\ w\in
		\mathcal{D}_\Gamma$, and

		\item\label{identifications} $\mathfrak{X}_R(z)=\mathfrak{X}_\Gamma(w)$ if and only if $z\sim w$ where $\sim$ is the equivalence relation on $\mathcal{K}(R)\sqcup K(\Gamma)$ defined in \eqref{conf_mat_equiv_rel}.
	\end{enumerate}
\end{defn}

\subsection{David homeomorphism and conformal mating}\label{sec-mateable}

\begin{defn}\label{def-david}
	An orientation-preserving homeomorphism $H: U\to V$ between domains in the Riemann sphere $\widehat{\C}$ is called a \textit{David homeomorphism} if it lies in the Sobolev class $W^{1,1}_{\textrm{loc}}(U)$ and there exist constants $C,\alpha,\varepsilon_0>0$ with
	\begin{align}\label{david_cond}
		\sigma(\{z\in U: |\mu_H(z)|\geq 1-\varepsilon\}) \leq Ce^{-\alpha/\varepsilon}, \quad \varepsilon\leq \varepsilon_0.
	\end{align}
\end{defn}

Here $\sigma$ is the spherical measure, and
$\mu_H= \frac{\partial H/ \partial\overline{z}}{\partial H/\partial z}$
is the Beltrami coefficient of $H$ (see \cite[Chapter 20]{AIM09} for more background on David homeomorphisms).  

The following result guarantees the existence of continuous extensions of circle homeomorphisms as David homeomorphisms of $\D$. 

\begin{prop}\label{extension_david_normal}
Let $A:\mathbb{S}^1\to\mathbb{S}^1$ be a \pwfm map of degree $d$. Assume further that none of the periodic break-points of $A$ is asymmetrically hyperbolic. Then, there exists a homeomorphism $H:\bS^1\to\bS^1$ that conjugates $z^d$ to $A$, and admits a continuous extension to $\D$ as a David homeomorphism.
\end{prop}
\begin{proof}
The existence of the topological conjugacy $H$ between $z^d$ and $A$ follows from the fact that $A$ is an expansive circle covering map of degree $d$ (recall that a \pwfm map is an expansive circle covering; see  Definition~\ref{def-pwfm}).

We now proceed to establish the David extension statement. Denote the break-points of (the piecewise M{\"o}bius definition of) $A$ by $x_1, \cdots, x_k$ (ordered counterclockwise), and by $C_{j}$ (for $j\in \{1,\cdots, k\}$) the round circle containing the hyperbolic geodesic connecting the endpoints of $I_j=[x_j, x_{j+1}]$. Further, denote the complementary component of $C_j$ containing $I_j$ by $U_j$. Since $A$ is piecewise M{\"o}bius (defined by members of $\textrm{Aut}(\disk)$), the Markov partition defined by $I_1,\cdots, I_k$ satisfies the following `complex Markov' property: $A(U_j)\supset U_{j'},\ \textrm{whenever}\ A(I_j)\supset I_{j'}.$ Furthermore, by Lemma~\ref{no_mixed}, no periodic break-point $x_j$ is of mixed type. The assumption that no periodic $x_j$ is asymmetrically hyperbolic implies that every periodic $x_j$ is symmetrically hyperbolic or symmetrically parabolic. By \cite[Theorem~4.12]{LMMN}, the map $H$ can be continuously extended to a David homeomorphism of $\disk$. 
\end{proof}

We now employ the machinery of David homeomorphisms to prove our first conformal combination theorem for mateable maps associated to Fuchsian groups and the special class of Bers rational maps given by complex polynomials in principal hyperbolic components.
Recall that for a complex polynomial $P$, the \emph{filled Julia set} $\mathcal{K}(P)$ is the set of all points whose forward orbits (under $P$) stay bounded.

\begin{prop}\label{conformal_mating_general_prop}
	Let $A:\mathbb{S}^1\to\mathbb{S}^1$ be a mateable map of degree $d$ associated to a Fuchsian group, and $P\in\mathcal{H}_d$ where $\mathcal{H}_d$ is the hyperbolic component of degree $d$ polynomials containing the map $z^d$.
	Then, the maps $\widehat{A}:\mathcal{D}\to\overline{\disk}$ and $P:\mathcal{K}(P)\to\mathcal{K}(P)$ are conformally mateable. 
\end{prop}
\begin{proof}
As $P\in\mathcal{H}_d$, the filled Julia set $\mathcal{K}(P)$ is a closed Jordan disk. Hence, there exists a Blaschke product $B$ of degree $d$ with an attracting fixed point in $\overline{\D}^c$, and a conformal isomorphism $\kappa$ from $\mathcal{K}(P)$ onto $\D^c$ that conjugates $P\vert_{\mathcal{K}(P)}$ to $B\vert_{\D^c}$.

By definition, a mateable map is a \pwfm map without asymmetrically hyperbolic periodic break-points. According to Proposition~\ref{extension_david_normal}, there exists a homeomorphism $\eta:\bS^1\to\bS^1$ that conjugates $z^d$ to $A$, and admits a continuous extension to $\D$ as a David homeomorphism.
	
Furthermore, by \cite[Lemma~3.9]{LMMN}, there exists a quasi-symmetric homeomorphism $h:\mathbb{S}^1\to\mathbb{S}^1$ that conjugates $B$ to $z^d$. We extend $h$ to a quasiconformal homeomorphism of $\widehat{\C}$, and denote the extension also  by $h$. By \cite[Proposition~2.5]{LMMN} (part ii), the map $\eta\circ h$ is a David homeomorphism of $\D$. Clearly, $\eta\circ h$ conjugates $B\vert_{\mathbb{S}^1}$ to $A$.
	
	Consider the topological dynamical system
	
	\begin{equation*}
		G(w):=\left\{\begin{array}{ll}
			B & \mbox{on}\ \disk^c, \\
			h^{-1}\left(\eta^{-1}\left(\widehat{A}\left(\eta(h(w))\right)\right)\right) & \mbox{on}\ \disk\setminus h^{-1}\left(\eta^{-1}\left(R\right)\right).
		\end{array}\right.
	\end{equation*}
	(By equivariance of $\eta\circ h$, the two definitions agree on $\mathbb S^1$.) 
	
	We define a Beltrami coefficient $\mu$ in the sphere as follows. In $\disk^c$ we let $\mu$ be the standard complex structure. In $\disk$ we let $\mu$ be the pullback of the standard complex structure under the map $\eta\circ h$. Since $\eta\circ h$ is a David homeomorphism of $\disk$, it follows that $\mu$ is a David coefficient on $\widehat{\C}$. 
	
	By the David Integrability Theorem \cite{David88} \cite[Theorem~20.6.2, p.~578]{AIM09}, there exists a David homeomorphism $\pmb{\Psi}$ of $\widehat{\C}$ with $\mu_{\pmb{\Psi}}=\mu$. Consider the map 
	$$
	F=   \pmb{\Psi} \circ G\circ \pmb{\Psi}^{-1}: \widehat{\C}\setminus \pmb{\Psi}(h^{-1}\left(\eta^{-1}\left(R\right)\right))\to\widehat{\C}.
	$$
We set $\mathbf{D}^\textrm{in}:=\pmb{\Psi}(\disk)$, $\mathbf{D}^\textrm{out}:=\pmb{\Psi}(\bbar{\disk}^c)$, and claim that $F$ is analytic on the interior of its domain of definition $\mathrm{Dom}(F)= \widehat{\C}\setminus \pmb{\Psi}(h^{-1}\left(\eta^{-1}\left(R\right)\right))$, and that $F$ is conformally conjugate to $P$ on $\mathbf{D}^\textrm{out}$ and to $\widehat{A}$ on $\mathbf{D}^\textrm{in}$.
	To see this, note first that
	by \cite[Theorem~2.2]{LMMN}, the map $\phi^\textrm{in}:=\pmb{\Psi}\circ h^{-1}\circ \eta^{-1}$ is conformal on $\disk$. That $\phi^\textrm{out}:=\pmb{\Psi}\circ\kappa$ is conformal on $\mathcal{K}(P)$ follows from the definition of $\mu$. This proves the claims regarding the conformal conjugacies. 
	It follows that $F$ is conformal in $\Int{(\mathrm{Dom}(F))}\setminus \pmb{\Psi}(\mathbb S^1)$. Since $\mathbb S^1$ is removable for $W^{1,1}$ functions, we conclude from \cite[Theorem~2.7]{LMMN} that $\pmb{\Psi}(\mathbb S^1)$ is locally conformally removable. This implies that $F$ is conformal on the interior of $\mathrm{Dom}(F)$.
\end{proof}

\section{Bowen-Series maps for Bers slice groups}\label{bs_sec}
Examples of \pwfm maps of the circle that are orbit equivalent to finitely generated Fuchsian groups are given by \emph{Bowen-Series maps}, which first  appeared in the work of Bowen and Series \cite{bowen,bowen-series}.

A finitely generated Fuchsian group $\Gamma$ (of the first kind) has a fundamental domain $R\left(\subset\D\right)$ that is a (possibly ideal) hyperbolic polygon. Denote the edges of $R$ by $\{s_i\}_{i=1}^n$ (labeled in counterclockwise order around the circle). Each edge $s_i$ of $R$ is identified with another edge $s_j$ by a corresponding element $h(s_i)\in\Gamma$. The set $\{h(s_i)\}_{i=1}^n$ forms a generating set for $\Gamma$.

Let $C(s_i)$ be the Euclidean circular arc in $\D$ containing $s_i$  and meeting  $\bS^1$ orthogonally (see \cite[Figure~2]{bowen-series} or \cite[Figure~1]{sullivan_survey}). Further, let $\mathcal{N}$ be the net in $\D$ consisting of all images of edges of $R$ under elements of $\Gamma$. The fundamental domain $R$ is said to satisfy the \emph{even corners} property if $C(s_i)$ lies completely in $\mathcal{N}$, for $i\in\{1,\cdots, n\}$.

\begin{defn}[Bowen-Series map]\label{b_s_map_def}
	Suppose that a fundamental domain $R$ of $\Gamma$ satisfies the even corners property. Label (following \cite{bowen-series}) the endpoints of $C(s_i)$ on $\bS^1$, $P_i, Q_{i+1}$ (with $Q_{n+1}=Q_1$) with $P_i$ occurring before $Q_{i+1}$ in the counterclockwise order. These points occur along the circle in the order $P_1, Q_1, P_2, Q_2$, $\cdots$, $P_n, Q_n$.
The \emph{Bowen-Series map} $A_{\Gamma, \textrm{BS}}:\bS^1\to\bS^1$ of $\Gamma$ (associated with the fundamental domain $R$) is defined piecewise as $A_{\Gamma, \textrm{BS}}\equiv h(s_i)$, on the sub-arc $[P_i, P_{i+1})$ of $\bS^1$ (traversed in the counterclockwise order).
\end{defn}

\begin{prop}\cite[Lemma~2.4]{bowen-series}\label{orbit_equiv_prop}
	The map $A_{\Gamma, \textrm{BS}}$ is orbit equivalent to $\Gamma$, except (possibly) at finitely many pairs of points modulo the action of $\Gamma$.
\end{prop}

For ease of notation, we will drop the subscript `$\textrm{BS}$' from $A_{\Gamma, \textrm{BS}}$ and denote it simply by $A_{\Gamma}$ throughout this section. Since Bowen-Series maps are the only \pwfm maps considered in this section, this will not lead to any confusion.

It is not hard to see that for Fuchsian groups uniformizing positive genus surfaces (possibly with punctures), the corresponding Bowen-Series maps are discontinuous (cf. \cite[\S 3.1]{sullivan_survey}).
Thus, to get continuous Bowen-Series maps, we need to restrict our attention to punctured sphere groups (possibly with orbifold points). In fact, it turns out that Bowen-Series maps of Fuchsian groups uniformizing spheres with finitely many punctures and, possibly, one/two order two orbifold points are coverings of $\mathbb{S}^1$ with degree at least two.

\subsection{Bowen-Series maps for Fuchsian punctured sphere groups}\label{b_s_punc_sphere_subsec}

In this subsection, we will study some general properties of Bowen-Series maps of punctured sphere Fuchsian groups. We will first introduce a Fuchsian group $G_d$ uniformizing a $(d+1)$-times punctured sphere equipped with a symmetric fundamental domain, for which the associated Bowen-Series map is a $C^1$ covering map of the circle.

Fix $d\geq 2$. For $j\in\{1,\cdots, d\}$, let $C_j$ be the hyperbolic geodesic of $\disk$ connecting $p_j:=e^{\pi i (j-1)/d}$ and $p_{j+1}:=e^{\pi i j/d}$, and $C_{-j}$ be the image of $C_j$ under reflection in the real axis. We further denote the complex conjugate of $p_j$ by $p_{-j}$, $j\in\{2,\cdots, d\}$. The M{\"o}bius automorphism $g_j$ of $\disk$ is defined as reflection in $C_j$ followed by complex conjugation. Evidently, $g_j$ sends $C_j$ onto $C_{-j}$  (cf.\ Figure~\ref{fund_dom_punctured_sphere_fig}). Note also that for $j\in\{1,\cdots,d-1\}$, the map $g_{j+1} g_j^{-1}$ is the composition of reflections in the circular arcs $C_{j+1}$ and $C_{j}$ that touch at $p_{j+1}$. It is now routine to check that $g_{j+1} g_j^{-1}$ is parabolic with its unique fixed point at $p_{j+1}$. Similarly, the maps $g_1, g_d$ are also parabolic with their unique fixed points at $p_1, p_{d+1}$, respectively.
Let
$$
G_d:=\langle g_1,\cdots, g_d\rangle.
$$
It follows that $G_d$ is a Fuchsian group with fundamental domain $R$ having $C_1,\cdots, C_d,$ $C_{-d}, \cdots, C_{-1}$ as its edges. Moreover, $\disk/G_d$ is a $(d+1)$-times punctured sphere.

\begin{figure}[h!]
\begin{tikzpicture}
\node[anchor=south west,inner sep=0] at (1,5) {\includegraphics[width=0.6\linewidth]{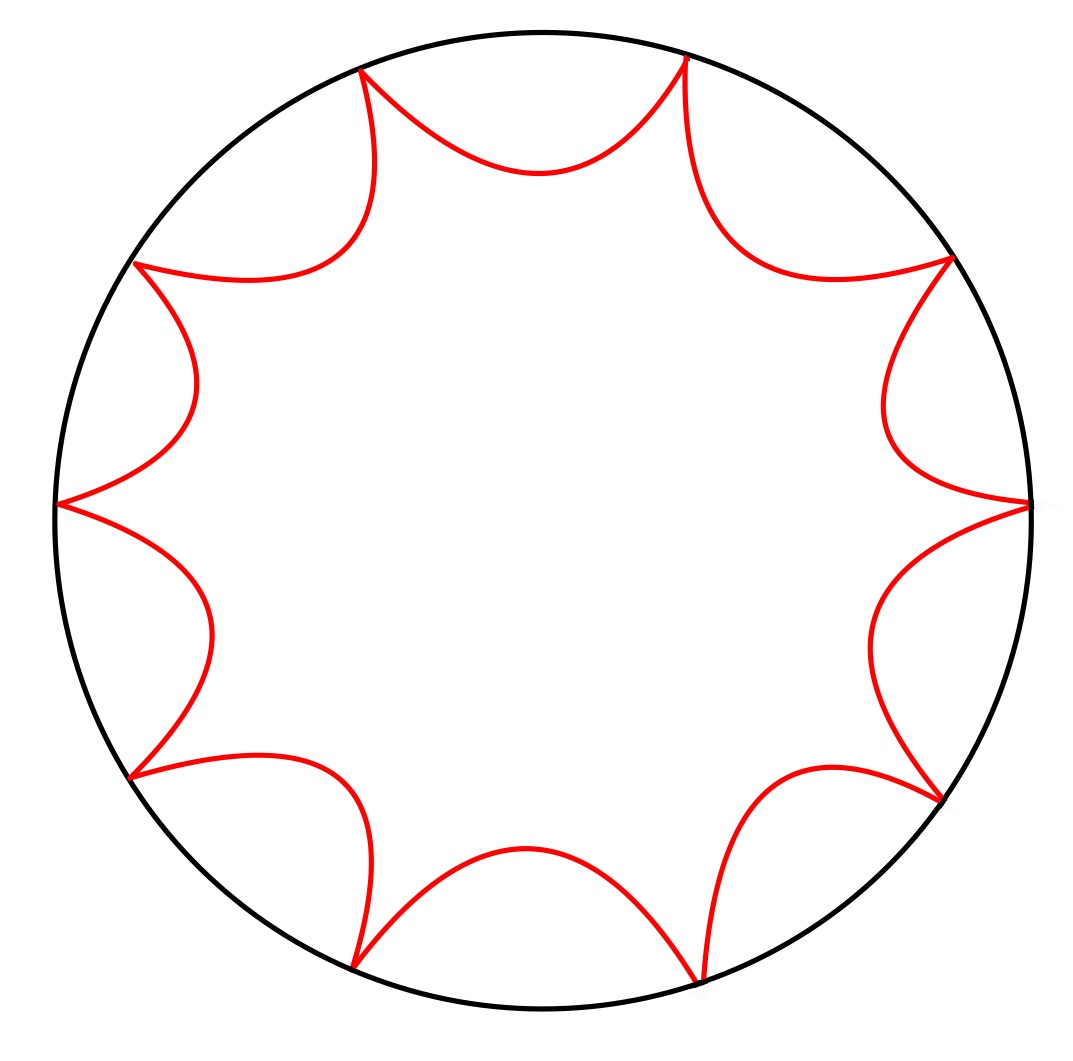}};
\node[anchor=south west,inner sep=0] at (6,1) {\includegraphics[width=0.5\linewidth]{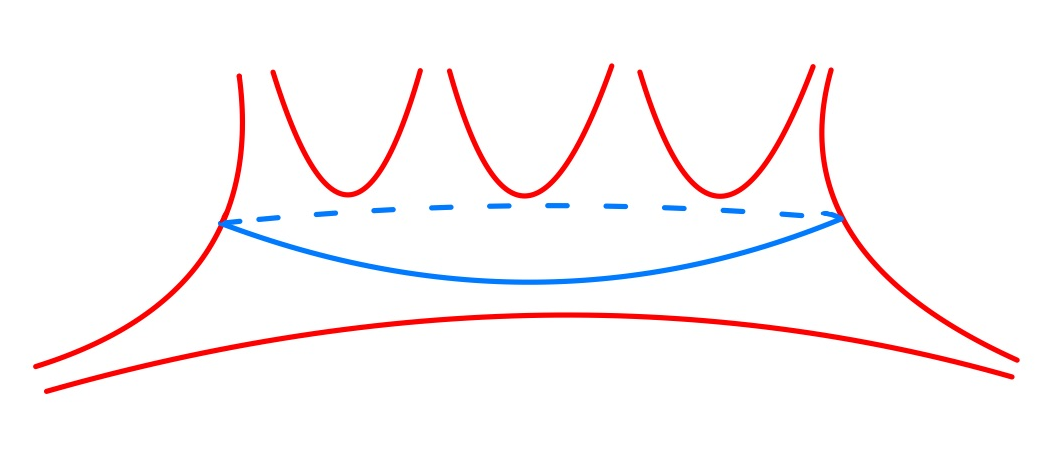}};
\node at (4.8,9) {\begin{Huge}$R$\end{Huge}};
\node at (6.9,9.4) {$C_1$};
\node at (6.7,8.1) {$C_{-1}$};
\node at (6.2,10.4) {$C_2$};
\node at (6,7.1) {$C_{-2}$};
\node at (4.8,11) {$C_3$};
\node at (4.8,6.7) {$C_{-3}$};
\node at (3.48,10.36) {$C_4$};
\node at (3.5,7.28) {$C_{-4}$};
\node at (2.75,9.54) {$C_5$};
\node at (2.96,8.1) {$C_{-5}$};
\node at (8.6,9.8) {$g_1$};
\node at (7.2,11.8) {$g_2$};
\node at (4.8,12.6) {$g_3$};
\node at (2.5,11.8) {$g_4$};
\node at (1.2,9.8) {$g_5$};
\node at (1.1,7.7) {$g_{5}^{-1}$};
\node at (2.6,5.7) {$g_{4}^{-1}$};
\node at (4.8,5) {$g_{3}^{-1}$};
\node at (7.4,5.9) {$g_{2}^{-1}$};
\node at (8.6,7.6) {$g_{1}^{-1}$};
\node at (8.7,8.9) {$p_1$};
\node at (1,8.9) {$p_6$};
\node at (8.08,10.78) {$p_2$};
\node at (8.08,6.7) {$p_{-2}$};
\node at (6,12.4) {$p_3$};
\node at (6.1,5.22) {$p_{-3}$};
\node at (3.44,12.36) {$p_4$};
\node at (3.5,5.32) {$p_{-4}$};
\node at (1.66,10.8) {$p_5$};
\node at (1.66,6.8) {$p_{-5}$};
\node at (6.48,3.42) {\begin{small}$\left[p_5\right]=\left[p_{-5}\right]$\end{small}};
\node at (8.54,3.64) {\begin{tiny}$\left[p_4\right]=\left[p_{-4}\right]$\end{tiny}};
\node at (10,3.7) {\begin{tiny}$\left[p_3\right]=\left[p_{-3}\right]$\end{tiny}};
\node at (12,3.42) {\begin{small}$\left[p_2\right]=\left[p_{-2}\right]$\end{small}};
\node at (5.7,1.4) {\begin{small}$\left[p_6\right]$\end{small}};
\node at (12.64,1.4) {\begin{small}$\left[p_1\right]$\end{small}};
\end{tikzpicture}
\caption{A fundamental domain of $G_5$}
\label{fund_dom_punctured_sphere_fig}
\end{figure}
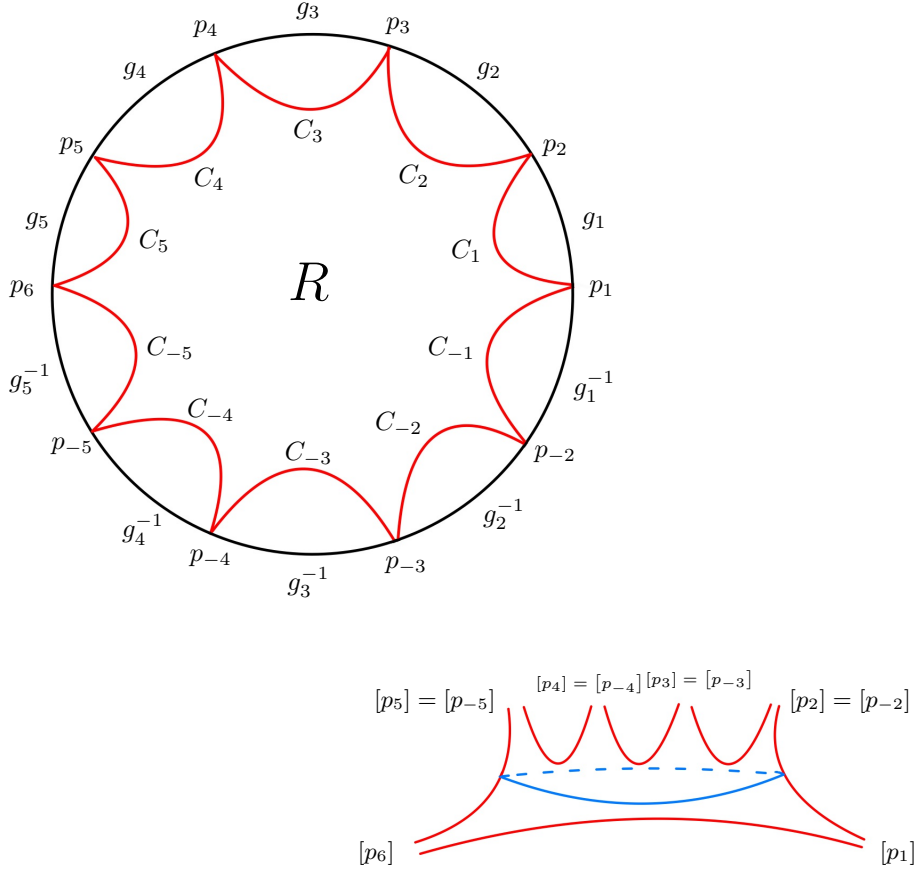
In Figure \ref{fund_dom_punctured_sphere_fig}, the fundamental domain $R$ uniformizes a six times punctured sphere. All ten vertices of $R$ are on $\mathbb{S}^1$, and they cut the circle into ten arcs. The corresponding Bowen-Series map acts on these arcs by the generators $g_j^{\pm 1}$ displayed next to them.
For $j\in\{1,\cdots, d\}$, we denote the arc of $\mathbb{S}^1$ connecting $p_j$ to $p_{j+1}$ by $I_j$, and the image of $I_j$ under reflection in the real axis by $I_{-j}$. Note that $A_{G_d}$ acts on $I_{\pm j}$ by $g_j^{\pm 1}$.

\begin{prop}\label{b_s_poly_conjugate_prop_1}
1) For $d\geq 2$, the Bowen-Series map $A_{G_d}:\bS^1\to\bS^1$ of $G_d$ (equipped with the fundamental domain $R$) is a $C^1$ \pwfm map of degree $2d-1$.
	
2) $A_{G_d}$ is a mateable map whose pieces generate the Fuchsian group $G_d$; in particular, $A_{G_d}$ is orbit equivalent to $G_d$.
\end{prop}
\begin{proof}
	1) As $A_{G_d}\equiv g_j^{\pm 1}$ on $I_{\pm j}$, it maps $I_{\pm j}$ onto $\mathbb{S}^1\setminus \Int{I_{\mp j}}$ as an orientation-preserving homeomorphism, fixes $p_1, p_{d+1}$, and maps $p_{\pm i}$ to $p_{\mp i}$ (for $i\in\{2,\cdots, d\}$). It easily follows from these properties that $A_{G_d}$ is a degree $2d-1$ covering of $\mathbb{S}^1$. It is also straightforward to check that at the endpoints of any $I_{\pm j}$, the left and right derivatives of $A_{G_d}$ agree. Hence, $A_{G_d}$ is $C^1$.
	
	Expansiveness of $A_{G_d}$ is a consequence of the facts that it is a $C^1$ covering map, it has only finitely many fixed points, and $\vert A'_{G_d}\vert\geq 1$ on $\mathbb{S}^1$ with equality only at the parabolic fixed points (compare \cite[Lemma~3.7]{LMMN}). 
	
	Clearly, $A_{G_d}$ is piecewise Fuchsian. To see that $A_{G_d}$ is \pwfm, observe that the arcs $I_1, I_{-1}, \cdots, I_d, I_{-d}$ form a Markov partition for $A_{G_d}$ with transition matrix
	$$
	\begin{bmatrix}
		1&0&1&1&\cdots&1&1\\
		0&1&1&1&\cdots&1&1\\
		\hdotsfor{7}\\
		1&1&1&1&\cdots&1&0\\
		1&1&1&1&\cdots&0&1\\
	\end{bmatrix}.
	$$
	Thus, $A_{G_d}$ is a \pwfm map.
	
	2) As $A_{G_d}$ is $C^1$, it has no asymmetrically hyperbolic periodic break-point. Since the pieces of $A_{G_d}$ generate the group $G_d$, we only need to check that $A_{G_d}$ is orbit equivalent to $G_d$.

	To this end, let us pick $x,y\in\mathbb{S}^1$ in the same grand orbit of $A_{G_d}$. Since $A_{G_d}$ acts by the generators $g_i^{\pm 1}$ ($i\in\{1,\cdots, d\}$) of the group $G_d$, it directly follows that there exists an element of $G_d$ that takes $x$ to $y$; i.e., $x$ and $y$ lie in the same $G_d$-orbit.
	
	Conversely, let $x,y\in\mathbb{S}^1$ lie in the same $G_d$-orbit; i.e., there exists $g\in G_d$ with $g(x)=y$. It suffices to prove grand orbit equivalence of $x$ and $y$ (under $A_{G_d}$) in the case where $g=g_i$. Therefore, we assume that $g_i(x)=y$. Note that $g_i$ carries $I_i$ to $\mathbb{S}^1\setminus \Int{I_{-i}}$. Thus, if $y\notin I_{-i}$, then $x$ must lie in $I_i$, and we have that $A_{G_d}(x)=g_i(x)=y$. On the other hand, if $y\in I_{-i}$, then we have that $A_{G_d}(y)=g_i^{-1}(y)=x$. In either case, $x$ and $y$ lie in the same grand orbit of $A_{G_d}$.
\end{proof}

\begin{rmk}
It is worth emphasizing that if the group $G_d$ is equipped with a different fundamental domain or a different pattern of side-pairings, then the resulting Bowen-Series map may be discontinuous. For instance, the four punctured sphere admits
an ideal hexagon as a fundamental domain with sides $C_1, \cdots C_6$ (arranged cyclically) and side-pairing
 transformations $g_1, g_2, g_3$ such that $g_i$ pairs $C_{2i-1}, C_{2i}$, for $i=1,2,3$. It is easy to that the associated Bowen-Series map is discontinuous.
\end{rmk}

\subsection{Bowen-Series maps for Fuchsian punctured sphere groups with torsion points}\label{b_s_punc_sphere_orbifold_subsec}

We now consider Bowen-Series maps for Fuchsian groups uniformizing punctured spheres with one/two orbifold points of order two. As in Subsection~\ref{b_s_punc_sphere_subsec}, we will first study the Bowen-Series maps associated with two specific Fuchsian groups $G_{d,1}, G_{d,2}$ (uniformizing spheres with $d$ punctures and one/two order two orbifold points respectively) equipped with special fundamental domains.

\subsubsection{The case of two orbifold points}

Fix $d\geq 2$. For $j\in\{1,\cdots, d\}$, let $C_j$ be the hyperbolic geodesic of $\disk$ connecting $p_j:=e^{\pi i (j-1)/d}$ and $p_{j+1}:=e^{\pi i j/d}$, and $C_{-j}$ be the image of $C_j$ under reflection in the real axis. Let us denote the common perpendicular bisector of $C_1$ and $C_{-d}$ by $\ell$. Consider the order two M{\"o}bius automorphism $g_1$ (respectively, $g_{-d}$) of $\disk$ defined as reflection in $C_1$ (respectively, $C_{-d}$) followed by reflection in $\ell$. Then, $g_1, g_{-d}$ preserve $C_1, C_{-d}$, and have elliptic fixed points on $C_1, C_{-d}$ (respectively). Next, consider the M{\"o}bius automorphisms $g_j$ ($j\in\{2, \cdots, d\}$) of $\disk$ defined as reflection in $C_j$ followed by reflection in $\ell$. By construction, $g_j$ carries $C_j$ onto $C_{1-j}$. It is also easy to see that each $g_{j+1} g_j^{-1}$ is a parabolic map with the unique fixed point at $p_{j+1}$ (where $j\in\{1,\cdots,d-1\}$), and $g_d g_{-d}$ is a parabolic map with  unique fixed point at $p_{d+1}$.   We define 
$$
G_{d,2}:=\langle g_1, g_2\cdots, g_d, g_{-d}\rangle.
$$
Then, $G_{d,2}$ is a Fuchsian group with a fundamental domain $R$ having $C_1,\cdots, C_d,$ $C_{-d}, \cdots, C_{-1}$ as its edges (see Figure~\ref{orbifold_punc_sphere_bs_fig} (left)). Moreover, $\disk/G_{d,2}$ is a sphere with $d$ punctures and two order two orbifold points.

Arguments similar to the ones used in the proof of Proposition~\ref{b_s_poly_conjugate_prop_1} yield the following result (we omit the proof).

\begin{prop}\label{b_s_poly_conjugate_prop_2}
1) For $d\geq 2$, the Bowen-Series map $A_{G_{d,2}}:\bS^1\to\bS^1$ of $G_{d,2}$ (equipped with the fundamental domain $R$) is a $C^1$ \pwfm map of degree $2d-1$.
	
2) $A_{G_{d,2}}$ is a mateable map whose pieces generate the Fuchsian group $G_{d,2}$; in particular, $A_{G_{d,2}}$ is orbit equivalent to $G_{d,2}$.
\end{prop}

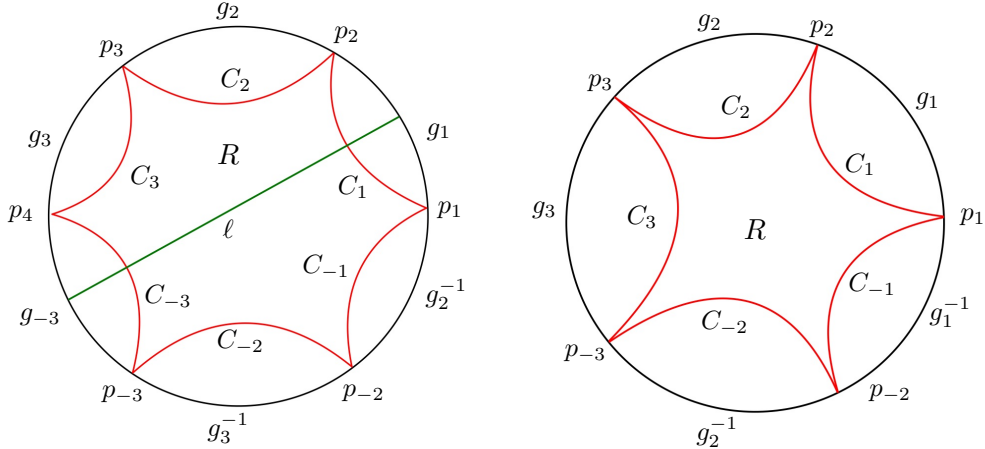
\begin{figure}[h!]
	\begin{tikzpicture}
		\node[anchor=south west,inner sep=0] at (1,0) {\includegraphics[width=0.96\linewidth]{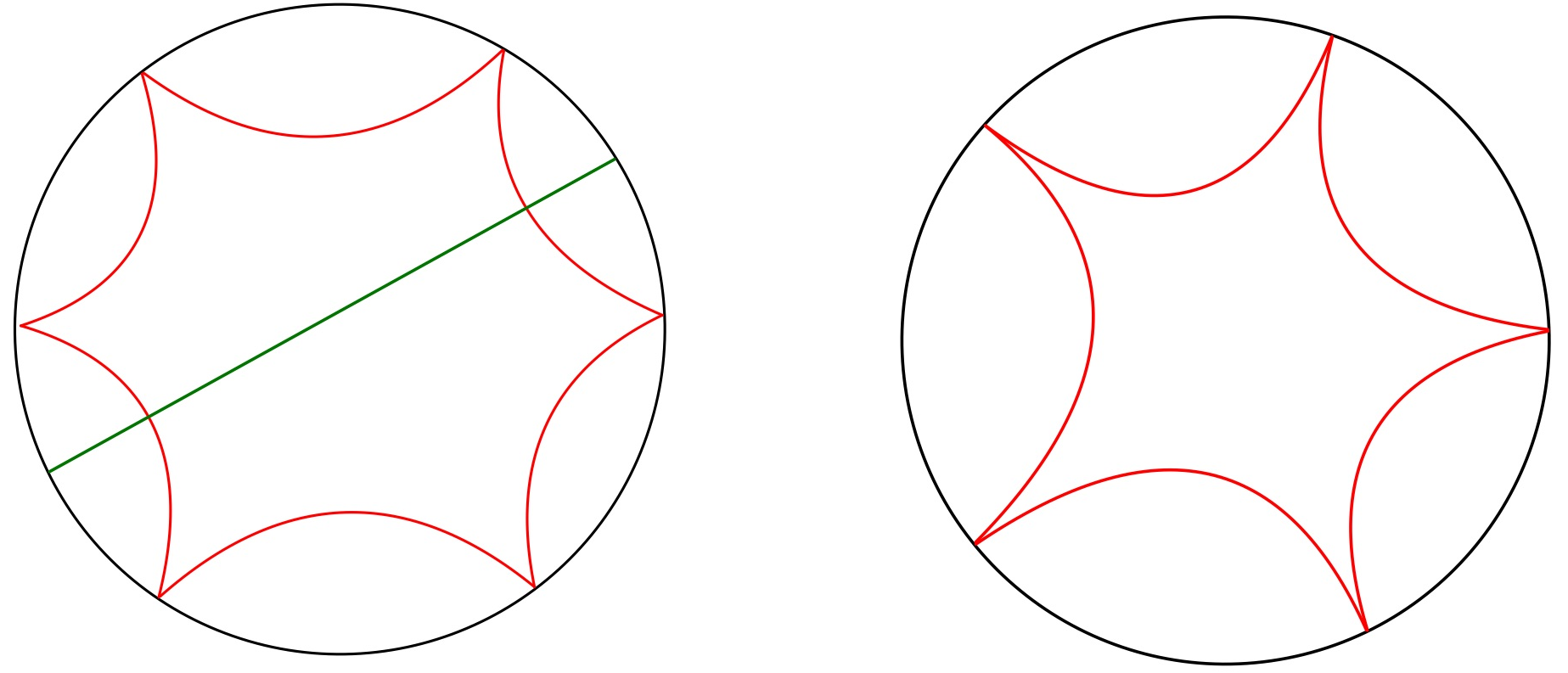}};
		\node at (3.5,3.5) {\begin{large}$R$\end{large}};
		\node at (3.5,2.5) {$\ell$};
		\node at (5.16,3.06) {$C_1$};
		\node at (4.8,2) {$C_{-1}$};
		\node at (3.6,4.48) {$C_2$};
		\node at (3.66,1) {$C_{-2}$};
		\node at (2.4,3.2) {$C_3$};
		\node at (2.7,1.6) {$C_{-3}$};
		\node at (6.3,3.8) {$g_1$};
		\node at (6.4,1.64) {$g_{2}^{-1}$};
		\node at (3.5,5.4) {$g_2$};
		\node at (3.5,-0.15) {$g_{3}^{-1}$};
		\node at (1,3.7) {$g_3$};
		\node at (1,1.32) {$g_{-3}$};
		\node at (6.45,2.75) {$p_1$};
		\node at (5.08,5.08) {$p_{2}$};
		\node at (1.96,4.9) {$p_{3}$};
		\node at (0.75,2.7) {$p_{4}$};
		\node at (5.3,0.32) {$p_{-2}$};
		\node at (2.1,0.3) {$p_{-3}$};
		
		\node at (10.5,2.5) {\begin{large}$R$\end{large}};
		\node at (11.9,3.36) {$C_1$};
		\node at (12.06,1.75) {$C_{-1}$};
		\node at (10.25,4.12) {$C_2$};
		\node at (10.1,1.25) {$C_{-2}$};
		\node at (9,2.66) {$C_3$};
		\node at (12.8,4.2) {$g_1$};
		\node at (13.1,1.4) {$g_{1}^{-1}$};
		\node at (9.9,5.25) {$g_2$};
		\node at (10,-0.2) {$g_2^{-1}$};
		\node at (7.7,2.75) {$g_3$};
		\node at (13.4,2.66) {$p_1$};
		\node at (11.4,5.16) {$p_{2}$};
		\node at (8.45,4.45) {$p_{3}$};
		\node at (12.3,0.32) {$p_{-2}$};
		\node at (8.25,0.84) {$p_{-3}$};
	\end{tikzpicture}
	\caption{The preferred fundamental domains $R$ of $G_{3,2}$ (on the left) and $G_{3,1}$ (on the right) are shown. The group $G_{3,2}$ (respectively, $G_{3,1}$) uniformizes a thrice punctured sphere with two (respectively, one) orbifold point(s) of order two. The actions of the corresponding Bowen-Series map on $\bS^1$ are also displayed.}
	\label{orbifold_punc_sphere_bs_fig}
\end{figure}

\subsubsection{The case of one orbifold point}

Fix $d\geq 2$. For $j\in\{1,\cdots, d\}$, let $C_j$ be the hyperbolic geodesic of $\disk$ connecting $p_j:=e^{2\pi i (j-1)/(2d-1)}$ and $p_{j+1}:=e^{2\pi i j/(2d-1)}$, and $C_{-j}$ be the image of $C_j$ under reflection in the real axis. Note that $C_d$ is the same as $C_{-d}$. Consider the M{\"o}bius automorphism $g_j$ of $\disk$ defined as reflection in $C_j$ followed by complex conjugation, for $j\in\{1,\cdots,d\}$. Then, $g_d$ has order two, preserves $C_d$, and has an elliptic fixed point on $C_d$. On the other hand, $g_j$ carries $C_j$ onto $C_{-j}$, for $j\in\{1,\cdots,d-1\}$. It is readily checked that $g_1^2$ is a parabolic transformation with  unique fixed point at $p_1$, while each $g_{j+1}g_j^{-1}$ is parabolic with  unique fixed point at $p_{j+1}$, for $j\in\{1,\cdots,d-1\}$.
We define 
$$
G_{d,1}:=\langle g_1, \cdots, g_d\rangle.
$$
Then, $G_{d,1}$ is a Fuchsian group with a fundamental domain $R$ having $C_1,\cdots, C_d,$ $C_{-(d-1)}, \cdots, C_{-1}$ as its edges (see Figure~\ref{orbifold_punc_sphere_bs_fig} (right)). Moreover, $\disk/G_{d,1}$ is a sphere with $d$ punctures and one order two orbifold point. Again, as in  Proposition~\ref{b_s_poly_conjugate_prop_1} we obtain the following result (and omit the proof).

\begin{prop}\label{b_s_poly_conjugate_prop_3}
1) For $d\geq 2$, the Bowen-Series map $A_{G_{d,1}}:\bS^1\to\bS^1$ of $G_{d,1}$ (equipped with the fundamental domain $R$) is a $C^1$ \pwfm map of degree $2d-2$.
	
2) $A_{G_{d,1}}$ is a mateable map whose pieces generate the Fuchsian group $G_{d,1}$; in particular, $A_{G_{d,1}}$ is orbit equivalent to $G_{d,1}$.
\end{prop}

\subsection{Mating Bowen-Series maps of Bers slice groups with polynomials}\label{sec-puncturedspherebuildingblock}

We now describe the simplest instance of our mating construction. For $d\geq 2$, consider the Fuchsian group $\Gamma_0\in\{G_d, G_{d,1}, G_{d,2}\}$ and the Bowen-Series map $A_{\Gamma_0}:\bS^1\rightarrow\bS^1$, which is a mateable map associated to $\Gamma_0$. For $\Gamma\in\mathcal{B}(\Gamma_0)$, we call the mateable map $A_\Gamma$ associated to $\Gamma$ and compatible with $A_{\Gamma_0}$ the \emph{Bowen-Series map} of $\Gamma$ (see Subsection~\ref{mateable_gen_subsec}).
Set $k=2d-1\ \textrm{if}\ \Gamma_0\in\{G_d, G_{d,2}\},\ \textrm{and}\ k=2d-2\ \textrm{if}\ \Gamma_0=G_{d,1}.$
Recall that $\mathcal{H}_k$ stands for the principal hyperbolic component in the space of degree $k$ polynomials.

\begin{theorem}\label{moduli_interior_mating_thm}
	Let $\Gamma\in\mathcal{B}(\Gamma_0)$, and $P\in\mathcal{H}_k$. Then, the maps $\widehat{A}_{\Gamma}:\mathcal{D}_{A_\Gamma}\to K(\Gamma)$ and $P:\mathcal{K}(P)\to\mathcal{K}(P)$ are conformally mateable.
\end{theorem}
\begin{proof}
Note that $\Gamma_0$ is conjugate to $\Gamma$ via a quasiconformal homeomorphism, and this quasiconformal map conjugates $\widehat{A}_{\Gamma_0}$ to $\widehat{A}_{\Gamma}$. Thus by a standard quasiconformal deformation argument, it suffices to show that the map $\widehat{A}_{\Gamma_0}:\mathcal{D}_{A_{\Gamma_0}}\to K(\Gamma_0)=\overline{\disk}$ is mateable with $P$. Since $A_{\Gamma_0}:\bS^1\to\bS^1$ is a mateable map of degree $k$, the existence of the desired conformal mating follows from Proposition~\ref{conformal_mating_general_prop}.
\end{proof}

\section{Higher Bowen-Series maps for Bers slice groups}\label{sec-fold}
The aim of this section is to describe a new class of mateable maps orbit equivalent to groups in the Bers slice of a Fuchsian punctured sphere group, 
	and conformally mateable with polynomials. 
The motivation here is different from the symbolic coding of  geodesics \cite{series-acta} as is our approach.

\subsection{Higher Bowen-Series maps for Fuchsian punctured sphere groups}\label{hbs_fuch_subsec}
Recall from Definition~\ref{def-canonicalextension-domain} that the fundamental domain of a \pwfm map $A$ is denoted by $R$. The set $\mathcal{D}=\overline{\D}\setminus R$ is the canonical domain of definition of $\widehat{A}$ in $\overline{\D}$.

Recall (Definition \ref{def-canonicalextension-domain}) that a side of $R$ is not regarded as a diagonal.

\begin{defn}\label{def-nofold}
	A \pwfm map $A: \bS^1 \to \bS^1$ is said to have a \emph{diagonal fold} if there
	exist consecutive edges $\alpha_1, \alpha_2$ of $\partial R$
	and a diagonal $\delta$ of $R$ such that $\hA(\alpha_i) = \delta$ for $i=1,2$. Note that if $a_1, a_2$ (resp. $ a_2, a_3$) are the endpoints of $\alpha_1$ (resp. $\alpha_2$) and $p, q$ are the endpoints of $\delta$, then
	$A(a_1)=p=A(a_3)$ and $A(a_2)=q$ by continuity of $A$ on $\bS^1$.
	
	A \pwfm map $A: \bS^1 \to \bS^1$ is said to be a \emph{\hdm} if
	\begin{enumerate}
		\item there exists an (open) ideal polygon $D \subset R$ such that all the edges 
		$\delta_1, \cdots, \delta_l$ of $D$ are  (necessarily non-intersecting) diagonals of $R$. We assume further that $\delta_1, \cdots, \delta_l$ are  cyclically ordered along $\partial D$. We shall call $D$ the \emph{inner domain} of $A$.
		\item If $p$ is an ideal vertex of $D$, then $A(p)=p$.
		\item For every edge $\alpha$ of $R$, $\hA(\alpha)$ is one of the diagonals
		$\delta_1, \cdots, \delta_l$.
		\item $A$ has no diagonal folds. 
	\end{enumerate}
\end{defn}

Cyclically ordering the edges $\alpha_1, \cdots, \alpha_k$ of $R$, it follows
from Definition \ref{def-nofold}, that under a \hdm $A$, consecutive edges
$\alpha_i, \alpha_{i+1}$ of $R$ go to consecutive edges of $D$. Note however that a counterclockwise cyclic ordering of edges of $R$ may be taken to a 
clockwise cyclic ordering of edges of $D$ under $A$.
In any case we have a continuous map $\hA: \partial R \to \partial D$.
Adjoining the ideal endpoints of $R$ and $D$, $\hA$ has a well-defined
degree $d$. 
Further, each edge of $D$ has exactly $|d|$ pre-images under $\hA$ since there are no folds. Also, since each $\delta_i$ is a diagonal of $R$,
$|d| >1$. We call $|d|$ the \emph{polygonal degree} of $A$. (Since $|d|>1$ we call $A$ a \hdm.)

 We refer the reader to Figure~\ref{cfmgen} below. Fix a (closed) fundamental domain of $\Gamma_0=G_{k-1}$ (see Subsection~\ref{b_s_punc_sphere_subsec} for the definition of $G_{k-1}$), given by an ideal $(2k-2)$-gon $W$ (the figure illustrates the $k=4$ case). For definiteness, let us assume that the ideal vertices of $W$ are the $(2k-2)$-th roots of unity. We fix the 
 following notation.
\begin{enumerate}
	\item The vertices of $W$ on the bottom semi-circle are numbered $1=1_-$, $2_-$ $\cdots$, $k_-=k$ in counterclockwise order.
	\item  The vertices of $W$ on the top semi-circle are numbered $1, 2, \cdots, k$ in clockwise order.
	\item Between vertices $i, i+1$ (and including $i, i+1$) on the top semi-circle, there are 
	$2k-2$ vertices given by the vertices of $g_i.W$ (noting that $g_i.W \cap W$ equals the bi-infinite geodesic $\overline{i (i+1)}$). We label the 
	$2k-4$ vertices strictly between $i, i+1$ as $\{i,2\}, \{i,3\}, \cdots,
	\{i,2k-3\}$ in clockwise order. 
\end{enumerate}

The generators of $\Gamma_0$ are given by $g_1, \cdots, g_{k-1} $, where $g_i$ takes the edge $\overline{i_- (i+1)_-}$ to the  bi-infinite geodesic $\overline{i (i+1)}$.

Define  $R$ as
$$
R=\Int{\left(W \cup \bigcup_{i=1, \cdots, k-1} g_i.W\right)}, 
$$ 
so that $\overline{i (i+1)}$
are diagonals of $R$.

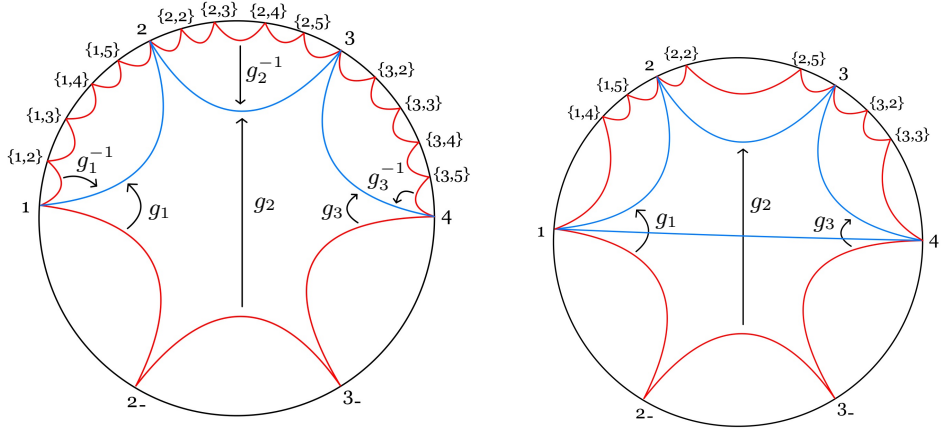
\begin{figure}[ht!]
\begin{tikzpicture}
\node[anchor=south west,inner sep=0] at (0,0) {\includegraphics[width=0.5\linewidth]{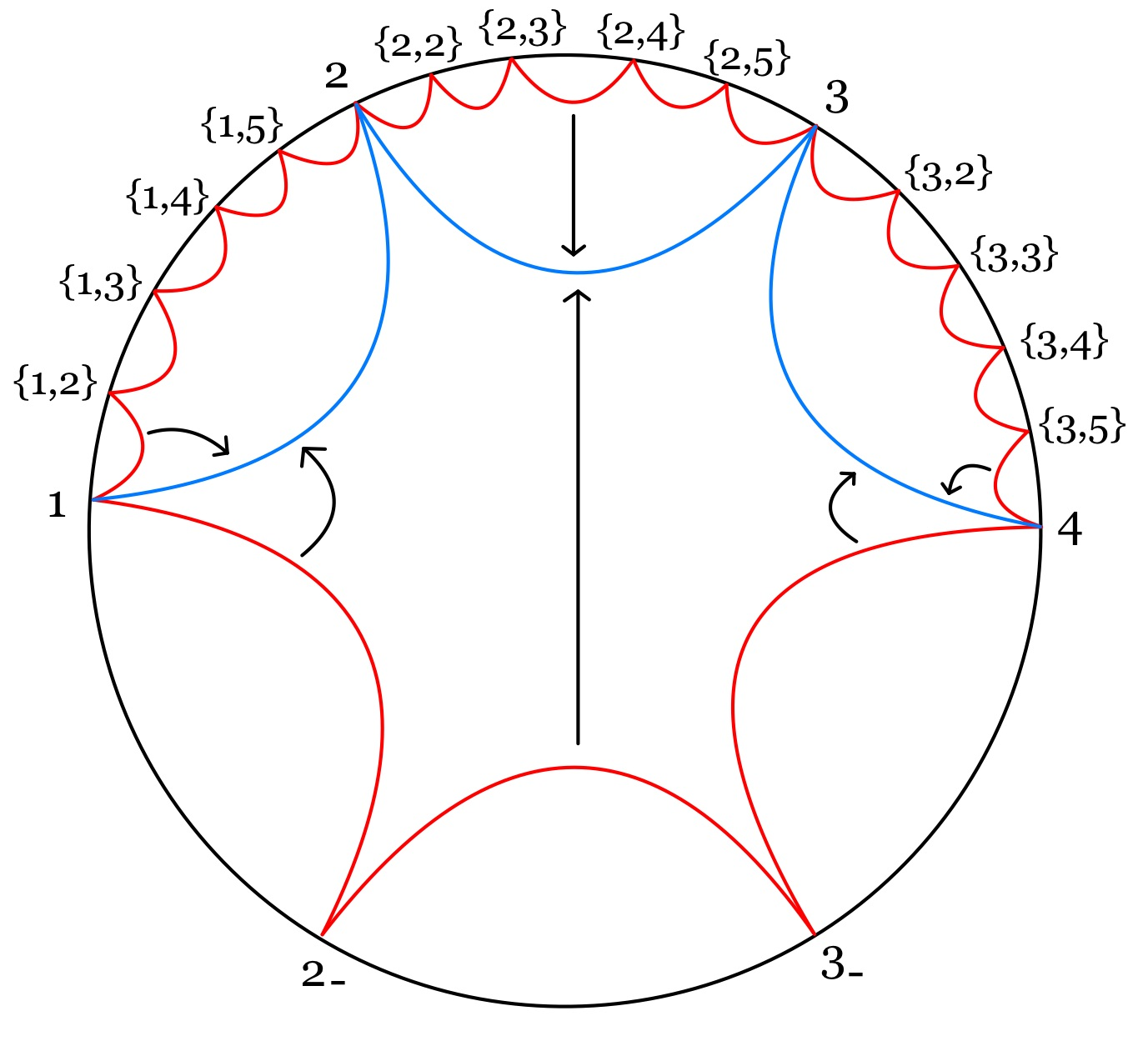}};
\node[anchor=south west,inner sep=0] at (7,0) {\includegraphics[width=0.44\linewidth]{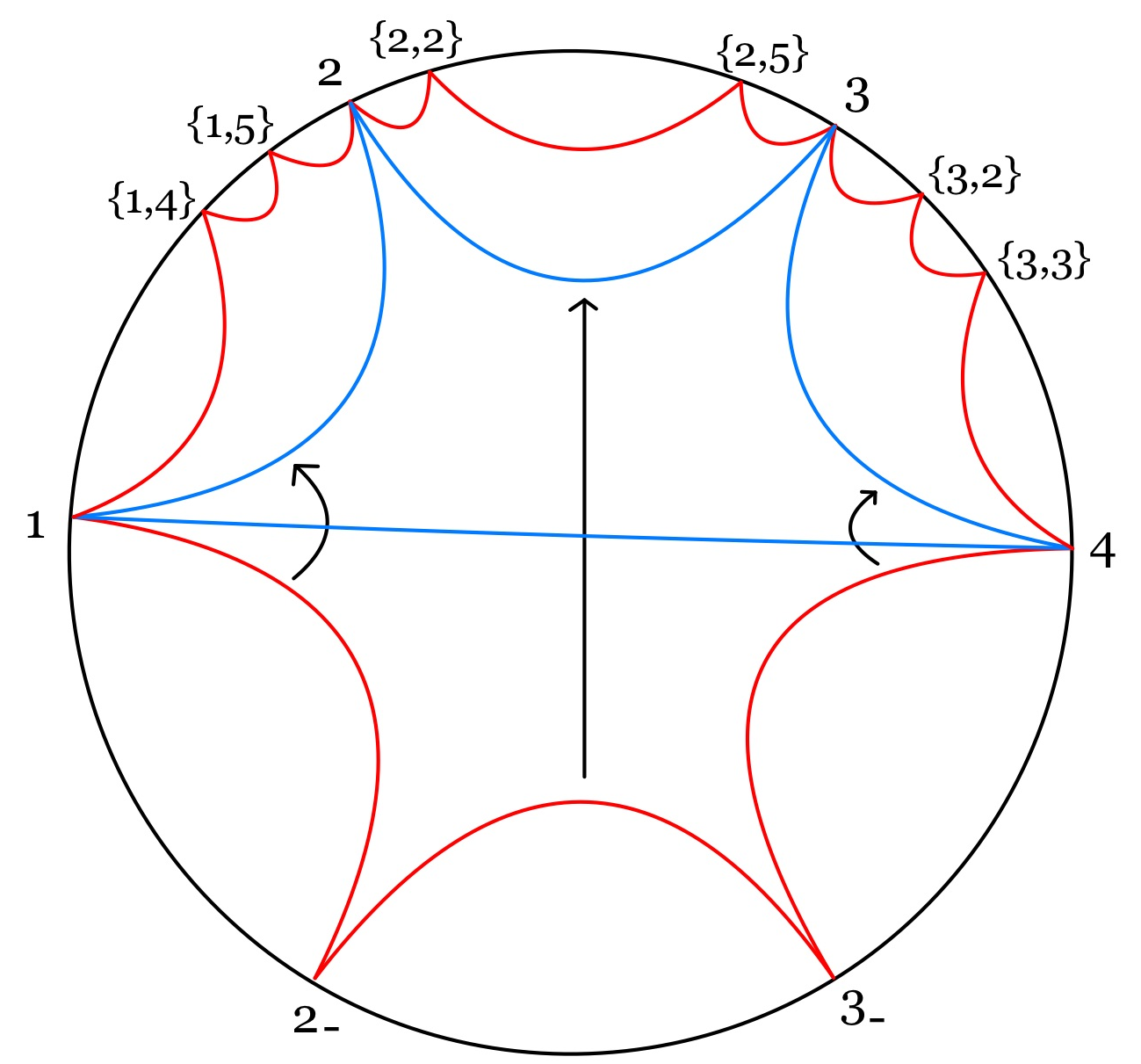}};

\node at (4.4,2.9) {\begin{small}$g_3$\end{small}};
\node at (3.5,3) {\begin{small}$g_2$\end{small}};
\node at (2.1,2.9) {\begin{small}$g_1$\end{small}};
\node at (1.25,3.6) {\begin{small}$g_1^{-1}$\end{small}};
\node at (3.5,4.8) {\begin{small}$g_2^{-1}$\end{small}};
\node at (5.12,3.38) {\begin{small}$g_3^{-1}$\end{small}};
\node at (10.94,2.75) {\begin{small}$g_3$\end{small}};
\node at (10.1,3) {\begin{small}$g_2$\end{small}};
\node at (8.84,2.78) {\begin{small}$g_1$\end{small}};
\end{tikzpicture}
\caption{Fundamental domains for $A_{\Gamma_0, \mathrm{aux}}$ and $A_{\Gamma_0, \mathrm{hBS}}$: $4$ punctures}
\label{cfmgen}
\end{figure}

We shall introduce below an auxiliary map $A_{\Gamma_0, \mathrm{aux}}$ in terms of its pieces. The main purpose of $A_{\Gamma_0, \mathrm{aux}}$ is to lead up to the
notion of a \emph{higher Bowen-Series map} in Definition \ref{higher_bs_def} below.
However, $A_{\Gamma_0, \mathrm{aux}}$ is more symmetrically defined and makes the
book-keeping easier.
The notation $\arc{ij}$ will denote an arc of $\bS^1$ with its endpoints at the break-points $i, j$ such that there are no other break-points of $A_{\Gamma_0, \mathrm{aux}}$ in the arc.

\noindent	$\bullet$ On the arc $\arc{\ i_- (i+1)_-\ }$, define $A_{\Gamma_0, \mathrm{aux}}$ to be $g_i$ for $i=1, \cdots, k-1$. Then $A_{\Gamma_0, \mathrm{aux}}(\arc{\ i_- (i+1)_-\ })$ equals the complement of (the interior of) the arc $\arc{\ i (i+1)\ }$ in $\bS^1$.
	
\noindent	$\bullet$ For every $i=1, \cdots, k-1$, and on each of the $k-1$ short arcs $\arc{\ \{i,j\} \{i,j+1\}\ }$ for $i \leq j \leq i+k-2$ in between $i, i+1$, define $A_{\Gamma_0, \mathrm{aux}}$ to be $g_i^{-1}$. Then $A_{\Gamma_0, \mathrm{aux}}(\cup_{j=i}^{i+k-2}\arc{\ \{i,j\} \{i,j+1\} })$ equals the upper semi-semicircle between $1$ and $k$. (Here, for notational convenience, we identify $\{i,1\}$ with $i$ and $\{i, i+2k-2\}$ with $i+1$.)
	Also, for $i \leq j \leq i+k-2$, $A_{\Gamma_0, \mathrm{aux}}$ maps the clockwise arc from $\{i,j\}$ to $\{i,j+1\}$ onto the clockwise arc from $j$ to $j+1$. We refer to the clockwise arcs from $\{i,j\}$ to $\{i,j+1\}$ (for $i \leq j \leq i+k-2$) as \emph{short folding arcs under $A_{\Gamma_0, \mathrm{aux}}$}.
	
\noindent	$\bullet$ For $i\in\{2,\cdots, k-1\}$ and $1\leq j\leq i-1$, let $j=i-s$, so that
	$1\leq s \leq i-1$. We  define $A_{\Gamma_0, \mathrm{aux}}$ to be $g_s \circ g_i^{-1}$ on $\arc{\ \{i,j\} \{i,j+1\}\ }$. Thus, for $j\leq i-1$, $A(\arc{\{i,j\} \{i,j+1\}})$ equals the counterclockwise (long) arc from $s$ to $s+1$.

\noindent	$\bullet$ For $i\in\{1,\cdots, k-2\}$ and $i+k-1 \leq j\leq 2k-3$, let $j=i+k-1+t$, so that
	$0\leq t \leq k-2-i$.
	We  define $A_{\Gamma_0, \mathrm{aux}}$ to be $g_{k-1-t} \circ g_i^{-1}$ on $\arc{\ \{i,j\} \{i,j+1\}\ }$.
	Thus, for $i+k-1 \leq j\leq 2k-3$, $A(\arc{ \{i,j\} \{i,j+1\}} )$ equals the counterclockwise (long) arc from $k-1-t$ to $k-t$. 
	
	We refer to the clockwise arcs from $\{i,j\}$ to $\{i,j+1\}$ (for $j\leq i-1$ or $ 
	i+k-1 \leq j$) as \emph{long folding arcs under $A$}. 
	
\noindent	$\bullet$ Note that $A_{\Gamma_0, \mathrm{aux}}(i)=i$ for all $i=1,\cdots, k$.

 Define $A_{\Gamma_0, \mathrm{hBS}}$ to be the minimal \pwfm\ map equaling $A_{\Gamma_0, \mathrm{aux}}$ on $\bS^1$, i.e.,  $A_{\Gamma_0, \mathrm{hBS}}$ is obtained from $A_{\Gamma_0, \mathrm{aux}}$ by removing superfluous break-points. Denote the canonical extension of $A_{\Gamma_0, \mathrm{hBS}}$ by $\widehat{A}_{\Gamma_0, \mathrm{hBS}}$, its canonical domain of definition in $\overline{\D}$ by $\mathcal{D}_{\Gamma_0, \mathrm{hBS}}$, and the fundamental domain of $\widehat{A}_{\Gamma_0, \mathrm{hBS}}$ by $R_{\Gamma_0, \mathrm{hBS}}$. Further, let $D$ be the open ideal polygon bounded by the bi-infinite geodesics $\overline{12}, \overline{23},\cdots, \overline{(k-1) k}, \overline{k 1}$. Evidently, all the edges of $D$ are (non-intersecting) diagonals of $R_{\Gamma_0, \mathrm{hBS}}$, each ideal vertex of $D$ is fixed by $A_{\Gamma_0, \mathrm{hBS}}$, each edge of $R_{\Gamma_0, \mathrm{hBS}}$ is mapped by $\widehat{A}_{\Gamma_0, \mathrm{hBS}}$ to an edge of $D$, and $\widehat{A}_{\Gamma_0, \mathrm{hBS}}$ has no diagonal folds. Therefore, $\widehat{A}_{\Gamma_0, \mathrm{hBS}}$ is a \hdm having $D$ as its inner domain. 

\begin{defn}\label{higher_bs_def}
	We call the \pwfm map $A_{\Gamma_0, \mathrm{hBS}}$ the \emph{higher Bowen-Series map} of $\Gamma_0$ (associated with the fundamental domain $W$). 
\end{defn}

The next result is about the relationship between Bowen-Series maps and higher Bowen-Series maps. In fact, the following proposition is a restatement of the above construction and gives a  more direct description of the higher Bowen-Series map of $\Gamma$ in terms of the Bowen-Series maps of $\Gamma_0$ associated with various overlapping fundamental domains. This point of view will be useful in Proposition~\ref{higher_bs_char_prop}, where we will prove a combinatorial characterization of higher Bowen-Series maps. 

\begin{prop}\label{hbs_alternative_prop}
Let $W$ be a (closed) fundamental domain for a Fuchsian group $\Gamma_0$ (uniformizing a $k$-times punctured sphere) given by an ideal $(2k-2)$-gon. We label the ideal vertices of $W$ as $1=1_-, 2_-,\cdots, (k-1)_-, k_-=k, k-1, \cdots, 2$ in counterclockwise order, and assume that the side pairing transformations of $W$ (generating $\Gamma_0$) are given by $g_1, \cdots, g_{k-1} $, where $g_i$ takes the edge $\overline{i_- (i+1)_-}$ to the  edge  $\overline{i (i+1)}$.
	
	Further, let $D$ be the interior of the ideal polygon bounded by the bi-infinite geodesics $\overline{12}$, $\overline{23}$, $\cdots$, $\overline{(k-1) k}$, $\overline{k 1}$, and $P$ the interior of the ideal polygon bounded by the bi-infinite geodesics $\overline{1_- 2_-}$, $\overline{2_- 3_-}$, $\cdots$, $\overline{(k-1)_- k_-}$, $\overline{k_- 1_-}$. Then the following hold.\\
1) $W=\overline{D}\cup\overline{P}$, and for each $j\in\{1,\cdots, k-1\}$, $\overline{D}\cup \overline{g_j(P)}$ is a (closed) fundamental domain for $\Gamma_0$.\\
2) On the clockwise arc from $j$ to $j+1$, the higher Bowen-Series map $A_{\Gamma_0, \mathrm{hBS}}$ equals the Bowen-Series map of $\Gamma_0$ associated with the (closed) fundamental domain $\overline{D}\cup \overline{g_j(P)}$ ($j\in\{1,\cdots, k-1\}$), and on the counterclockwise arc from $1$ to $k$, $A_{\Gamma_0, \mathrm{hBS}}$ equals the Bowen-Series map of $\Gamma_0$ associated with the fundamental domain $W=\overline{D}\cup \overline{P}$.
		
	 Conversely, a map $A:\bS^1\to\bS^1$ defined as in part 2 of the proposition is a higher Bowen-Series map.
\end{prop}

	The preceding description of $A_{\Gamma_0, \mathrm{hBS}}$ shows that it is made up of Bowen-Series maps corresponding to various (overlapping) fundamental domains of $\Gamma_0$. This justifies the terminology `higher Bowen-Series maps'.

\begin{prop}\label{prop-cfdhd-orbeq-gen}
Let $A_{\Gamma_0, \mathrm{hBS}}$ be as above. Then $A_{\Gamma_0, \mathrm{hBS}}$ is a $C^1$ mateable map whose pieces generate the Fuchsian group $\Gamma_0$; in particular, $A_{\Gamma_0, \mathrm{hBS}}$ is orbit equivalent to $\Gamma_0$.
\end{prop}

\begin{proof} 
Clearly, $A_{\Gamma, \mathrm{hBS}}: \bS^1 \to \bS^1$ is piecewise Fuchsian. Since each ideal vertex of the inner domain $D$ is a parabolic fixed point of $A_{\Gamma, \mathrm{hBS}}$, one readily checks that $A_{\Gamma, \mathrm{hBS}}$ is $C^1$. That the pieces of $A_{\Gamma_0, \mathrm{hBS}}$ form a Markov partition for the map follows from the construction of $A_{\Gamma_0,\mathrm{hBS}}$. Expansiveness of $A_{\Gamma_0,\mathrm{hBS}}$ is a consequence of the facts that it is a $C^1$ covering map, it has only finitely many fixed points, and $\vert A'_{\Gamma_0,\mathrm{hBS}}\vert\geq 1$ on $\mathbb{S}^1$ with equality only at the parabolic fixed points (compare \cite[Lemma~3.7]{LMMN}). This shows that $A_{\Gamma_0,\mathrm{hBS}}$ is a \pwfm map without asymmetrically hyperbolic periodic break-points.

To show that $A_{\Gamma, \mathrm{hBS}}$ is a mateable map, it remains to argue that it is orbit equivalent to $\Gamma_0$. It is enough to  show that $x, y$ lie in the same grand orbit under $A_{\Gamma_0, \mathrm{hBS}}$ in the special  case that
	$y = g_i(x)$ for some $i=1, \cdots, k-1$.  \\
\noindent $ \bullet$
	 If $y$ lies in the complement of the clockwise arc from $i$ to
		$i+1$, then $x$ lies in the counterclockwise arc from $i_-$ to
		$(i+1)_-$ and $y=A_{\Gamma_0, \mathrm{hBS}}(x)$.\\
\noindent $ \bullet$ If $y$ lies in any of the short folding arcs; i.e., in
		the clockwise arc from $\{i,i\}$ to
		$\{i,i+k\}$, then we rewrite $y=g_i(x)$ as $g_i^{-1}(y)=x$. Since the piece of $A_{\Gamma_0, \mathrm{hBS}}$ on such an arc is $g_i^{-1}$, we now have that $A_{\Gamma_0, \mathrm{hBS}}(y)=x$.

	It remains to deal with the case where $y$ lies in one of the long folding arcs $\arc{\ \{i,j\} \{i,j+1\}\ }$ for $1\leq j\leq i-1$ (where $i>1$) or $i+k-1 \leq j\leq 2k-3$ (where $i<2k-1$).
	We first consider $1\leq j\leq i-1$, and let $j=i-s$  as in the definition of $A_{\Gamma_0, \mathrm{hBS}}$. Then $x \in \arc{\ s_- (s+1)_-\ }$, the counterclockwise arc from $s_-$ to $ (s+1)_-$. Further, on $\arc{\ s_- (s+1)_-\ }$, $A_{\Gamma_0, \mathrm{hBS}}=g_s$, and on
	$\arc{\ \{i,j\} \{i,j+1\}\ }$, $A_{\Gamma_0, \mathrm{hBS}}=g_s \circ g_i^{-1}$. Thus, $$A_{\Gamma_0, \mathrm{hBS}}(y) = 
	g_s \circ g_i^{-1} \circ g_i (x) = g_s(x) = A_{\Gamma_0, \mathrm{hBS}}(x), $$ and hence $x, y$
	lie in the same grand orbit.
	
	Finally, let $i+k-1 \leq j\leq 2k-3$, and let $j=k-1+t$. Then $x \in \arc{\ (k-t-1)_- (k-t)_-\ }$, the counterclockwise arc from $(k-t-1)_-$ to $(k-t)_-$.
	On $\arc{\ (k-t-1)_- (k-t)_-\ }$, $A_{\Gamma_0, \mathrm{hBS}}=g_{k-t-1}$. Also, on 
	$\arc{\ \{i,j\} \{i,j+1\}\ }$, $A_{\Gamma_0, \mathrm{hBS}}=g_{k-t-1} \circ g_i^{-1}$. Hence 
	$$A_{\Gamma_0, \mathrm{hBS}}(y) = 
	g_{k-t-1} \circ g_i^{-1} \circ g_i (x) = g_{k-t-1}(x) = A_{\Gamma_0, \mathrm{hBS}}(x), $$ and  $x, y$
	lie in the same grand orbit.
\end{proof}

\noindent {\bf Degree of $A_{\Gamma, \mathrm{hBS}}$:}\\ We first observe that the polygonal degree of $A_{\Gamma_0, \mathrm{hBS}}$ is $k-1$. This follows from the fact that the number of components of $A_{\Gamma_0, \mathrm{hBS}}^{-1}(\overline{1k})$, the pre-image of the diameter of $W$, equals $k-1$. 

We now compute the degree of $A_{\Gamma_0, \mathrm{aux}}: \bS^1 \to \bS^1$; this is equal to the degree of $A_{\Gamma_0,\mathrm{hBS}}$ on $\bS^1$. There are $(k-1)(k-2)$ long folding arcs in the upper semi-circle (see Figure \ref{cfmgen} or
Proposition \ref{prop-cfdhd-orbeq-gen}). Each, after being mapped forward by $A_{\Gamma_0, \mathrm{aux}}$, covers the lower semi-circle once. Each of the short
folding arcs in the upper semi-circle cover arcs only in the upper semi-circle. Finally, each of the $k-1$ arcs $\arc{\ i_- (i+1)_-\ }$ in the
lower semi-circle, after being mapped forward by $A_{\Gamma_0, \mathrm{aux}}$, covers the lower semi-circle once. Hence, the degree of $A_{\Gamma_0, \mathrm{aux}}$ is given by
\begin{equation}\label{eq-degcfm}
	deg (A_{\Gamma_0, \mathrm{aux}}) = (k-1)(k-2) + 0 + k-1 = (k-1)^2 = (\chi-1)^2,
\end{equation}
where $\chi = 2-k$ is the Euler characteristic of $S_{0,k}$.\\

\subsection{Mating higher Bowen-Series maps of Bers slice groups with polynomials}\label{hbs_mate_subsec}
Consider the Fuchsian group $\Gamma_0$ and the higher Bowen-Series map $A_{\Gamma_0,\mathrm{hBS}}:\bS^1\rightarrow\bS^1$ introduced in Subsection~\ref{hbs_fuch_subsec}. Note that $A_{\Gamma_0,\mathrm{hBS}}$ is a mateable map associated to $\Gamma_0$. 
For $\Gamma\in\mathcal{B}(\Gamma_0)$, we call the mateable map $A_\Gamma$ associated to $\Gamma$ and compatible with $A_{\Gamma_0,\mathrm{hBS}}$ the \emph{higher Bowen-Series map} of $\Gamma$ (see Subsection~\ref{mateable_gen_subsec}).
Recall that $\mathcal{H}_d$ stands for the principal hyperbolic component in the space of degree $d$ polynomials.

\begin{theorem}\label{thm-cfdmateable}
	Let $\Gamma\in\mathcal{B}(\Gamma_0)$, and $P\in\mathcal{H}_{(k-1)^2}$. Then, $\widehat{A}_{\Gamma, \mathrm{hBS}}:\mathcal{D}_{A_\Gamma, \mathrm{hBS}}\to K(\Gamma)$ and $P:\mathcal{K}(P)\to\mathcal{K}(P)$ are conformally mateable. 
\end{theorem}

\begin{proof}
This can be seen by applying the proof of Theorem~\ref{moduli_interior_mating_thm} mutatis mutandis to the present setting.
\end{proof}

\noindent {\bf A non-example:}
The following description of the higher Bowen-Series map $A_{\Gamma_0, \mathrm{hBS}}$ on $\bS^1$ is straightforward to check from its construction:

\begin{equation*}
A_{\Gamma_0, \mathrm{hBS}}\ = \ \left\{\begin{array}{ll}
                    A_{\Gamma_0, \mathrm{BS}}, & \mbox{on}\  \displaystyle \left(\bigcup_{i=1}^{k-1} \arc{\ i_-(i+1)_-\ }\right) \cup \left(\bigcup_{i=1}^{k-1}\bigcup_{j=i}^{i+k-2} \arc{\ \{i,j\}\{i,j+1\}\ } \right), \\
                    A_{\Gamma_0, \mathrm{BS}}^{\circ 2}, & \mbox{otherwise},
                                          \end{array}\right. 
\end{equation*}
where $A_{\Gamma_0, \mathrm{BS}}$ denotes the Bowen-Series map of $\Gamma_0$ associated with the fundamental domain $W$.

In fact, agreement of $A_{\Gamma_0, \mathrm{hBS}}$ and $A_{\Gamma_0, \mathrm{BS}}$ on the arcs $\arc{\ \{i,j\}\{i,j+1\}\ }$ ($i\in\{1,\cdots, k-1\}, j\in\{i,\cdots, i+k-2\}$) played an important role in the proof of orbit equivalence of $\Gamma_0$ and $A_{\Gamma_0, \mathrm{hBS}}$ (see Proposition~\ref{prop-cfdhd-orbeq-gen}). In this subsection, we shall show that if one replaces $A_{\Gamma_0, \mathrm{BS}}$ by $A_{\Gamma_0, \mathrm{BS}}^{\circ 2}$ on these arcs as well, the resulting minimal \pwfm map 
\begin{equation*}
B := \ \left\{\begin{array}{ll}
                    A_{\Gamma_0, \mathrm{BS}} & \mbox{on}\  \bS^1\cap\{z: \mathrm{Im}(z)\leq 0\}, \\
                    A_{\Gamma_0, \mathrm{BS}}^{\circ 2} & \mbox{on}\  \bS^1\cap\{z: \mathrm{Im}(z)\geq 0\},
                                          \end{array}\right. 
\end{equation*}
is \emph{not} orbit equivalent to $\Gamma_0$. Although the definition of the map $B$ may seem unmotivated at this point, it will naturally crop up in the proof of Theorem~\ref{thm-char-hdm-eoe}, where lack of orbit equivalence between $\Gamma_0$ and $B$ will be essentially used.

\begin{prop}\label{not_oe_prop}
The map $B:\bS^1\to\bS^1$ is not orbit equivalent to $\Gamma_0$.
\end{prop}
\begin{proof}
By way of contradiction, let us assume that $B$ and $\Gamma_0$ are orbit equivalent; i.e., $\Gamma_0.x=\mathrm{GO}_B(x)$, for all $x\in\bS^1$. Since $\Gamma_0$ and $A_{\Gamma_0, \mathrm{BS}}$ are orbit equivalent by Proposition~\ref{b_s_poly_conjugate_prop_1}, it follows that $\mathrm{GO}_{A_{\Gamma_0,\mathrm{BS}}}(x)=\mathrm{GO}_B(x)$, for all $x\in\bS^1$.

The definition of $B$ implies that it acts as $g_2^{-1}\circ g_1^{-1}$ on $\arc{\ \{1,2\}\{1,3\}\ }$. Moreover, since $g_2^{-1}\circ g_1^{-1}$ maps the arc $\arc{\ \{1,2\}\{1,3\}\ }$ to the complement of the interior of the arc $\arc{\ 2_- 3_-\ }$, it follows that $g_2^{-1}\circ g_1^{-1}$, and hence $B$, has a fixed point $p$ in $\arc{\ \{1,2\}\{1,3\}\ }$ (compare Figure~\ref{cfmgen}). Also note that $q:=A_{\Gamma_0, \mathrm{BS}}(p)=g_1^{-1}(p)\in\arc{\ 2 3\ }$ belongs to the $A_{\Gamma_0, \mathrm{BS}}$-grand orbit of $p$. On the other hand, $q$ is also a fixed point of $B$: 
$$
B(q)= A_{\Gamma_0, \mathrm{BS}}^{\circ 2}(A_{\Gamma_0, \mathrm{BS}}(p))=A_{\Gamma_0, \mathrm{BS}}(A_{\Gamma_0, \mathrm{BS}}^{\circ 2}(p))=A_{\Gamma_0, \mathrm{BS}}(B(p))=A_{\Gamma_0, \mathrm{BS}}(p)=q.
$$ 
Thus, $p$ and $q$ cannot lie in the same grand orbit under $B$. Hence, $\mathrm{GO}_{A_{\Gamma_0,\mathrm{BS}}}(p)\neq\mathrm{GO}_B(p)$, a contradiction.
\end{proof}

\section{Combinatorial characterization of Bowen-Series maps and higher Bowen-Series maps}\label{sec-combchar}

\subsection{Combinatorial characterization of Bowen-Series maps}\label{bs_char_sec}
The canonical extension of a Bowen-Series map of a Fuchsian group uniformizing a sphere with finitely many punctures and zero/one/two order two orbifold points
(as discussed in Sections \ref{b_s_punc_sphere_subsec} and \ref{b_s_punc_sphere_orbifold_subsec}) restricts to a self-homeomorphism of the boundary of its domain of definition. We will now see that this property characterizes Bowen-Series maps (of Fuchsian groups uniformizing spheres with finitely many punctures and zero/one/two order two orbifold points) among all \pwfm maps.

Let us quickly recall some notation. Following Definition~\ref{pwm_def}, we denote the pieces of a piecewise M{\"o}bius map $A:\bS^1\to\bS^1$ by $g_1,\cdots, g_k\in\textrm{Aut}({\disk})$ such that $A\vert_{I_j}=g_j$, where $\{I_1,\cdots, I_k\}$ is a partition of $\bS^1$ by closed arcs (i.e., $\mathbb{S}^1=\bigcup_{j=1}^k I_j$, and $\Int{I_m}\cap\Int{I_n}=\emptyset$ for $m\neq n$). We set $\Gamma_A=\langle g_1,\cdots, g_k\rangle$. The canonical domain of definition of $\widehat{A}$ in $\overline{\D}$ is denoted by $\mathcal{D}$. Finally, we set $R=\disk\setminus\mathcal{D}$ to be the fundamental domain of  $\widehat{A}$, and denote the set of all ideal vertices of $R$ (in $\mathbb{S}^1$) by $S$.

\begin{prop}\label{bs_char_1}
	Let $A:\mathbb{S}^1\to\mathbb{S}^1$ be a \pwfm map. Then the following are equivalent.
	\begin{enumerate}\upshape
		\item The canonical extension $\widehat{A}$ maps the boundary of its (canonical) domain of definition onto itself; i.e., $\widehat{A}(\partial R)=\partial R$. 
		\item $\disk/\Gamma_A$ is a sphere with finitely many punctures (possibly with one/two order two orbifold points), and $A$ is a Bowen-Series map of $\Gamma_A$. In particular, $A$ is orbit equivalent to $\Gamma_A$.
	\end{enumerate}
\end{prop}

Before proceeding with the proof of Proposition \ref{bs_char_1}, we provide a concrete example.

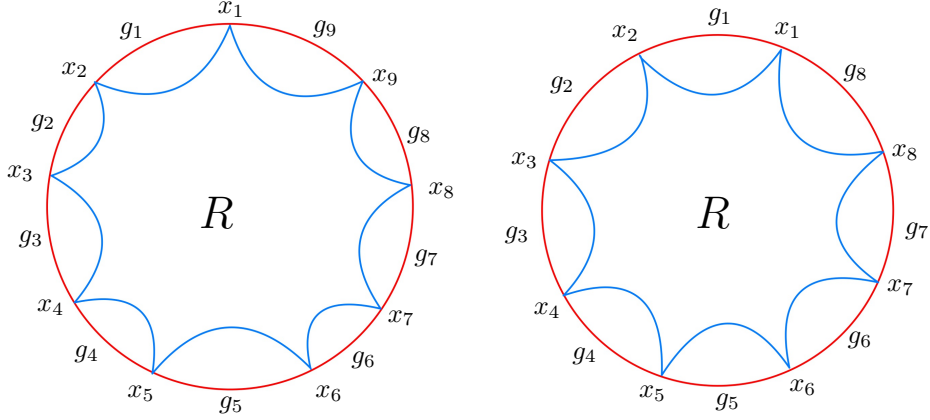
\begin{figure}[ht!]
	\begin{tikzpicture}
		\node[anchor=south west,inner sep=0] at (0,0) {\includegraphics[width=0.492\linewidth]{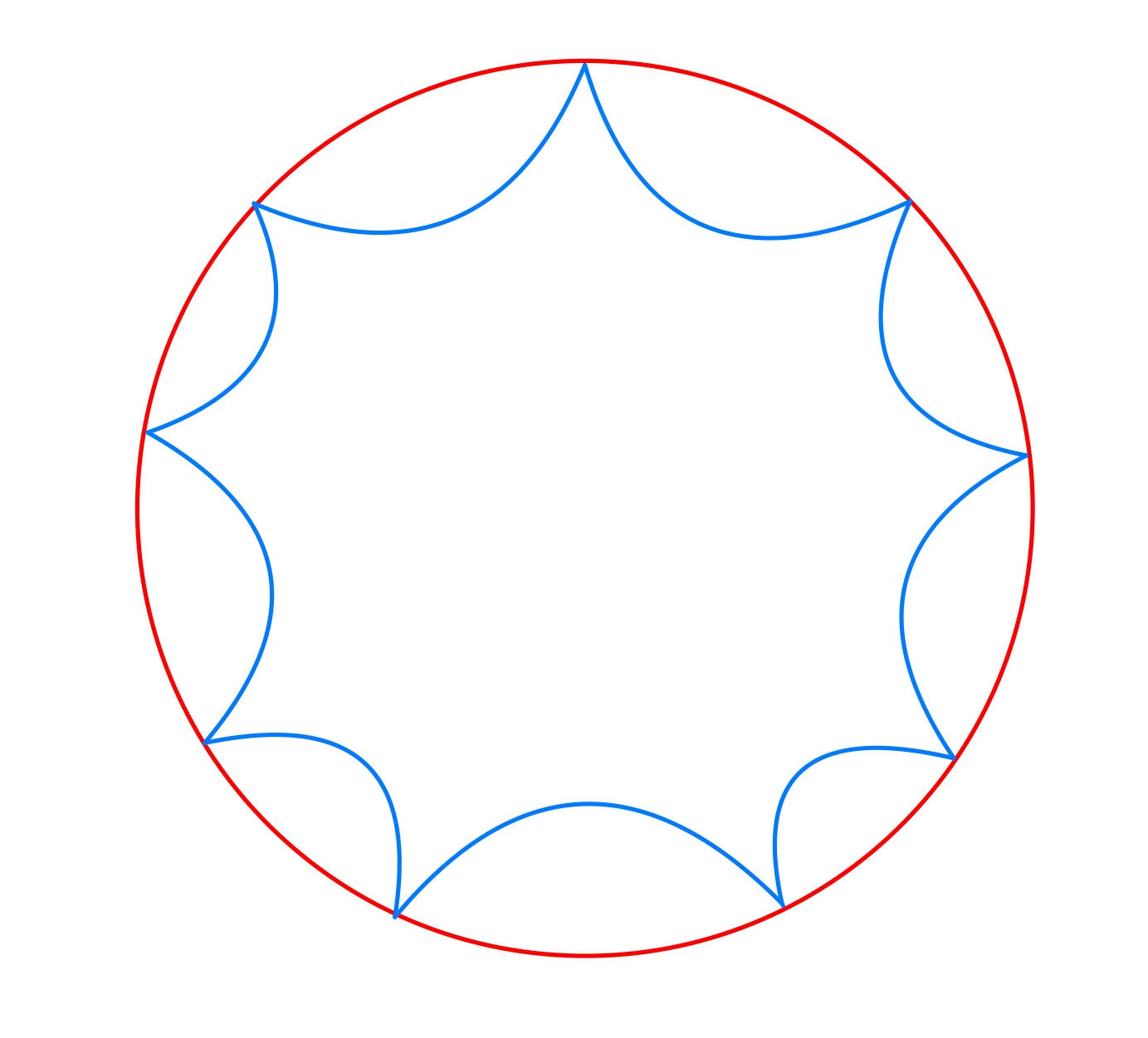}};
		\node[anchor=south west,inner sep=0] at (6.6,0) {\includegraphics[width=0.48\linewidth]{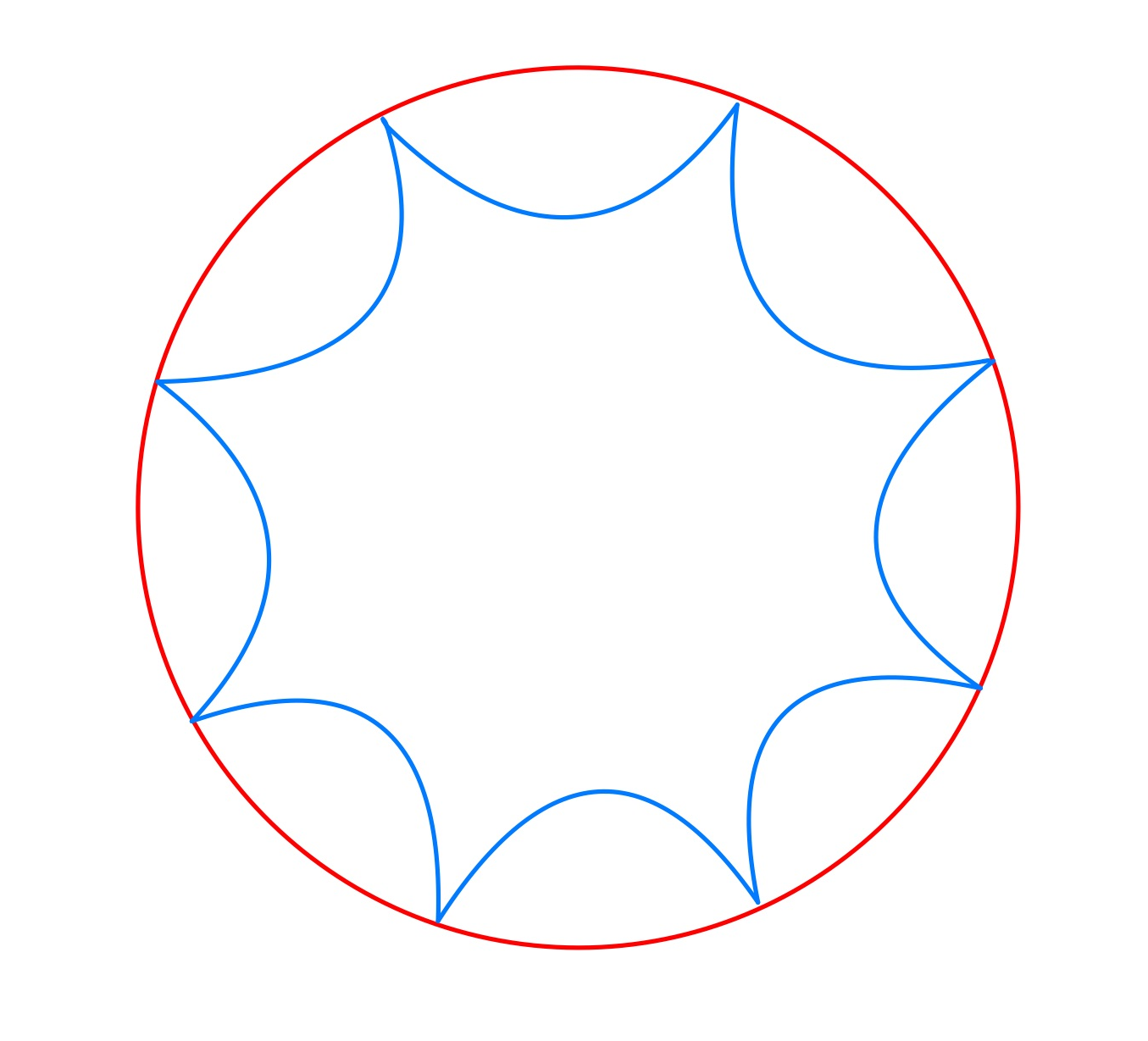}};
		\node at (3,2.8) {\begin{huge}$R$\end{huge}};
		\node at (9.6,2.8) {\begin{huge}$R$\end{huge}};
		\node at (3.2,5.5) {$x_1$}; 
		\node at (1.88,5.22) {$g_1$}; 
		\node at (1.12,4.7) {$x_2$}; 
		\node at (0.7,4) {$g_2$}; 
		\node at (0.4,3.32) {$x_3$}; 
		\node at (0.54,2.45) {$g_3$}; 
		\node at (0.8,1.54) {$x_4$}; 
		\node at (1.28,0.92) {$g_4$}; 
		\node at (2,0.4) {$x_5$}; 
		\node at (3.2,0.24) {$g_5$}; 
		\node at (4.5,0.4) {$x_6$}; 
		\node at (4.96,0.84) {$g_6$}; 
		\node at (5.46,1.4) {$x_7$}; 
		\node at (5.8,2.2) {$g_7$}; 
		\node at (6,3.1) {$x_8$}; 
		\node at (5.7,3.84) {$g_8$}; 
		\node at (5.25,4.6) {$x_9$}; 
		\node at (4.45,5.25) {$g_9$}; 
		
		\node at (10.6,5.22) {$x_1$};
		\node at (9.7,5.4) {$g_1$};
		\node at (8.44,5.18) {$x_2$};
		\node at (7.6,4.5) {$g_2$};
		\node at (7.1,3.5) {$x_3$};
		\node at (7,2.5) {$g_3$};
		\node at (7.4,1.5) {$x_4$};
		\node at (7.9,0.9) {$g_4$};
		\node at (8.8,0.4) {$x_5$};
		\node at (9.8,0.28) {$g_5$};
		\node at (10.8,0.48) {$x_6$};
		\node at (11.56,1.1) {$g_6$};
		\node at (12.1,1.8) {$x_7$};
		\node at (12.32,2.56) {$g_7$};
		\node at (12.16,3.6) {$x_8$};
		\node at (11.5,4.66) {$g_8$};
	\end{tikzpicture}
	\caption{Illustrating  the proof of Proposition~\ref{bs_char_1}.} 
	\label{odd_even_case_fig}
\end{figure}

In the left figure in Figure \ref{odd_even_case_fig}, $\partial R$ has $9$ sides; i.e., $k=9$ and $r=5$. In the right figure in Figure \ref{odd_even_case_fig}, $\partial R$ has $8$ sides; i.e., $k=8$ and $r=4$.

\begin{proof}[Proof of Proposition \ref{bs_char_1}:] 1) $\implies$ 2). We refer the reader to Figure \ref{odd_even_case_fig}.
	The condition $\widehat{A}(\partial R)=\partial R$ implies that $\widehat{A}\vert_{\partial R}$ is a self-homeomorphism preserving the set $S$.
	(Recall that $S$ is the set of all ideal vertices of $R$ in $\mathbb{S}^1$.) In particular, $\widehat{A}\vert_{S}$ is conjugate to the action of a symmetry $\phi$ of a regular $k$-gon restricted to its vertices.
	
	First note that if $\phi$ is a rotational symmetry of a regular $k$-gon, then there exists some $l\in\{1,\cdots, k\}$ such that $A(x_j)=x_{j+l-1}$, for each $j\in\{1,\cdots, k\}$. But this will force $A$ to be a homeomorphism of $\mathbb{S}^1$ contradicting the assumption that $A\vert_{\mathbb{S}^1}$ has degree at least two. Therefore, $\phi$ must be a reflection symmetry of a regular $k$-gon.
	
	We now need to consider two cases.
	\smallskip

	\noindent\textbf{Case 1: $k=2r-1$, for some integer $r\geq 2$.} For a regular polygon with odd number of sides, the axis of the reflection symmetry $\phi$ connects the midpoint of one side to the opposite vertex. Thus, possibly after renumbering the ideal vertices $x_1,\cdots, x_k$ of $R$, we can assume that the ``axis of symmetry'' of $\widehat{A}\vert_{\partial R}$ connects $x_1$ to the geodesic of $\disk$ connecting $x_r$ to $x_{r+1}$. Hence, $A(x_1)=x_1$, and $\{x_{1+i}, x_{1-i}\}$ is a $2$-cycle of $A$, for $i\in\{1,\cdots, r-1\}$.
	
	In particular, $A$ maps the geodesic connecting $x_{r}$ to $x_{r+1}$ onto itself, and switches it endpoints. Thus $g_{r}$ must have a fixed point on this geodesic. It follows that $g_{r}^{\circ 2}$ has three fixed points, and hence must equal the identity map. In other words, $g_{r}$ has an elliptic fixed point of order two on the side of $\partial R$ connecting $x_{r}$ to $x_{r+1}$.
	
	Let $\Gamma':=\langle g_{1},\cdots, g_{r}\rangle\leq\Gamma_A$. Note that $R$ is (the interior of) an ideal polygon in $\disk$, and the M{\"o}bius transformations $g_{1},\cdots, g_{r}$ pair the sides of $\partial R$. By the Poincar{\'e} polygon theorem, $R$ is (the interior of) a fundamental polygon for $\Gamma'$. Furthermore, the side pairing patterns show that $\Sigma:=\disk/\Gamma'$ is a sphere with $r$ punctures and an orbifold point of order two. The ideal vertices of $R$ correspond to the punctures of $\Sigma$, so they are parabolic fixed points of elements of $\Gamma'$.
	
	For $i\in\{0,\cdots, r-2\}$, the element $g_{k-i}\circ g_{1+i}\in\Gamma_A$ fixes the parabolic fixed points $x_{1+i}$ and $x_{2+i}$. Hence, these  maps must equal the identity map; i.e., $g_{k-i}=g_{1+i}^{-1}$, for $i\in\{0,\cdots, r-2\}$. Since $g_1,\cdots, g_k$ generate the group $\Gamma_A$, we now conclude that $\Gamma_A=\Gamma'$, and $A$ is the Bowen-Series map of $\Gamma_A$ associated to a fundamental polygon having $R$ as its interior.
	\bigskip

	\noindent\textbf{Case 2: $k=2r$, for some $r\geq 2$.} For a regular polygon with even number of sides, the axis of the reflection symmetry $\phi$ either connects the midpoints of opposite sides or connects opposite vertices. Thus, possibly after renumbering the ideal vertices $x_1,\cdots, x_k$ of $R$, we can assume that the ``axis of symmetry'' of $\widehat{A}\vert_{\partial R}$ either connects $x_1$ to $x_{r+1}$, or connects the geodesic of $\disk$ having endpoints at $x_1, x_2$ to the geodesic having endpoints at $x_{r+1}, x_{r+2}$. 
	
	In the former case, we have that $A(x_1)=x_1$, $A(x_{r+1})=x_{r+1}$, and $\{x_{1+i}, x_{1-i}\}$ is a $2$-cycle of $A$, for $i\in\{1,\cdots, r-1\}$. Again, setting $\Gamma':=\langle g_{1},\cdots, g_{r}\rangle\leq\Gamma_A$, we see in light of the Poincar{\'e} polygon theorem that $R$ is (the interior of) a fundamental polygon for $\Gamma'$ with side pairing transformations $g_1,\cdots, g_r$, and $\Sigma:=\disk/\Gamma'$ is a sphere with $r+1$ punctures. As the ideal vertices of $R$ correspond to the punctures of $\Sigma$, they are parabolic fixed points of elements of $\Gamma'$. For $i\in\{0,\cdots, r-1\}$, the element $g_{k-i}\circ g_{1+i}\in\Gamma_A$ fixes the parabolic fixed points $x_{1+i}$ and $x_{2+i}$ of elements of $\Gamma_A$. By discreteness of $\Gamma_A$, these maps must equal the identity map; i.e., $g_{k-i}=g_{1+i}^{-1}$, for $i\in\{0,\cdots, r-1\}$. Therefore, $\Gamma_A=\Gamma'$, and $A$ is the Bowen-Series map of $\Gamma_A$ associated to a fundamental polygon having $R$ as its interior.
	
	In the latter case, $\{x_{2+i}, x_{1-i}\}$ is a $2$-cycle of $A$ for $i\in\{0,\cdots, r-1\}$. Combining the arguments used above, one easily sees that in this case, $\Gamma_A=\langle g_{1},\cdots, g_{r+1}\rangle$, and $\Sigma:=\disk/\Gamma_A$ is a sphere with $r$ punctures and two orbifold points of order two. Moreover, $A$ is the Bowen-Series map of $\Gamma_A$ associated to a fundamental polygon having $R$ as its interior.
	
	2) $\implies$ 1). This directly follows from the definition of Bowen-Series maps.
\end{proof}

\subsection{Combinatorial characterization of higher Bowen-Series maps}\label{higher_bs_char_sec}

In this subsection, we will employ the notation introduced in Definition~\ref{def-nofold} of a \hdm.

In Section~\ref{sec-fold}, we defined and studied the structure of a specific \hdm, which we called a higher Bowen-Series map (see Definition~\ref{higher_bs_def}). The canonical extension of a higher Bowen-Series map has the additional property of being injective on each connected component of $\partial R\setminus \partial D$ (here we consider the boundary in the closed disk $\overline{\D}$). As we shall see next, this property characterizes higher Bowen-Series maps among all higher degree maps without folding. Recall that for a piecewise Fuchsian map $A$, the group $\Gamma_A$ is the Fuchsian group generated by the pieces of $A$.

\begin{prop}\label{higher_bs_char_prop}
	Let $A:\bS^1\to\bS^1$ be a \hdm. Then, the following are equivalent.
	\begin{enumerate}\upshape
		\item The canonical extension $\widehat{A}$ is injective on each connected component of $\partial R\setminus \partial D$ (boundary taken in $\overline{\D}$). 
		
		\item $\Gamma_A$ is a punctured sphere Fuchsian group, and $A$ is a higher Bowen-Series map of $\Gamma_A$. In particular, $A$ is orbit equivalent to $\Gamma_A$.
	\end{enumerate}
\end{prop}
\begin{proof}
	1) $\implies$ 2). We first note that thanks to Proposition~\ref{prop-cfdhd-orbeq-gen}, orbit equivalence of $A$ and $\Gamma_A$ will follow once we prove that $A$ is a higher Bowen-Series map.
	
	We label the edges $\alpha_1, \cdots, \alpha_l$ of $R$ and the edges $\delta_1, \cdots, \delta_k$ of $D$ (in counterclockwise order) such that $\delta_1$ and $\alpha_1$ have a common endpoint, and denote the vertices of $D$ by $p_1,\cdots, p_k$ (in counterclockwise order) where $\alpha_1$ and $\delta_1$ meet at $p_1$ (see Figure~\ref{hdm_fig}). We further assume that $\widehat{A}\vert_{\alpha_r}\equiv g_r$, for $r\in\{1,\cdots, l\}$. By definition, under a \hdm $A$, consecutive edges $\alpha_r, \alpha_{r+1}$ of $R$ go to consecutive edges of $D$. Thus, we have a continuous map $\hA: \partial R \to \partial D$.
	
	Note that both the endpoints of $\delta_1$ are fixed by $A$. The assumption that $\widehat{A}$ is injective on each connected component of $\partial R\setminus \partial D$ now implies that $\hA: \partial R \to \partial D$ is orientation-reversing, and the connected component of $\partial R\setminus \partial D$ containing $\alpha_1$ consists of exactly $k-1$ edges. (If the map were orientation-preserving, the degree of $A$ would be one, violating the assumption that $A$ has degree greater than one.) Moreover, $\widehat{A}$ maps $\alpha_1,\cdots, \alpha_{k-1}$ bijectively onto the edges $\delta_k,\cdots, \delta_2$ of $D$. Hence, each of the M{\"o}bius maps $g_1,\cdots, g_{k-1}$ carry the edges $\alpha_1,\cdots, \alpha_{k-1}$ of $R$ to the edges $\delta_k,\cdots, \delta_2$ of $D$. Thus, the M{\"o}bius maps $g_{1}^{\pm 1},\cdots, g_{k-1}^{\pm}$ pair the sides of the (closed) ideal polygon $W_1$ in $\disk$ bounded by $\alpha_1,\cdots, \alpha_{k-1}$, $\delta_2,\cdots, \delta_k$. Setting
	$\Gamma_1:= \langle g_1,\cdots, g_{k-1}\rangle,$ one sees that $W_1$ is a (closed) fundamental polygon for $\Gamma_1$. Furthermore, the side pairing patterns show that $\disk/\Gamma_1$ is a sphere with $k$ punctures where the vertices of $D$ bijectively correspond to the punctures of $\disk/\Gamma_1$. It also follows that on the counterclockwise arc from $p_1$ to $p_2$, the map $A$ agrees with the Bowen-Series map of $\Gamma_1$ associated with the fundamental polygon $W_1$.
	
	For $j\in\{1,\cdots, k\}$, we set
	$\Gamma_j:= \langle g_{1+(j-1)(k-1)},\cdots, g_{j(k-1)}\rangle.$ Applying the arguments of the previous paragraph repeatedly, one sees that the connected component of $\partial R\setminus \partial D$ containing $\alpha_{1+(j-1)(k-1)}$ consists of exactly $k-1$ edges, namely $\alpha_{1+(j-1)(k-1)},\cdots, \alpha_{j(k-1)}$. We denote by $W_j$ the closed ideal polygon bounded by the edges $\alpha_{1+(j-1)(k-1)},\cdots, \alpha_{j(k-1)}$ of $R$ and all the edges of $D$ except $\delta_j$. Then, $g_{1+(j-1)(k-1)},\cdots, g_{j(k-1)}$ pair the sides of $W_j$, and hence $W_j$ is a (closed) fundamental polygon for $\Gamma_j$. Furthermore, the side pairing patterns show that $\disk/\Gamma_j$ is a sphere with $k$ punctures where the vertices of $D$ bijectively correspond to the punctures of $\disk/\Gamma_j$. It also follows that on the counterclockwise arc from $p_j$ to $p_{j+1}$, the map $A$ agrees with the Bowen-Series map of $\Gamma_j$ associated with the fundamental polygon $W_j$.
	
	\begin{figure}[htb!]
	\begin{tikzpicture}
	\node[anchor=south west,inner sep=0] at (0,0) {\includegraphics[width=0.8\linewidth]{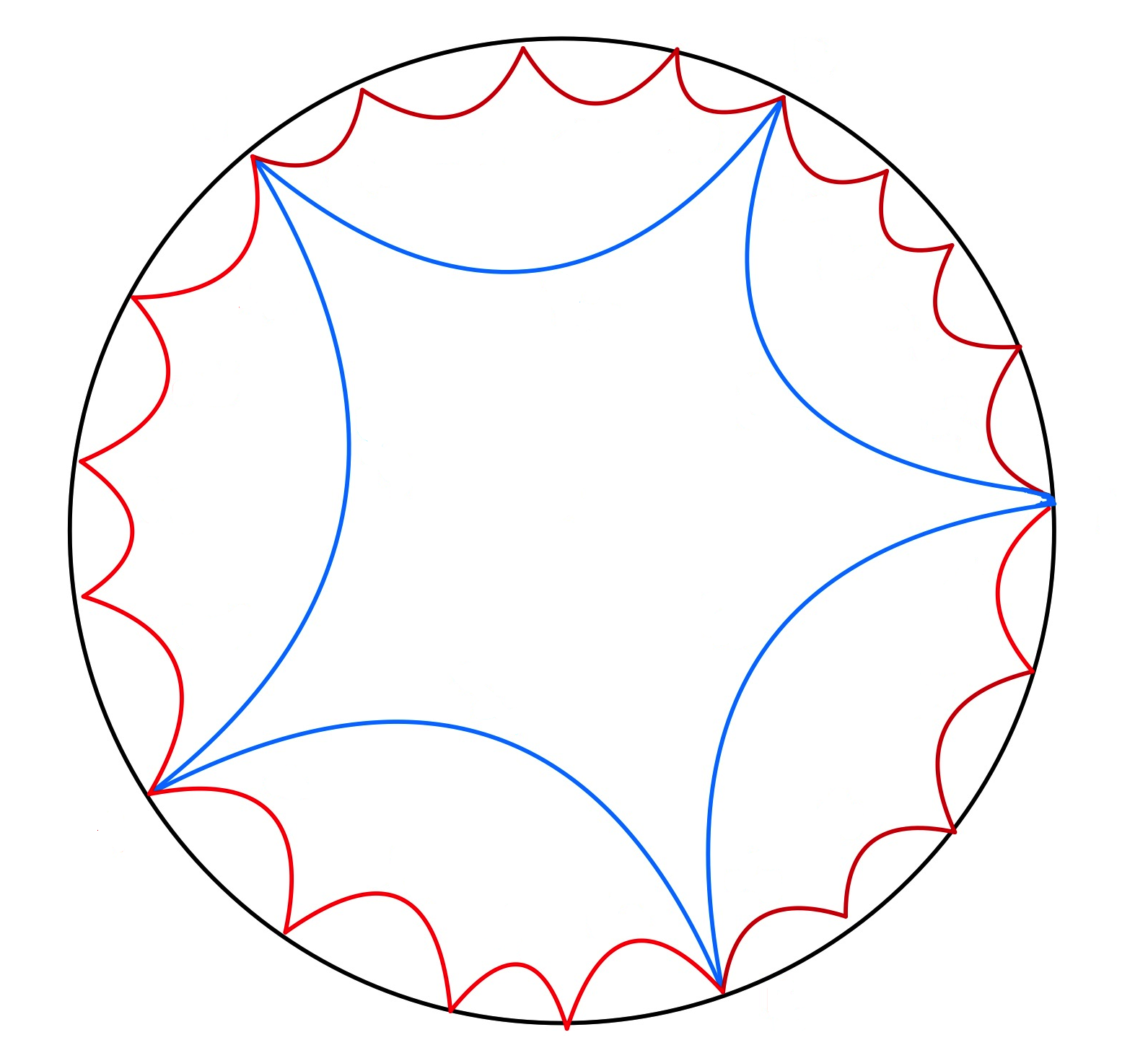}};
	\node at (5,5) {\begin{Huge}$D$\end{Huge}};
	\node at (4.2,3.2) {\begin{Large}$\delta_1$\end{Large}};
	 \node at (4.2,2.2) {\begin{large}$P_1$\end{large}};
	\node at (6.3,3.8) {\begin{Large}$\delta_2$\end{Large}};
         \node at (7,3.2) {\begin{large}$P_2$\end{large}};
        \node at (6.6,6) {\begin{Large}$\delta_3$\end{Large}};
          \node at (7.2,6.2) {\begin{large}$P_3$\end{large}};
         \node at (5,6.6) {\begin{Large}$\delta_4$\end{Large}};
          \node at (4.8,7.28) {\begin{large}$P_4$\end{large}};
        \node at (3.36,5.4) {\begin{Large}$\delta_5$\end{Large}};
        \node at (2.7,5) {\begin{large}$P_5$\end{large}};
        \node at (2.5,2.3) {$\alpha_1$};
        \node at (3.5,1.6) {$\alpha_2$};
        \node at (4.5,1.1) {$\alpha_3$};
        \node at (5.5,1.2) {$\alpha_4$};
        \node at (6.7,1.45) {$\alpha_5$};
        \node at (7.5,2.2) {$\alpha_6$};
        \node at (8.1,2.9) {$\alpha_7$};
        \node at (8.54,4.1) {$\alpha_8$};
        \node at (8.44,5.44) {$\alpha_9$};
        \node at (7.94,6.4) {$\alpha_{10}$};
        \node at (7.6,7) {$\alpha_{11}$};
        \node at (7,7.6) {$\alpha_{12}$};
        \node at (6.3,8.16) {$\alpha_{13}$};
        \node at (5,8.2) {$\alpha_{14}$};
        \node at (4,8) {$\alpha_{15}$};  
        \node at (3,7.7) {$\alpha_{16}$};
        \node at (2.3,6.7) {$\alpha_{17}$};
        \node at (1.8,5.5) {$\alpha_{18}$};
        \node at (1.6,4.36) {$\alpha_{19}$};
        \node at (2,3.36) {$\alpha_{20}$};
       
         	\end{tikzpicture}
		\caption{A \hdm arising from a five times punctured sphere group.}
		\label{hdm_fig}
	\end{figure}
	
	Further, the preceding discussion shows that $l=k(k-1)$. We set 
	$$
	\Gamma\equiv \Gamma_A:=\langle g_1,\cdots, g_{k(k-1)} \rangle.
	$$
	Note that the vertices of $D$ are fixed points of parabolic elements of $\Gamma$, and hence the elements of $\Gamma$ that fix a vertex $p_j$ ($j\in\{1,\cdots, k\}$) of $D$ form an infinite cyclic subgroup of $\Gamma$. Let $h_j$ be a generator of the stabilizer subgroup of $p_j$. 
	\bigskip
	
	\noindent\textbf{Claim 1: $\Gamma=\langle h_1,\cdots, h_k \rangle$.}
	\begin{proof}[Proof of Claim]
		Let us fix $j\in\{1,\cdots, k\}$. As $g_{1+(j-1)(k-1)}$ fixes the vertex $p_j$ of $D$, it is a power of $h_j$. Now observe that both $g_{1+(j-1)(k-1)}$ and $g_{2+(j-1)(k-1)}$ carry a vertex of $R$ to the vertex $p_{j-1}$ of $D$, so $g_{2+(j-1)(k-1)}\circ g_{1+(j-1)(k-1)}^{-1}$ fixes $p_{j-1}$. Therefore, $g_{2+(j-1)(k-1)}\circ g_{1+(j-1)(k-1)}^{-1}$ is a power of $h_{j-1}$. It follows that $g_{2+(j-1)(k-1)}$ lies in $\langle h_{j-1}, h_j \rangle$. Inductively, one sees that $g_r$  lies in $\langle h_1,\cdots, h_k\rangle$, for $r\in\{1,\cdots, k(k-1)\}$. The claim now follows.
	\end{proof}
	
	As $\Gamma$ is generated by $k$ parabolic transformations $h_1,\cdots, h_k$ (whose powers generate a $k$-times punctured sphere group $\Gamma_1\leq\Gamma$), it follows that $\D/\Gamma$ is a $k$-times punctured sphere.
	\bigskip
	
	\noindent\textbf{Claim 2: $\Gamma_j=\Gamma,\ j\in\{1,\cdots, k\}$.} 
	\begin{proof}[Proof of Claim]
		To this end, note that $\Gamma_j$ is generated by the parabolic elements 
		$$
		g_{1+(j-1)(k-1)}, g_{2+(j-1)(k-1)}\circ g_{1+(j-1)(k-1)}^{-1},\cdots, g_{j(k-1)}\circ g_{k-2+(j-1)(k-1)}^{-1}, g_{j(k-1)}^{-1}
		$$ 
		that fix the vertices of $D$. Hence, each such element is a power of some $h_i$. Since $\Gamma_j$ is a lattice, it must be a finite index subgroup of $\Gamma$. Therefore, each of the above parabolic elements must be equal to some $h_i$ or $h_i^{-1}$. This shows that $\Gamma_j=\Gamma$.
	\end{proof}
	
	For $j\in\{1,\cdots, k\}$, let us denote the interior of the ideal polygon bounded by $\delta_j, \alpha_{1+(j-1)(k-1)},\cdots, \alpha_{j(k-1)}$ by $P_j$. Then, $W_j=\overline{D}\cup \overline{P_j}$, $j\in\{1,\cdots, k\}$. To show that $A$ is a higher Bowen-Series map, it now suffices to prove (in the light of Proposition~\ref{hbs_alternative_prop}) that 
	$$
	P_j=g_{k+1-j}(P_1),\ \textrm{for}\ j\in\{2,\cdots, k\}.
	$$ 
	We demonstrate this for $j=2$, the proofs for other values of $j$ are similar.
	
	First observe that applying the arguments of the proof of Claim 2 on $j=1$, we can choose the generators of the stabilizer subgroups of $p_1,\cdots, p_k$ as follows:
	\begin{equation}
		h_1= g_1,\ h_2= g_{k-1}^{-1},\ h_3=g_{k-1}\circ g_{k-2}^{-1},\ \cdots,\ h_{k-1}=g_{3}\circ g_{2}^{-1},\ h_k=g_2\circ g_1^{-1}.
		\label{hdm_explicit_formula_1}
	\end{equation}
	Again, applying the arguments of the proof of Claim 2 on $j=2$, we first see that $g_k$ is equal to $h_2$ or $h_2^{-1}$. But we already know that $g_{k-1}=h_2^{-1}$. The mapping properties of $\widehat{A}$ (more precisely, the fact that $g_{k-1}$ carries $\alpha_{k-1}$ onto $\delta_2$ while $g_k$ carries $\alpha_k$ onto $\delta_1$, see Figure~\ref{hdm_fig}) now imply that $g_k=h_2=g_{k-1}^{-1}$. Next, the element $g_{k+1}\circ g_k^{-1}$ is equal to $h_1$ or $h_1^{-1}$; i.e., $g_{k+1}$ is equal to $g_1\circ g_k$ or $g_1^{-1}\circ g_k$. Observe that $g_{k+1}(\alpha_{k+1})=\delta_k$, while $g_k$ sends $\alpha_{k+1}$ into the region bounded by $\delta_1$ and the counterclockwise arc from $p_1$ to $p_2$. This forces $g_{k+1}$ to be equal to $g_1\circ g_k=g_1\circ g_{k-1}^{-1}$. Continuing this way, one concludes that $g_{k+i}= g_i\circ g_{k-1}^{-1}$, for $i\in\{1,\cdots, k-2\}$. This shows that $g_{k-1}^{-1}(\delta_2)=\alpha_{k-1}$, $g_{k-1}^{-1}(\alpha_k)=\delta_1$, and $g_{k-1}^{-1}(\alpha_{k+i})=\alpha_i$, $i\in\{1,\cdots, k-2\}$. Hence, $P_2=g_{k-1}(P_1)$.
	
	1) $\implies$ 2). This is clear from the definition of higher Bowen-Series maps.
\end{proof}

Propositions \ref{bs_char_1} and \ref{higher_bs_char_prop} give a combinatorial characterization   of Bowen-Series and higher
Bowen-Series maps among piecewise Fuchsian maps of the circle. While
these are not adequate to rule out the existence of other  mateable maps
of the circle, 
 they do suggest a positive answer to 
the following question, a version of which was formulated and explicitly posed to us by the referee (see also Question \ref{qn-moduli} below):

\begin{qn}
Are the (continuous) Bowen-Series and higher
Bowen-Series maps the only mateable maps in the sense of Definition \ref{def-mateable}? 
\end{qn}

\subsection{Higher Bowen-Series as a first return map}\label{hbs_first_return_subsec} In this subsection, we shall show that the higher Bowen-Series map arises naturally out of a piecewise M\"obius map defined on a \emph{disjoint} union of two circles in the complex plane. The corresponding Kleinian surface groups, as we shall show in Section \ref{mating_boundary_groups_subsec}, arise naturally from pinching a special collection of simple closed curves  (see Lemma \ref{inv_lami_pwfm} and Theorem~\ref{first_return_conf_model_thm}) on a punctured sphere, giving rise to a Kleinian surface group with accidental parabolics. For the time being, consider the following simple extension of a Bowen-Series map. We refer the reader to Figure \ref{fund_dom_punctured_sphere_fig}. If we pinch the diameter in the standard representation 
of a Bowen-Series map $A_{G_d}$ associated with the Fuchsian group $G_d$ uniformizing a sphere with $d+1$ punctures (in Figure \ref{fund_dom_punctured_sphere_fig}, the diameter is $\bbar{p_1,p_6}$), we obtain  two circles $ \bS^1_+, \bS^1_-$ attached at a point.
It will be more convenient to regard $ \bS^1_+, \bS^1_-$ as disjoint circles equipped with 
an auxiliary quotient map that identifies the south pole of one to the north pole of the other. The auxiliary quotient map will not play any role in the discussion in this paragraph.
The map $A_{G_d}$
induces a new map $\widetilde{A}_{G_d}$ on $\bS^1_+ \sqcup \bS^1_-$. The fundamental domain $R$ of $A_{G_d}$ gets pinched to $R_+ \sqcup R_-$, where each $R_\pm$ is an ideal polygon with $d$ sides . Further, if $\tilde{g}_1, \cdots, \tilde{g}_{d}$ are the pieces of $\widetilde{A}_{G_d}$ restricted to
$\bS^1_+$, then $(\tilde{g}_1)^{-1}, \cdots, (\tilde{g}_{d})^{-1}$ are the pieces of $\widetilde{A}_{G_d}$ restricted to
$\bS^1_-$. Further, $\widetilde{A}_{G_d}$ maps $\partial R_+$ to $\partial R_-$ and vice versa. As we shall 
see below (Proposition \ref{hbs_return_map_prop} and Corollary \ref{hbs_first_return_cor}), the first return map of $\widetilde{A}_{G_d}$ to $\bS^1_+$ (or $\bS^1_-$)  in this setup gives a higher Bowen-Series map.

\begin{rmk}\label{rmk-nopinchornfld}
	The above discussion also shows that a  higher Bowen-Series map for the orbifold groups in Section \ref{b_s_punc_sphere_orbifold_subsec} involves pinching a geodesic passing through the orbifold point, and hence the elliptic elements degenerate to parabolics and no torsion survives in the pinched group. Thus, pinching the examples in Section \ref{b_s_punc_sphere_orbifold_subsec} will not furnish new examples.
\end{rmk}

We turn now to a more general setup.
Let $A_+, A_-:\bS^1\to\bS^1$ be \pwfm maps. We denote the fundamental domains of $A_+, A_-$ by $R_+, R_-$, and the Fuchsian groups generated by the pieces of $A_+, A_-$ by $\Gamma_+, \Gamma_-$ (respectively). One can naturally associate a fiberwise dynamical system to the above pair of maps; namely
\begin{equation}
\mathbf{A}:\left(\overline{\D}\setminus R_+\right)\times\{+\}\bigsqcup\left(\overline{\D}\setminus R_-\right)\times\{-\}\to\overline{\D}\times\{+,-\},\ (z,\pm)\mapsto (\widehat{A}_\pm(z), \mp).
\end{equation}

We will now show that under mild conditions on the maps $A_+, A_-$, higher Bowen-Series maps emerge as first return maps of $\mathbf{A}$ to each sheet.

\begin{prop}\label{hbs_return_map_prop}
	Suppose that the \pwfm maps $A_+, A_-$ satisfy the following conditions.
	\begin{enumerate}
		\item $\widehat{A}_+$ maps $\partial R_+$ homeomorphically onto $\partial R_-$.
		
		\item $\widehat{A}_-\circ \widehat{A}_+$ preserves each edge of $\partial R_+$.
		
		\item The group generated by $\Gamma_+$ and $\Gamma_-$ is discrete (i.e., Fuchsian).
	\end{enumerate}  
	
	Then, the first return map of $\mathbf{A}$ on $\overline{\D}\times\{+\}$ is a higher Bowen-Series of a punctured sphere Fuchsian group.
\end{prop}
\begin{proof}
	Let $\delta_1,\cdots, \delta_k$ be the sides of $\partial R_+$ (ordered counterclockwise) such that $\widehat{A}_+\vert_{\delta_j}\equiv h_j$, where $h_1,\cdots, h_k$ are the pieces of $A_+$. As $\widehat{A}_+:\partial R_+\to\partial R_-$ is a homeomorphism, $\partial R_-$ also has $k$ edges, and hence $A_-$ has $k$ pieces. 
	
	We denote the edges of $\partial R_-$ by $\delta_{-1},\cdots, \delta_{-k}$ such that $\delta_{-j}=\widehat{A}_+(\delta_j)=h_j(\delta_j)$. Since $\widehat{A}_+\vert_{\partial R_+}$ is a homeomorphism, the edges $\delta_{-1},\cdots, \delta_{-k}$ must be cyclically ordered. Observe that if these edges were ordered counterclockwise, then $A_+$ would be a homeomorphism of the circle (this follows from the fact that M{\"o}bius maps preserve orientation). But this contradicts our hypothesis that $A_+$ is a \pwfm map (recall that a \pwfm map is required to be a circle covering of degree at least two). Therefore, the edges $\delta_{-1},\cdots, \delta_{-k}$ are ordered clockwise.
	
	Let us denote the pieces of $A_-$ by $h_{-1},\cdots, h_{-k}$, where $\widehat{A}_-\vert_{\delta_{-j}}\equiv h_{-j}$. As  each edge of $\partial R_+$ is preserved under $\widehat{A}_-\circ \widehat{A}_+$ (by the second condition), we have that $\widehat{A}_-(\delta_{-j})=h_{-j}(\delta_{-j})=\delta_j$. Moreover, $A_-\circ A_+$ fixes all the ideal vertices of $R_+$.

\begin{figure}[ht!]
\begin{tikzpicture}
\node[anchor=south west,inner sep=0] at (0,0) {\includegraphics[width=0.8\linewidth]{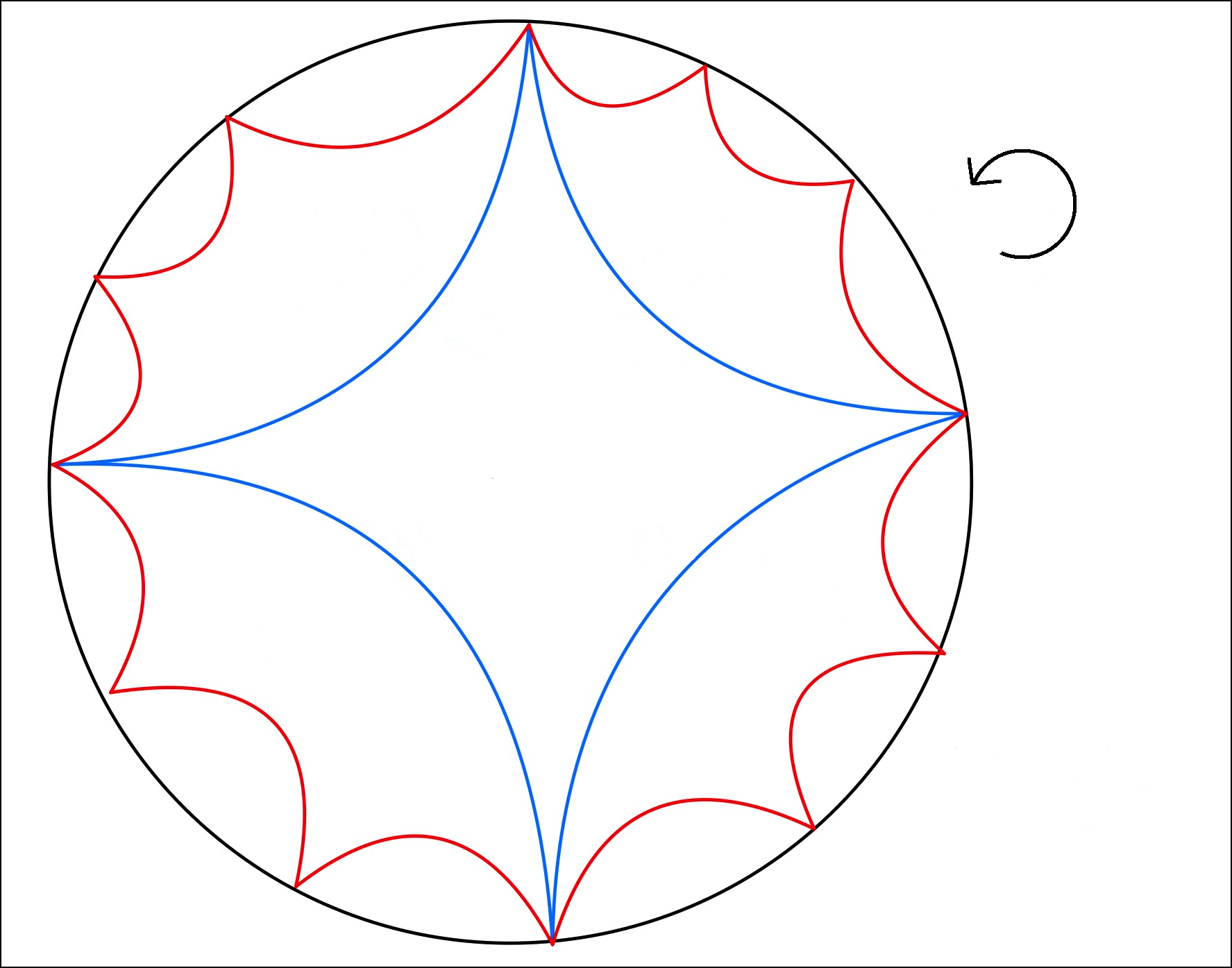}};
\node at (4.4,4.2) {\begin{Huge}$R_+$\end{Huge}};
\node at (2.8,2.8) {$h_1^{-1}(R_-)$};
\node at (3.8,3.2) {$\delta_1$};
\node at (6.2,2.8) {$h_2^{-1}(R_-)$};
\node at (5.25,3.2) {$\delta_2$};
\node at (6,6) {$h_3^{-1}(R_-)$};
\node at (5.25,5.25) {$\delta_3$};
\node at (2.8,6) {$h_4^{-1}(R_-)$};
\node at (3.8,5.25) {$\delta_4$};
\node at (8.6,1) {\begin{huge}$\overline{\D}\times\{+\}$\end{huge}};
\node at (9.4,6.4) {\begin{huge}$\mathbf{A}^{\circ 2}$\end{huge}};

\end{tikzpicture}
\caption{The first copy of $\overline{\D}$. The first return map $\mathbf{A}^{\circ 2}=\widehat{A}_-\circ \widehat{A}_+$ to this fiber is a \pwfm map. }
\label{first_return_fig}
\end{figure}
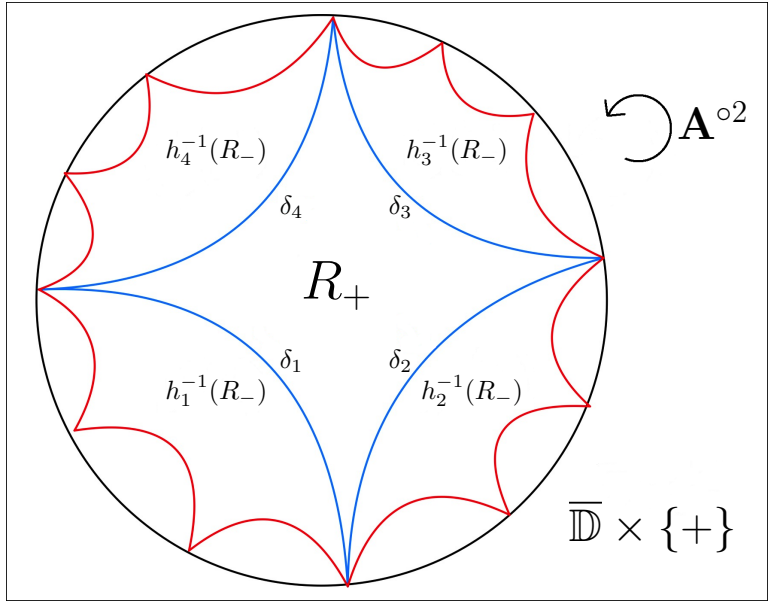
	
	Set 
	$$
	\mathbf{R}:= \Int{\left( \overline{R_+}\cup\bigcup_{j=1}^k \overline{h_j^{-1}(R_-)}\right)}.
	$$
	Clearly, the first return map of $\mathbf{A}$ on $\overline{\D}\times\{+\}$ is a \pwfm map having $\mathbf{R}\times\{+\}$ as its fundamental domain (here we use the third condition to conclude that the group generated by the pieces of $\mathbf{A}^{\circ 2}:\left(\overline{\D}\setminus\mathbf{R}\right)\times\{+\}\to\overline{\D}\times\{+\}$ is Fuchsian). The following properties now follow from the mapping properties of $h_{\pm j}$ described above (compare Figure~\ref{first_return_fig}).
	
	\begin{enumerate}\upshape
		\item Each edge of $\partial\mathbf{R}\times\{+\}$ is mapped under $\mathbf{A}^{\circ 2}$ to some $\delta_j\times\{+\}$. 
		
		\item $\mathbf{A}^{\circ 2}$ fixes each ideal vertex of $R_+\times\{+\}$. 
		
		\item $\mathbf{A}^{\circ 2}:\left(\overline{\D}\setminus\mathbf{R}\right)\times\{+\}\to\overline{\D}\times\{+\}$ has no diagonal fold.
	\end{enumerate}
(The fundamental domain of the first return map $\mathbf{A}^{\circ 2}$ in Figure \ref{first_return_fig} is the (open) ideal polygon $\mathbf{R}$ bounded by the red edges. $\mathbf{A}^{\circ 2}$ maps $\partial\mathbf{R}$ to the boundary of the (open) polygon $R_+\left(\subsetneq \mathbf{R}\right)$ bounded by the blue edges, which is the fundamental domain of the \pwfm map $A_+$.)

Thus, $\mathbf{A}^{\circ 2}:\left(\overline{\D}\setminus\mathbf{R}\right)\times\{+\}\to\overline{\D}\times\{+\}$ is a \hdm with inner domain $R_+\times\{+\}$. Finally, the mapping properties of $h_{\pm j}$ also show that this \hdm is injective on each component of $\left(\partial\mathbf{R}\setminus\partial R_+\right)\times\{+\}$. Hence, we can apply Proposition~\ref{higher_bs_char_prop} to conclude that the first return map of $\mathbf{A}$ on $\overline{\D}\times\{+\}$ is a higher Bowen-Series map of a punctured sphere Fuchsian group. 
\end{proof}

The following special situation, which we will encounter in Subsection~\ref{boundary_group_b_s_subsec}, is of particular interest.

\begin{cor}\label{hbs_first_return_cor}
	Suppose that the pieces of the \pwfm maps $A_+, A_-$ are given by $\{h_1,\cdots, h_k\}$ and $\{h_1^{-1}, \cdots, h_k^{-1}\}$ (respectively). Assume further that the edge sets of $\partial R_+, \partial R_-$ are $\{\delta_1,\cdots, \delta_k\}$ and $\{\delta_{-1},\cdots, \delta_{-k}\}$ (respectively) such that the following hold for all $j\in\{1,\cdots, k\}$.
	
	\begin{enumerate}
		\item $\widehat{A}_+\vert_{\delta_j}\equiv h_j$, $\widehat{A}_-\vert_{\delta_{-j}}\equiv h_j^{-1}$.
		
		\item $h_j(\delta_j)= \delta_{-j}$.
	\end{enumerate}  
	
	Then, the first return map of $\mathbf{A}$ on $\overline{\D}\times\{+\}$ is a higher Bowen-Series of a punctured sphere Fuchsian group.
\end{cor}
\begin{proof}
	Evidently, the canonical extension $\widehat{A}_+$ carries $\partial R_+$ homeomorphically onto $\partial R_-$, and $\widehat{A}_-\circ \widehat{A}_+$ brings each edge of $\partial R_+$ back to itself. Moreover, the Fuchsian groups generated by the pieces of $A_+$ and $A_-$ are equal. The result now follows from Proposition~\ref{hbs_return_map_prop}.
\end{proof}

Another application of Proposition~\ref{hbs_return_map_prop} is that the second iterate of the Bowen-Series map of a Fuchsian group uniformizing a sphere with $k$ punctures and zero/one/ two orbifold points of order two is the higher Bowen-Series map of a Fuchsian group uniformizing a sphere with approximately $2k$ punctures.

\begin{cor}\label{second_iterate_hbs_cor} Let $d\geq 2$.
\noindent\begin{enumerate}

\item If $\Gamma_0\in\{G_d, G_{d,2}\}$, then $A_{\Gamma_0,\mathrm{BS}}^{\circ 2}=A_{\Gamma_1,\mathrm{hBS}}$, where $\Gamma_1$ is an index two subgroup of $\Gamma_0$ with $\D/\Gamma_1\cong S_{0,2d}$.

\item  If $\Gamma_0= G_{d,1}$, then $A_{\Gamma_0,\mathrm{BS}}^{\circ 2}=A_{\Gamma_1,\mathrm{hBS}}$, where $\Gamma_1$ is an index two subgroup of $\Gamma_0$ with $\D/\Gamma_1\cong S_{0,2d-1}$.
\end{enumerate}

In all cases, the second iterate of the Bowen-Series map of $\Gamma_0$ is orbit equivalent to an index two subgroup of $\Gamma_0$.
\end{cor}
\begin{proof}
Recall that the canonical extension of the Bowen-Series map of $\Gamma_0=G_d$ restricts as an involution on the boundary of its fundamental domain. Thus, we can apply Proposition~\ref{hbs_return_map_prop} on $A_+=A_-=A_{\Gamma_0,\mathrm{BS}}$ to conclude that $A_{\Gamma_0,\mathrm{BS}}^{\circ 2}$ is the higher Bowen-Series map of a subgroup $\Gamma_1\leq\Gamma_0$ that uniformizes $S_{0,2d}$ (note that $\D/\Gamma_1$ has $2d$ punctures since the inner domain of $A_{\Gamma_1,\mathrm{hBS}}$, which is equal to the fundamental domain of $A_{\Gamma_0,\mathrm{BS}}$, has $2d$ ideal vertices). That $\left[\Gamma\colon\Gamma'\right]=2$ follows from a straightforward Euler characteristic computation.

Exactly the same proof applies to the cases with torsion elements. Finally, the statement about orbit equivalence is a consequence of Proposition~\ref{prop-cfdhd-orbeq-gen}.
\end{proof}

As the degree of the second iterate of a map is the square of the degree of the map, Corollary~\ref{second_iterate_hbs_cor} explains the appearance of the square in Equation \eqref{eq-degcfm} in Section \ref{sec-fold}.

\section{2-point characterizations and moduli}\label{sec-pattern}
Proposition \ref{bs_char_1} gives a combinatorial characterization of Bowen-Series maps. The aim of this Section is to obtain a more dynamical
characterization (Theorem \ref{thm-aisbs}) in terms of orbit equivalences of pairs of points. This characterization is necessary to determine the moduli space of matings 
in Section \ref{sec-moduli}.
Some amount of technology needs to be developed to obtain this characterization, and the right setup appears 
to be that of discrete laminations and patterns (see Definitions \ref{def-Gpattern} and \ref{def-Apattern}).
We quickly recall the notion of geodesic laminations \cite{thurstonnotes}

\begin{defn}\label{def-lami}
	
	A \emph{geodesic lamination} $\mathcal{L}^*$ on a hyperbolic surface $\Sigma=\D/\Gamma$ is a closed set of mutually disjoint simple geodesics on the surface. Lifting $\mathcal{L}^*$ to the universal cover $\D$ yields a $\Gamma$-invariant closed set of mutually disjoint geodesics that are allowed to have common endpoints on $\mathbb{S}^1$. We call this a \emph{geodesic lamination} on $\D$, and denote it by $\mathcal{L}$. Each component of a lamination is called a \emph{leaf}.  
	
	A \emph{transverse measure} $\mu$ for $\mathcal{L}^*$ is an assignment of a positive, finite measure to each transversal $\tau$ to $\mathcal{L}^*$, supported on $\tau\cap\mathcal{L}^*$ and invariant under homotopy. A \emph{measured lamination} is a geodesic lamination equipped with a transverse measure of full support. 
\end{defn}

\begin{rmk}\label{support_rem}
	We remark that not every geodesic lamination admits a transverse measure of full support, see \cite[\S 3.9]{marden-book} for details.
\end{rmk}

Let $A=A_\Gamma$ be a piecewise Fuchsian Markov map topologically conjugate
 to the polynomial map $p: z \mapsto z^d$ on $\bS^1$.
It will be important to relate $A$-invariant laminations to
$\Gamma$-invariant laminations. We now proceed to show that both $A$- and $\Gamma$-invariance imposes
considerable restrictions on laminations.
Note however that such laminations are entirely on the `group side'
of the picture, inasmuch as we are concerned with combinatorial restrictions
on laminations invariant under both $\Gamma$ and a \pwfm $\, A=A_\Gamma$ orbit equivalent to $\Gamma$.

In the context of hyperbolic geometry, the following notion goes back to Schwartz
\cite{schwarz-inv,schwartz-pihes} (see also \cite{msw2,bimj,mahan-pattern,bis-gafa} for the connection to rigidity questions). Let $\dbS$ denote $\big((\bS^1 \times \bS^1 \setminus \Delta)/\sim\big)$, where $\Delta$ denotes the diagonal and $\sim$ denotes the flip equivalence relation.

\begin{defn}\label{def-Gpattern}
	Let $\Gamma$ be a finite co-volume  Fuchsian group. A \emph{$\Gamma$-pattern} is a closed, discrete $\Gamma$-invariant subset of
	$\dbS$.
\end{defn}
Let $\gamma$ be  a closed geodesic in $\Hyp^2/\Gamma$, possibly with self-intersections. Let $\til{\gamma}$ be the pre-image of $\gamma$ in the universal cover 
	$\Hyp^2$. Then the collection 
\begin{center}
	$\{(p,q): \, p, q$  
	 are ideal endpoints  of 
		a  bi-infinite   geodesic in  $\til{\gamma} \}$
\end{center}
	 provides an example of a $\Gamma$-pattern. Thus, a $\Gamma$-pattern need not correspond to a lamination.

\subsection{$A$-patterns}\label{sec-apattern} In this subsection, we shall proceed to construct a \pwfm 
version of a $\Gamma$-pattern.
For $\Gamma$ as above, let $A$ be a degree $d$ \pwfm map on $\bS^1$, where the pieces of $A$ are from $\Gamma$. Let $\{I_j\}$ be the (finitely many cyclically ordered) intervals of definition of the pieces
$g_j$ of $A$.

We need to pay special attention to  the endpoints of the intervals
$I_j$.
Let $R$ denote the fundamental domain (see Definition \ref{def-canonicalextension-domain}) of $A$ in $\D$ with cyclically ordered ideal points
$x_1, \cdots, x_k$. Thus, the endpoints of $I_j$ are $x_j$ and $x_{j+1}$ (mod $k$). 
Let $S=\{x_1,\cdots, x_k\}$.
The Markov property
of the map $A$ ensures that $A(S) \subset S$. Thus, we have the following:

\begin{rmk}\label{rmk-pwfm-Spoints}
	Each $x_i \in S$ is pre-periodic under $A$. Thus,  grand orbits of $x_i \in S$ are in one-to-one correspondence with periodic orbits in $S$ under 
	(forward iteration by) $A$.
\end{rmk}




As before, $\big((\bS^1 \times \bS^1\setminus \Delta)/\sim\big)$ is denoted by $\dbS$.
The $A$-action on $\bS^1$ induces an action $\da: \dbS \to (\bS^1 \times \bS^1)/\sim $.

\begin{defn}\label{def_invariance_pwfm}
	Let $A$ denote a \pwfm map with pieces in a finite co-volume Fuchsian 
	group $\Gamma$. A subset $\LL$ of
	$\dbS$ is\\
\noindent (1) \emph{$A$-forward invariant} if $\da (\LL) = \LL$.\\
\noindent (2) \emph{$A$-backward invariant under branches} or simply {$A$-backward invariant} if the following holds:\\
if $\{p,q\} \in \LL$,  and $g^{-1}\in\Gamma$ is a branch of $A^{-1}$ defined on an arc $\sigma$ containing $p,q$ such that $g^{-1}(p) = p_1$ and $ g^{-1}(q) = q_1$, then $\{p_1,q_1\} \in \LL$. For 	$\{p,q\} \in \dbS$, as above, $\{p_1,q_1\}$ satisfying 	such a condition $ g^{-1}(p) = p_1$ and $ g^{-1}(q) = q_1$, where $g^{-1}\in\Gamma$ is a branch of $A^{-1}$ is called an $A$-pre-image of $\{p,q\}$.\\	
	For any $\{p,q\} \in \dbS$, the  backward orbit of 
	$\{p,q\}$ under $A$ consists of all iterated $A$-pre-images of $\{p,q\}$ (under branches of $A^{-n}$, $n \in \natls$) along with the element $\{p,q\}$. It will be denoted as $\boa (\{p,q\})$.
	The grand orbit of 
	$\{p,q\}$ under $A$ is defined to be the union of the forward orbit of
	$\{p,q\}$ and the backward orbits of all elements in the forward orbit of
	$\{p,q\}$ (with the convention that if $\da^i(\{p,q\})\in\Delta$ for some $i$, then the backward orbit of $\da^i(\{p,q\})$ is empty). The grand orbit of 
	$\{p,q\}$ under $A$  will be denoted as $\goa (\{p,q\})$.  For a bi-infinite geodesic $\gamma\subset\D$ having its endpoints at $p,q$, we will use the notation $\boa(\{p,q\})$ and $\boa(\gamma)$ (respectively, $\goa(\{p,q\})$ and $\goa(\gamma)$) interchangeably.
\end{defn}

Along the lines of Definition
\ref{def-Gpattern}, we define:

\begin{defn}\label{def-Apattern} Let $A$ denote a \pwfm map with pieces in a finite co-volume Fuchsian 
	group $\Gamma$.
	An \emph{$A$-pattern} is a closed discrete  subset $\LL$ of
	$\dbS$ that is forward and backward invariant under $\da$.
\end{defn}

\begin{lemma}\label{lem-apattern}
	Let $R$ denote the fundamental domain of a \pwfm \
	 map $A$ and $\alpha$ denote an edge of $R$. 
	Let $p,q$ denote the endpoints of $\alpha$. 
	Then $\goa(\{p,q\})\cap\Delta$ contains at most one point, and $\goa (\{p,q\})\cap\dbS$ is an $A$-pattern.
\end{lemma}

\begin{proof}
        The definition of an $A$-pre-image of an element $\{r,s\}\in\dbS$ guarantees that the $A$-pre-images of $\{r,s\}$ lie in $\dbS$. On the other hand, if the forward orbit of $\{r,s\}\in\dbS$ under $\da$ hits the diagonal $\Delta$ in some finite time $i$, then there are no forward images of $\da^i(\{r,s\})$ under $\da$ (i.e., forward iteration of $\da$ on $\{r,s\}$ stops at time $i$). It follows that $\goa(\{r,s\})$ intersects $\Delta$ in at most one point.
        
	It now suffices to show that $\goa (\{p,q\})\cap\dbS$ is closed and discrete in $\dbS$. Note that $\goa (\{p,q\})\cap\dbS$ equals the union $\cup_i\boa (\{A^{\circ i}(p),A^{\circ i}(q)\})$ in $\dbS$.
	Since the forward orbit $\{A^{\circ i}(p),A^{\circ i}(q)\}$ of the endpoints
	of any edge of $R$ is necessarily finite (since all vertices of $R$ are pre-periodic under $A$), it suffices to show that the  backward orbit of each edge or diagonal is closed and discrete. This follows
	from the hypothesis that $A$ is a \pwfm map as follows.

	Since $A$ is expansive, it follows that the diameter of each element of $A^{-j}(\{p,q\})$, considered
	as a $2$-point subset of $\bS^1$ tends to zero as $j\to+\infty$.
	Thus, any accumulation point of $\goa (\{p,q\})\cap\dbS$ in $(\bS^1 \times \bS^1)/\sim\, \,$ lies on the diagonal.
	Hence  $\goa (\{p,q\})\cap\dbS$ is a closed and discrete subset of $\dbS$ as required.
\end{proof}

The proof of Lemma \ref{lem-apattern} actually furnishes  more:

\begin{cor}\label{cor-apattern}
	Let $R$ denote the fundamental domain of a \pwfm map $A$ and $\alpha_i$, $i=1, \cdots, k$ denote the boundary edges of $R$. 
	Let $p_i,q_i$ denote the endpoints of $\alpha_i$, and $\P := \{\{p_i,q_i\}: i=1, \cdots, k\}$.
	Then $\goa (\P)\cap\Delta$ is a (possibly empty) finite set, and $\goa (\P)\cap\dbS$ is an $A$-pattern.
\end{cor}

\noindent {\bf Eliminating special diagonals:}\\
For any diagonal or edge $\gamma=\{x,y\}$ of $R$, $A(\gamma)$ is defined
as $\da(\{x,y\})=\{A(x),A(y)\}$, or equivalently as the bi-infinite geodesic joining 
$A(x),A(y)$. However, the $A$-pre-image of a diagonal is not automatically defined. This is because a branch of $A^{-1}$ can be defined (as an element of $\Gamma$) on a diagonal 
$\{x,y\}$ if and only 
if there is a branch of $A^{-1}$ defined on one of the arcs of $\bS^1$
cut off by $x, y$.
Suppose $A$ has degree $d$.
Since $A$ is $\pwfm$, $d^k$ branches of $A^{-k}$ are always defined.
The following useful observation connects pieces and branches.

\begin{rmk}\label{rmk-piecebranch}
	$h$ is a branch of $A^{-1}$ if and only if $h^{-1}$ is a piece of $A$.
	More generally, $h(\in \Gamma)$ is a branch of $A^{-k}$ if and only if $h^{-1}$ is a piece of $A^{\circ k}$.
\end{rmk}

\begin{defn}\label{def-diagonal}
	Let $p,q$ denote the endpoints of a diagonal $\delta$ of $R$. If there exists a branch $h$ of $A^{-1}$ and an edge $\alpha$ 
	of $R$ with endpoints $p_1, q_1$ such that $h(p_1)=p$, $h(q_1)=q$, then
	$\delta$ (or equivalently the pair $\{p,q\}$) will be called a 
	\emph{special diagonal} of $R$. 
\end{defn}

We summarize the above discussion in the following (rather useful) statement:
\begin{prop}[No special diagonals]\label{prop-nodiagonals}
	Let $A$ be a minimal \pwfm map. Let $R$ be a fundamental domain of $A$.
	Then special diagonals in $R$  do not exist for $A$.
\end{prop}

\begin{proof}
	Suppose $\delta$ is a special diagonal of $R$ with endpoints $p,q$. Then, by Remark \ref{rmk-piecebranch}, there exists an arc $\sigma$ of $\bS^1$
	between $p, q$ and $g\in \Gamma$ such that $A|_\sigma=g$. As $\delta$ is a diagonal of $R$, this contradicts minimality of $A$.
\end{proof}

Recall that $\mathcal{D} = \overline{\D} \setminus R$. Let  $\hA$ be the canonical extension (Definition \ref{def-canonicalextension}) of $A$ to $\mathcal{D}$. Proposition \ref{prop-nodiagonals} can be restated as saying that if $\alpha$ is an edge of $R$, then every component of $\widehat{A}^{-1}(\alpha)$ is either an edge of $R$ or lies outside the closure $\overline{R}$.

We set up some notation as follows. Each edge $\alpha$ of $R$ bounds a unique (closed) half-plane 
$\mathcal{D}_\alpha \subset \mathcal{D}$, whose boundary contains an arc $I_\alpha \subset \bS^1$ such that $I_\alpha$ is the domain of a piece of $A$.
Thus, $A|_{I_\alpha}=g_\alpha$ for some $g_\alpha\in \Gamma$, and
$\hA|_{\mathcal{D}_\alpha}=g_\alpha$. 
We have
the following  stronger repelling condition on $\hA$. 

\begin{cor}\label{cor-strongrepulse}
Let $A$ be a minimal \pwfm map. Then,
	$\hA^{-1}(\Int{\mathcal{D}}) \subset \Int{\mathcal{D}}$, and the set of break-points of $\hA^{-1}(\Int{\mathcal{D}})$ on $\bS^1$ contain those of $\Int{\mathcal{D}}$.
	Hence, 
	$$
	\cdots\hA^{-n}(\Int{\mathcal{D}}) \subset \hA^{-(n-1)}(\Int{\mathcal{D}})\subset \hA^{-(n-2)}(\Int{\mathcal{D}})\subset \cdots\hA^{-1}(\Int{\mathcal{D}}) \subset \Int{\mathcal{D}},
	$$
	and $\cap_i \hA^{-i}(\Int{\mathcal{D}}) = \emptyset$. 
	Also, every boundary edge of $\hA^{-i}(\Int{\mathcal{D}})$ maps to an edge of $R$ under $\hA^{\circ i}$.
\end{cor}

\begin{proof}
	Since $A: \bS^1 \to \bS^1$ is a degree $d$ map, $A^{-1}(I_\alpha)$
	is a disjoint union of $d$ arcs. By Proposition \ref{prop-nodiagonals},  each of these arcs is the boundary at infinity
	of an open half-plane contained in $\Int{\mathcal{D}}$. Also, since $S$ is preserved under $A$,
	$S \subset A^{-1}(S)$. Hence, $\hA^{-1}(\Int{\mathcal{D}}) \subset \Int{\mathcal{D}}$, and the set of break-points at infinity of $\hA^{-1}(\Int{\mathcal{D}})$ contain those of $\Int{\mathcal{D}}$.
	
	Iterating $\hA^{-1}$, we get the second assertion. Since $A$ is expansive, 
	$\cap_i \hA^{-i}(\Int{\mathcal{D}}) = \emptyset$. The last assertion is clear.
\end{proof}

\subsection{Characterizing Bowen-Series maps through patterns}\label{sec-agpattern}

\begin{defn}\label{def-Abicondn}
	Let $A$ be a \pwfm map with fundamental domain $R$. 
	
	\begin{enumerate}
		\item We say that $A$ is \emph{backward edge-orbit equivalent} to $\Gamma$ if for every edge $\alpha$ of $R$, we have $\boa(\alpha)=\Gamma.\alpha$.
		\item $A$ is said to be \emph{simplicial} if for every edge $\alpha$ of $R$,
		$\hA(\alpha)$ is also an edge of $R$.
	\end{enumerate}
\end{defn}

\begin{lemma}\label{lem-ApatisalmostGpat-simplicial}
	Suppose that $A$ is simplicial. Then for each edge $\alpha$
	of $R$, we have $\goa(\alpha)\subset \Gamma.\alpha$.
\end{lemma}

\begin{proof} 
	For every edge $\alpha$, $\hA(\alpha)$ is an edge of $R$ belonging to
	$\Gamma.\alpha$. Also, each connected component of $\hA^{-i}  (\alpha)$
	equals $g.\alpha$ for some branch $g$ of  $\hA^{-i}$. Hence the grand orbit
	$\goa (\alpha)$ is contained in $\Gamma.\alpha$.
\end{proof}

We  observe the following restriction on $\Gamma$-patterns (see, for instance, the proof of \cite[Proposition 2.3]{mahan-relrig}):

\begin{lemma}\label{lem-alpha-gpattern} Let $\Gamma$ be a Fuchsian lattice, $A$ 
	a \pwfm map with pieces lying in $\Gamma$, and $R$ a fundamental domain of $A$.
	Let $p,q$ be the
	endpoints  of an edge $\alpha$ of $R$. Then $\Gamma. \{p,q\}$ is a $\Gamma$-pattern if and only if
	\begin{enumerate}
		\item either $p, q$ are parabolic break-points,
		\item or $p, q$ are hyperbolic break-points corresponding to attracting 
		and repelling fixed points of a hyperbolic element.
	\end{enumerate}
	
\end{lemma}

\begin{proof} The points $p, q \in \bS^1$ are break-points of $A$ and hence fixed points of some elements of $\Gamma$.
	Suppose first that  the break-points of $A$ are stabilized by parabolics in $\Gamma$.
	Hence, under the covering projection $\Pi: \D \to \D/\Gamma$, $\Pi(\alpha)$ is a bi-infinite geodesic in $\Sigma=  \D/\Gamma$, whose ends go down cusps of $\Sigma$. Therefore $\Gamma.\alpha$
	is a `discrete lamination'; i.e., a  closed subset of $\D$ consisting of a countable collection of bi-infinite geodesics, none of which is accumulated on by others. Hence $\Gamma. \{p,q\}$ is closed and discrete.
	
	Next, if $p, q$ are hyperbolic break-points corresponding to attracting 
	and repelling fixed points of a hyperbolic element $g$, then $\alpha$ is stabilized by $g$ and
	$\Pi(\alpha)$ is a closed geodesic in $\Sigma$. Hence $\Gamma. \{p,q\}$ is a $\Gamma$-pattern.
	
	Conversely, suppose one of the endpoints $p$ of $\alpha$ is fixed by a hyperbolic element $g$ and the other endpoint is a fixed point of a non-trivial element $h \in \Gamma$ with $h \neq g$. Let $\lambda$ denote the
	bi-infinite geodesic in $\D$ stabilized by $g$.
	Then $\{g^n(\alpha)\}_{n\in\Z}$ accumulates on $\lambda$, and $\Gamma. \{p,q\}$ is not closed and discrete; i.e., it is not a $\Gamma$-pattern.
\end{proof}

We now show that backward edge-orbit equivalence characterizes Bowen-Series maps:

\begin{prop}\label{prop-nocollapsetoedge} Let $A$ be a minimal \pwfm map that is backward edge-orbit equivalent to $\Gamma$. Then 
 all break-points of $A$ are parabolic,
	  $A$ is simplicial,
	and   $\hA(\partial R)=\partial R$.
\end{prop}

\begin{proof} \noindent {\bf All break-points of $A$ are parabolic:} Let $\alpha$ be an edge of $R$ with endpoints $p, q$.
	By Lemma \ref{lem-apattern}, $\boa (\{p,q\})$ is an $A$-pattern. In particular, $\boa (\{p,q\})$ is a closed, discrete
	subset of $\dbS$. Since $A$ is backward edge-orbit equivalent to $\Gamma$,
	$\Gamma.\{p,q\}$ is a closed, discrete
	subset of $\dbS$. By Lemma \ref{lem-alpha-gpattern}, 
	
	\begin{enumerate}
		\item either $p, q$ are parabolic break-points,
		\item or $p, q$ are hyperbolic break-points corresponding to attracting 
		and repelling fixed points of a hyperbolic element.
	\end{enumerate}
	
	Since $\alpha$ was arbitrary, it follows that either all break-points are hyperbolic
	or all break-points are parabolic. 
	Suppose all break-points are hyperbolic. Then, by Lemma \ref{lem-alpha-gpattern} again, each edge $\alpha$ has endpoints $p, q$
	the attracting and repelling fixed points  of a hyperbolic $g \in \Gamma$.
	This forces all the edges to coincide; i.e., $R$ is empty, and $A$ has exactly two break-points given by the attracting and repelling fixed points
	of a single  hyperbolic $g \in \Gamma$ (by discreteness of $\Gamma$). But this is impossible for a \pwfm map. This establishes the first conclusion
	of the Lemma.
	
	\noindent {\bf $A$ is simplicial:}	Let  $\alpha$ be an  edge of $R$. Then, since $S$ is pre-periodic under $A$, it follows that $\hA(\alpha)$ is either an edge or a diagonal of $R$. We now observe that 
	$\hA(\alpha)=\beta$ cannot be a diagonal of $R$. To do so, assume that $\beta$ is a diagonal. Let $g \in \Gamma$ be the piece of $A$ restricted to an arc of $\bS^1$ subtended by $\alpha$. Then $\beta=g.\alpha$.
	By the definition of backward edge-orbit equivalence, and $\boa(\alpha)$, there exists $n \in \natls$ such that $\beta$ is a connected component of $\hA^{-n}(\alpha)$.
	But this contradicts Corollary \ref{cor-strongrepulse}. Hence $\hA(\alpha)$ is  an edge  of $R$. 
	
	\noindent {\bf $\mathbf{\hA(\partial R)=\partial R}$:}
	We argue the last conclusion by contradiction. The second conclusion ensures that $\hA(\partial R) \subset
	\partial R$. Suppose that $\hA(\partial R) \subsetneq
	\partial R$.
	By continuity of $\hA$, there exists a contiguous family of intervals $I_1, 
	\cdots, I_r \subset \partial R$ such that  $\hA(\partial R)=I_1 \cup
	\cdots \cup I_r =J \subsetneq \partial R$. Hence $\hA(J) \subset J$.

	Let $\alpha $ be an edge of $R$ not contained in $J$. Then $\hA(\alpha)=\beta$ is an edge contained in $J$.
	Let $\beta = g.\alpha$, where $g \in \Gamma$.
	Since $\boa (\alpha) = \Gamma.\alpha$, it follows that there exists $n \in \natls$ such that $\beta$ is a connected component of
	$\hA^{-n}(\alpha)$. But then $\hA^{\circ n}(\beta)$ must equal $\alpha$, since $A$, and hence $A^{\circ n}$ is simplicial.
	It follows that $\alpha$ is contained in $J$, a contradiction. 
\end{proof}

\begin{rmk}\label{rmk-stab} Proposition \ref{prop-nocollapsetoedge} shows that each edge $\alpha$ of $R$
	is, in fact, periodic under forward iteration by $\hA$.
	
	Since the endpoints of any $I_s$ are
	parabolic break-points of $A$, the stabilizer in $\Gamma$ of the edge $\alpha_s$
	of $R$ joining the endpoints of $I_s$ is either trivial (when $\Gamma$ is torsion-free) or $\Z/2\Z$ (in case $\Gamma$ has torsion). Let $w$   be a  fixed point of $\widehat{A}$ on $\partial R$.
	\begin{enumerate}
		\item Then either $w$ is a break-point of $A$, and $\hA$   exchanges the two edges $\alpha, \beta$ of $R$ incident on $w$, or
		\item there exists an edge $I_l$ of $R$ such that $w \in I_l$,
		$\hA(I_l)=I_l$ and $\hA$ acts
		as an order 2 isometry on $I_l$ fixing $w$.
	\end{enumerate}
\end{rmk}

\begin{theorem}\label{thm-aisbs} Let $A$ be a minimal \pwfm map backward edge-orbit equivalent to $\Gamma$.
	Then
	$A$ is the Bowen-Series map for $\Gamma$ corresponding to the fundamental domain $R$. In particular,
	$R$ is a fundamental domain of $\Gamma$ and $\D/\Gamma$ is a sphere with punctures with possibly one or two orbifold points of order 2.
\end{theorem}

\begin{proof}
	Proposition \ref{prop-nocollapsetoedge} ensures $\hA(\partial R)=\partial R$. The Theorem now follows from Proposition \ref{bs_char_1}.
\end{proof}

\begin{rmk}\label{rmk-semigroup}
	Theorem \ref{thm-aisbs} and
	the discussion preceding Remark \ref{rmk-pwfm-Spoints} now shows that $\boa(\alpha)$ 
	equals a $\Gamma$ \emph{semigroup} orbit for the semigroup generated by branches of $\hA^{-1}$ that
	are defined on $\alpha$.  We emphasize that these branches generate $\Gamma$ as a semigroup
	and that successive iterates under branches of $\hA^{-1}$ on $\alpha$ will give us a semigroup orbit
	rather than a group orbit. 
\end{rmk}

\subsection{Edge-orbit equivalence of $\hA$}\label{sec-eoe}

All Fuchsian groups considered in this subsection will be assumed to be torsion-free.

For a minimal \pwfm map $A$, Corollary \ref{cor-strongrepulse} allows us to define $\hA$-grand orbits
of bi-infinite geodesics in $\D$.
Let $\hA:\mathcal{D} \to \overline{\D}$ be the canonical extension of $A$.
Assume that $A$ has degree $d$.
Then by Corollary \ref{cor-strongrepulse},  
$$
\cdots\hA^{-n}(\Int{\mathcal{D}}) \subset \hA^{-(n-1)}(\Int{\mathcal{D}})\subset \hA^{-(n-2)}(\Int{\mathcal{D}})\subset \cdots\subset\hA^{-1}(\Int{\mathcal{D}}) \subset \Int{\mathcal{D}}.
$$
Then for any bi-infinite geodesic $\alpha \subset \partial\mathcal{D}$, there exists a unique (closed) half-plane $\mathcal{D}_\alpha \subset \mathcal{D}$ bounded by the edge $\alpha$ of $R$ and an interval $I_\alpha \subset \bS^1$ 
(see Definition \ref{def-canonicalextension}) such that $\hA^{-1}(\alpha)$
contains at least $d$ bi-infinite geodesics, one corresponding to each branch of $\hA^{-1}$. Further, the fact that $A$ preserves the set of ideal vertices of $R$ implies that exactly one of the two following possibilities occur:
\begin{enumerate}
	\item There exists $k \in \natls$ such that $\hA^{\circ k} (\alpha)$ is contained in $\hA(\mathcal{D}) \setminus \mathcal{D}$ and hence
	$\hA^{\circ (k-1)} (\alpha)$ is contained in $\mathcal{D}$.
	
	\item $\alpha$ is pre-periodic under iteration by $\hA$. Hence,
	there exists 
	$k \in \natls$ such that $\hA^{\circ k} (\alpha)$ is a periodic edge of
	$R$, and  all forward iterates of $\alpha$ under $\hA$ are
	edges of $R$.
\end{enumerate}
In the first case, forward iteration stops when $\hA^{\circ k} (\alpha)$ exits
$\mathcal{D}$; in the second case, $\hA^{\circ k} (\alpha)$ is periodic under $\hA$.
In the first case, the \emph{grand orbit of $\alpha$ under $\hA$} is defined as the union of 
\begin{enumerate}
	\item all forward  iterates $\hA^{\circ i}(\alpha),\ 0 \leq i \leq k$, where $k$ is the smallest positive integer 
	for which $\hA^{\circ k} (\alpha)$ exits
	$\mathcal{D}$, and
	\item all iterated pre-images of $\hA^{\circ i}(\alpha)$ under $\hA,\ 0\leq i\leq k$.
\end{enumerate}
In the second case, the \emph{grand orbit of $\alpha$ under $\hA$} is defined as the union of all
backward iterates of all forward iterates of $\alpha$.
In either case, we obtain  a collection of bi-infinite geodesics in $\D$.
We denote the grand orbit of $\alpha$ under $\hA$ by $\goah(\alpha)$.

Replacing $\alpha$ by its ideal endpoints, $\{p,q\}$, the set of ideal endpoints of bi-infinite geodesics in $\goah(\alpha)$ is a subset of $\dbS$ and we denote it as  $\goah(\{p,q\})$.

\begin{defn}\label{def-eoe}
	$\hA$ is said to be \emph{edge-orbit equivalent } to $\Gamma$ if 
	\begin{enumerate}
	\item for every edge $\alpha$ of $R$, we have $\goah (\alpha)= \goa(\alpha)=\Gamma.\alpha$, and
	
	\item if a diagonal $\{r, s\}$ of $R$ is periodic under $\da$, then its endpoints must be fixed under $A$. 
	\end{enumerate}
	\end{defn}

\begin{rmk}
In Definition~\ref{def_invariance_pwfm}, we defined the grand orbit $\goa(\{p,q\})$ of $\{p,q\}\in\dbS$ using the map $\da$ which only records the action of $\hA$ at infinity. On the other hand, the grand orbit $\goah (\{p,q\})$ is defined in terms of the action of $\widehat{A}$ on bi-infinite geodesics in the interior of the disk. As an example of the difference between these two grand orbits, we note that if an edge $\{p, q\}$ of $R$ maps to a diagonal $\{r,s\}$ of $R$ under $\hA$, then $\{A(r), A(s)\}$ always lies in $\goa(\{p,q\})$, but not necessarily in $\goah(\{p,q\})$. 

The condition $\goah (\{p,q\})=\goa (\{p,q\})$ and the second condition of Definition~\ref{def-eoe} ensure compatibility between the boundary action and the interior action of $\hA$.
\end{rmk}

\begin{lemma}\label{lem-diagonalAfixed} 
Let $A$ be a minimal \pwfm map such that $\hA$ is edge-orbit equivalent  to $\Gamma$. Let $\delta$ be a diagonal  of $R$ such that $\delta = \hA (\alpha)$ for some edge $\alpha$ of $R$. Let $r, s$ be the ideal endpoints of $\delta$.
	Then $A(r)=r, A(s)=s$.
\end{lemma}

\begin{proof}
Let us denote the endpoints of $\alpha$ by $p,q$.
We first argue that $A(r)\neq A(s)$. In fact, if $A$ maps $r, s$ to the same point, then $\{A(r), A(r)\}\in\goa(\{r, s\})=\goa(\{p,q\})=\Gamma.\{p,q\}$ (by edge-orbit equivalence of $\hA$ and $\Gamma$). But a group element cannot map the distinct points $p, q$ to the same point. Thus, $A(r)\neq A(s)$; i.e., $A(r), A(s)$ must be the ideal endpoints of either an edge or a diagonal of $R$.

If $A(r), A(s)$ are also the ideal endpoints of $\delta$, then $\{r, s\}$ is fixed under $\da$, and hence the second defining property of edge-orbit equivalence implies that $r$ and $s$ are fixed by $A$. We claim that $A(r), A(s)$ cannot be the ideal endpoints of a diagonal other than $\delta$. Indeed, if $A(r), A(s)$ are the ideal endpoints of a diagonal $\delta'\neq\delta$, then $\delta'\in\goa(\alpha)=\goah(\alpha)=\goah(\delta)$ (here we use the first defining property of edge-orbit equivalence). But this is impossible as $\hA$ is not defined on the diagonals of $R$.

Now suppose that $A(r), A(s)$ are the ideal endpoints of an edge $\alpha'$ of $R$. Then, $\alpha'\in\goa(\alpha)=\goah(\alpha)=\goah(\delta)$. This implies that some iterate of $\hA$ carries $\alpha'$ onto $\delta$. As $\da$ sends the endpoints of $\delta$ to those of $\alpha'$, we conclude that $\{r, s\}$ is periodic under $\da$. Once again, by the second defining property of edge-orbit equivalence, we have that $A(r)=r$, $ A(s)=s$.
\end{proof}

\begin{lemma}\label{lem-goah-allparab} Suppose $A$ is a minimal \pwfm map such that $\hA$ is edge-orbit equivalent to $\Gamma$. Then
	\begin{enumerate}
		\item 	 For every edge $\alpha$ of $R$,
		$\goah (\alpha) $ is a $\Gamma$-pattern.
		\item All break-points of $A$ are parabolic.
	\end{enumerate}
	
\end{lemma}

\begin{proof}
Edge-orbit equivalence of $\hA$ and $\Gamma$ tells us that $\goah(\alpha)=\goa(\alpha)=\Gamma.\alpha\subset\dbS$. By Lemma~\ref{lem-apattern}, $\goa (\alpha)$ is a closed discrete subset of $\dbS$. Moreover, $\Gamma.\alpha$ is $\Gamma$-invariant by definition. This completes the proof of the fact that $\goah(\alpha)$ is a $\Gamma$-pattern.
	
	Next, Lemma \ref{lem-alpha-gpattern} and the first part of Proposition \ref{prop-nocollapsetoedge} show that all break-points of $A$ (i.e., the vertices of $R$) are parabolic.
\end{proof}

\begin{defn}\label{def-discretelamn}
	A $\Gamma$-pattern $\LL$ is called a \emph{discrete $\Gamma$-lamination} if no pair $\{p_1, q_1\},
	\{p_2, q_2\} \in \LL$ is linked, or equivalently if $\alpha_i$ is the bi-infinite geodesic in $\D$
	joining $\{p_i, q_i\}$ for $i=1,2$, then $\alpha_1, \alpha_2$ do not intersect. 
\end{defn}

\begin{lemma}\label{lem-goah-discretelamn}
Suppose $A$ is a minimal \pwfm map such that $\hA$ is edge-orbit equivalent to $\Gamma$. Let $\mathcal{L}':=\{\hA (\alpha):\ \alpha\ \textrm{is\ an\ edge\ of}\ R\}$. Then the following hold.
\begin{itemize}
\item No two distinct bi-infinite geodesics in $\LL'$ intersect in $\D$.

\item $\Gamma.\LL'=\bigcup_{\alpha \subset \partial R} \goah (\alpha)$, where the union is taken over all boundary edges $\alpha$ of $R$.

\item $\bigcup_{\alpha \subset \partial R} \goah (\alpha)$ is a discrete $\Gamma$-lamination.

\end{itemize}
\end{lemma}
\begin{proof}
1) Clearly, neither two distinct edges of $R$, nor an edge and a diagonal of $R$ intersect in $\D$. Thus, it suffices to show that if $\alpha_1, \alpha_2$ are edges of $R$ such that $\delta_i=\hA(\alpha_i)$ ($i=1,2$) are distinct diagonals (of $R$), then $\delta_1$ and $\delta_2$ do not intersect in $\D$. By way of contradiction, assume that $\delta_1$ and $\delta_2$ intersect. Further suppose that $\hA\vert_{\alpha_i}=g_i,\ i\in\{1,2\}$. Then, the bi-infinite geodesic $\delta'':= g_2^{-1}\circ g_1(\alpha_1)=g_2^{-1}(\delta_1)$ intersects $\alpha_2$, and hence its two ideal endpoints lie in two different pieces of $A$. It follows that $\hA$ is not defined on $\delta''$. Since $\delta''\in\Gamma.\alpha_1=\goah(\alpha_1)$ (by edge-orbit equivalence), some iterate of $\hA$ must carry $\alpha_1$ onto $\delta''$. But this is impossible as $\delta''$ is neither an edge nor a diagonal of $R$.

2) This directly follows from the definitions of $\LL'$ and edge-orbit equivalence.

3) That $\bigcup_{\alpha \subset \partial R} \goah (\alpha)$ is a  $\Gamma$-pattern follows from the first part of Lemma \ref{lem-goah-allparab}, since the union is finite.
	
By way of contradiction, suppose that the bi-infinite geodesics $\gamma_1,\gamma_2\in\bigcup_{\alpha \subset \partial R} \goah (\alpha)$ intersect in $\D$. As no two distinct geodesics in $\LL'$ intersect, both $\gamma_1,\gamma_2$ must be iterated $\hA$-pre-images of geodesics in $\LL'$. In order that they intersect, they must lie in the interior of a common $\mathcal{D}_{\alpha}$, where $\alpha$ is an edge of $R$ (see the discussion before Corollary~\ref{cor-strongrepulse}). Moreover, each geodesic in the $\hA$-forward orbit of $\gamma_i$ ($i\in\{1,2\}$) either lies in $\mathcal{D}$ or equals a diagonal of $R$ in $\LL'$. Since $\hA$ acts by a single group element on each $\mathcal{D}_{\alpha}$, the geodesics $\hA(\gamma_1),\hA(\gamma_2)$ also intersect in $\D$. So in light of the first part of this proposition, they must lie in the interior of a common $\mathcal{D}_{\alpha}$ as well. Iterating this argument, one sees that $\hA^{\circ j}(\alpha_1), \hA^{\circ j}(\alpha_2)$ lie in the interior of a common $\mathcal{D}_{\alpha}$, for all $j\geq 0$. But this contradicts the fact that $\gamma_1,\gamma_2\in\bigcup_{\alpha \subset \partial R} \goah (\alpha)$.
\end{proof}

\begin{defn}\label{def-innerdomain}
	Let $\{w_1, \cdots, w_k\} \subset S$ be the set of ideal vertices of $R$ that
	are fixed under $A$. Then the interior of the convex hull of $w_1, \cdots, w_k$
	will be called the \emph{inner domain} of $A$, and will be denoted as $D$.
\end{defn}

Note that if $\hA$ is edge-orbit equivalent to $\Gamma$, then by Definition \ref{def-eoe} and Lemma~\ref{lem-diagonalAfixed}, any diagonal $\delta$ of $R$ of the form $\delta = \hA(\alpha)$
for an edge $\alpha$ of $R$ lies in $\overline{D}$.

\begin{defn}\label{def-fold} 
	$\hA$ is said to have a folding if there exist adjacent edges $\alpha, \beta$ of the fundamental domain $R$ such that the bi-infinite
	geodesics $\hA(\alpha)$ and $\hA(\beta)$ are the same.
\end{defn}

\begin{prop}\label{prop-goah-dichotomy}
	Suppose $A$ is a (minimal) mateable map without folding such that $\hA$ is edge-orbit equivalent to the Fuchsian group $\Gamma$ generated by its pieces. Then either $\hA(\partial R)=\partial R$, or, 
	for every edge $\alpha$ of $R$, there exists $n \in \natls$ such that $\hA^{\circ n}(\alpha)$ is a diagonal of $R$ contained in 
	$\overline{D}$. Further, 
	\begin{enumerate}
		\item The closure $\overline{D} $ of $D$ in $\D$ is contained in $R$.
		\item For every edge $\alpha$ of $R$, either $\hA(\alpha)$ is an edge 
		of $R$ or a diagonal of $R$ contained in 
		$\overline{D}$. 
	\end{enumerate}

	In either case, $\Sigma=\D/\Gamma$ is homeomorphic to a sphere with punctures.
\end{prop}

\begin{proof}	
	Let $w\in S$ be an ideal vertex of $D$. Then $A(w)=w$. Let $\overline{w u_1}$
	be an edge of $R$. First, $u_1$ cannot be fixed under $A$, since 
	this will force $\hA(\overline{w u_1}) = \overline{w u_1}$, and since all break-points
	of $A$ are parabolic by Lemma \ref{lem-goah-allparab}, the stabilizer in $\Gamma$ of the pair of points $\{w,u_1\}$ is trivial.
	Hence $\hA$ must be the identity on the interval of $\bS^1$ that is
	bounded by $\{w,u_1\}$ and
	supports a piece of $A$.  This violates expansivity of $A$.
	
	We observe next that if $u_1$ is periodic, then under $A$, it must map to the unique vertex $v_1\neq u_1$ that is adjacent to $w$ on $\partial R$.
	Else $\overline{w A(u_1)}$ must be a diagonal of $R$, forcing $A(u_1)$
	to be fixed under $A$, contradicting the assumption that $u_1$ is periodic.
	Since $u_1$ cannot be fixed by the previous paragraph, it follows
	that $A(u_1)=v_1$. Similarly, $A(v_1) = u_1$.
	
	Let $u_2$ be  the unique vertex of $R$ other than 
	$w$ that is adjacent to $u_1$. Similarly, let $v_2$ be  the unique vertex of $R$ other than 
	$w$ that is adjacent to $v_1$. Then 
	\begin{enumerate}
		\item either $A(u_2) = v_2$,
		\item or $A(u_2) = w$.
	\end{enumerate}
	Else $\hA(\overline{u_1 u_2})$ is a diagonal, forcing $v_1$ to be a fixed point under $A$, a contradiction. If $A(u_2) = w$, then the two adjacent edges
	$\hA(\overline{u_1 u_2})$ and $\hA(\overline{w u_1})$ are folded over the edge $\overline{w v_1}$, contradicting the hypothesis.
	Hence,  $A(u_2) = v_2$.
	
	Proceeding inductively, we observe that either there is a folding, or
	$A(u_k) = v_k$ for all $k$, where $u_k, v_k$ are defined as above.
	The process terminates when 
	\begin{enumerate}
		\item Either $u_k=v_k$, in which case $A(u_k)=u_k$,
		\item Or, $u_k, v_k$ are adjacent vertices and $A$ exchanges them.
	\end{enumerate}
	In either case, we obtain $\hA(\partial R) = \partial R$.
	Summarizing the above argument, we have that if $u_1$ is periodic under $A$, then $\hA(\partial R)= \partial R$.
	
	Hence, if $\hA(\partial R)\neq \partial R$, then for every ideal vertex $w$
	of $D$ and ideal vertex
	$u_1$ of $R$ adjacent to $w$, either
	$\hA(\overline{w u_1})$ or $\hA^{\circ 2}(\overline{w u_1})$ is a diagonal of $R$ contained in $\overline{D}$. Let $k=1$ or $2$ be such that $\hA^{\circ k}(\overline{w u_1})$ is a diagonal of $R$ contained in $\overline{D}$. If $\hA^{\circ k}(\overline{u_1 u_2})$ is well-defined and is an edge of $R$, then one of the endpoints of $\hA^{\circ k}(\overline{u_1 u_2})$ is the ideal vertex $A^{\circ k}(u_1)$ of $D$. But then, $\hA^{\circ k}(\overline{u_1 u_2})$ must be mapped to a diagonal of $R$ contained in $\overline{D}$ by $\hA$ or $\hA^{\circ 2}$. Thus, $\hA^{\circ l}(\overline{u_1 u_2})$ is a diagonal of $R$ contained in $\overline{D}$ for some $l\geq 1$. Iterating the argument, we conclude that 
	for every edge $\alpha$ of $R$, there exists $n \in \natls$ such that $\hA^{\circ n}(\alpha)$ is a diagonal of $R$ contained in 
	$\overline{D}$. We note also that this argument further shows that
	for every edge $\alpha$ of $R$, either $\hA(\alpha)$ is an edge 
	of $R$ or a diagonal of $R$ contained in 
	$\overline{D}$. 
	
	Note further that the closure $\overline{D} $ of $D$ in $\D$ is contained in $R$. Else, there is an edge $\alpha$ of $D$ that is also an edge of $R$, forcing $\hA$ to fix $\alpha$ pointwise and violating expansivity of $A$. (This is similar to the argument in the first paragraph of the proof of the present proposition and we omit details.)
	
	It remains to establish the topology of $\Sigma$. When 
	$\hA(\partial R)= \partial R$, this was shown in Proposition \ref{bs_char_1}. Else, for every edge $\alpha$, $\hA^{\circ n}(\alpha) \subset \overline{D}$ for some $n \in \natls$. Thus, there exists a diagonal of $R$; namely,
	$\delta \subset \overline{D}$, such that $\boa(\delta) = \goah (\alpha)$.
	Hence, by Lemma~\ref{lem-goah-discretelamn}, there exists a lamination
	$\LL_D \subset \overline{D}$ consisting of finitely many leaves such that
	$\Gamma.\LL_D = \bigcup_{\alpha \subset \partial R} \goah (\alpha)$,
	where the union is taken over all edges of $R$. In fact, such a lamination $\LL_D$ is given by $\LL'\cap\bbar{D}$, where $\LL'$ is as in  Lemma~\ref{lem-goah-discretelamn}.
	
	All ideal vertices of $D$ are fixed points of $A$. Hence, by orbit equivalence of $A$ and $\Gamma$, they lie in distinct $\Gamma$-orbits of parabolic fixed
	points in $\bS^1$. Since $\Gamma.\LL_D = \bigcup_{\alpha \subset \partial R} \goah (\alpha)$ is a lamination,
	$\gamma.\bbar{D} \cap \bbar{D} = \emptyset$ (here $\overline{D}$ is the closure of $D$ in $\D$) for all non-trivial $\gamma \in \Gamma$ (else, some non-trivial element of $\Gamma$ will carry an ideal vertex of $D$ to another ideal vertex of $D$). Hence $\bbar{D}$ embeds in $\Sigma$
	under the covering map $\Pi:\D \to \Sigma$; i.e.,
	$\Pi(\bbar{D})\subset \Sigma$ is a closed, embedded ideal polygon. We conclude by showing that
	$\Sigma$ cannot have a handle or a new puncture. Else there exists a subsurface 
	$\Sigma_0$ of infinite fundamental group and with geodesic boundary such that $\Pi(\bbar{D}) \cap \Sigma_0=\emptyset$ (simply by choosing $\Sigma_0=\Sigma\setminus\Pi(\overline{D})$). Then a lift $\widetilde{\Sigma}_0$
	to $\D$ is an infinite sided polygon not intersecting $\Gamma.\LL_D$. But each complementary region of $\D \setminus \bigcup_{\alpha \subset \partial R} \goah (\alpha)$ is a finite-sided polygon, a contradiction.
\end{proof}

As an immediate consequence of the above proof, we have:

\begin{cor}\label{cor-numberofpunct}
	Suppose $A$ is a (minimal) mateable map without folding such that $\hA$ is edge-orbit equivalent to the Fuchsian group $\Gamma$ generated by its pieces, and $\hA(\partial R)\neq \partial R$. Then $\Sigma = \D/\Gamma$ is homeomorphic to $S_{0,k}$, where $k$ is the number of ideal vertices of $D$ (or equivalently, the number of edges of $D$).
\end{cor}

The rest of this subsection is devoted to proving that a \pwfm map $A$ satisfying the hypothesis of Corollary \ref{cor-numberofpunct}
is, in fact, a \emph{higher Bowen-Series map}, thus characterizing such mateable maps
(see Theorem \ref{thm-char-hdm-eoe} below).

\noindent {\bf Checkerboard Tiling:}
We set up some notation first.
Let $\mathbf{U} = \Sigma \setminus \Pi(\bbar{D})$. Since $\Sigma$ is homeomorphic to $S_{0,k}$ (where $k$ is the number of ideal vertices of $D$) and $\Pi$ gives an embedding of 
$\bbar{D}$ in $\Sigma$, it follows that $\mathbf{U}$ is also an ideal hyperbolic polygon with $k$ ideal vertices. Let $\TT$ denote the tiling of $\D$ induced by the lifts of $\Pi(\bbar{D})$ and $\mathbf{U}$ to
the universal cover $\D$. Choose a lift of $\mathbf{U}$ adjacent
to $D$ (i.e., sharing a boundary edge with $D$) and call it $U$. Then the translates $\Gamma.D$ and  $\Gamma.U$ give a `checkerboard
tiling' structure to $\TT$; i.e., any pair of tiles $T_1, T_2$ in $\TT$ that share a boundary edge are of the form $T_1 \in \Gamma.D$
and $T_2 \in \Gamma.U$. We shall call $\TT$ the \emph{checkerboard
	tiling} of $\D$ induced by $D, U$. Note that the discrete lamination $\bigcup_{\alpha \subset \partial R}\goah (\alpha)$
equals $\Gamma.\LL_D$, where $\LL_D$ is the finite lamination contained in $\overline{D}$ as in the proof of Proposition \ref{prop-goah-dichotomy} above. For convenience of notation, let
$\LL= \Gamma.\LL_D$.
Then, any leaf
of $\LL$ is contained in $\gamma.\overline{D}$ for some $\gamma \in \Gamma$.
We shall refer to translates of $D$ (resp. $U$) by $\Gamma$ as
\emph{$D$-tiles (resp. $U$-tiles)}

The checkerboard tiling $\TT$ shows that there are $k$ $U$-tiles 
$U_1, \cdots, U_k$ adjacent
to the inner domain $D$, where $w_i, w_{i+1}$ are ideal vertices of $U_i$ ($i+1$ mod $k$). 
Further, for each such $U_i$, there exist $k-1$ adjacent $D$-tiles $D_{ij}$,
$j= 1, \cdots, k-1$, apart from the inner domain $D$. 
Since each edge of $R$ eventually falls on some leaf of $\LL_D$ under forward iteration of $\hA$, we have $\partial R \subset \LL$.
Let $\LL_{ij} = \LL \cap \bbar{D_{ij}}$,
and let $\partial R_{ij} = \LL_{ij} \cap \partial R$ denote the
collection of edges of $R$ in $\bbar{D_{ij}}$. Further, let $\gamma_{ij} \in \Gamma$ be such that $\gamma_{ij}.D_{ij} =D$, so $\gamma_{ij}. \LL_{ij}=\LL_D$. With this notation, we have:

\begin{lemma}\label{barrier_edge_lem}
\noindent \begin{enumerate}
\item Every ideal vertex of $D$ belongs to the closure $\overline{\LL_D}$ of $\LL_D$ in $\C$. Moreover, $\overline{\LL_D}$ is connected.

\item $\partial R \subset \bigcup_{\substack{i=1, \cdots, k\\ j= 1, \cdots, k-1}} \bbar{D_{ij}}$; i.e., $\partial R=\bigcup_{\substack{i=1, \cdots, k\\ j= 1, \cdots, k-1}}\partial R_{ij}$.
\end{enumerate}
\end{lemma}
\begin{proof}
1) It follows from the proof of Proposition~\ref{prop-goah-dichotomy} that  for every ideal vertex $w$ of $D$ and ideal vertex $u$ of $R$ adjacent to $w$, either $\hA(\overline{w u})$ or $\hA^{\circ 2}(\overline{w u})$ is a leaf of $\LL_D$. Since $A(w)=w$, this leaf of $\LL_D$ has $w$ as one of its endpoints. This shows that every ideal vertex of $D$ belongs to the closure $\overline{\LL_D}$ of $\LL_D$ in $\C$. 

Recall that each leaf of $\LL_D$ is an iterated forward image of some edge of $R$ under $\hA$. Hence, to prove connectedness of $\overline{\LL_D}$, it suffices to argue that if $\alpha_1$ and $\alpha_2$ are two distinct adjacent edges of $R$, then the leaves of $\LL_D$ on which they eventually land have at least one common endpoint. To this end, assume that $\hA^{\circ n_i}(\alpha_i)=\delta_i\in\LL_D$, for some positive integers $n_1, n_2$. If $n_1=n_2$, then the desired conclusion follows from continuity of $\hA$. Else, we may assume without loss of generality that $n_1>n_2$. Then by continuity of $\hA$, $\hA^{\circ n_2}(\alpha_1)$ is an edge of $R$ having one endpoint in common with $\delta_2=\hA^{\circ n_2}(\alpha_2)$. Finally, since the endpoints of $\delta_2$ are ideal vertices of $D$, the arguments used in the previous paragraph show that $n_1\in\{n_2+1, n_2+2\}$ and $\delta_1=\hA^{\circ n_1}(\alpha_1)$ has (at least) one endpoint in common with $\delta_2$. In Figure \ref{barrier_fig},  an a priori possible lamination $\LL_D=\{\overline{w_2 w_3}, \overline{w_4 w_1}, \overline{w_1 w_3}\}$ is given along with its translates $\LL_{41}, \LL_{42}, \LL_{43}$. In this case, $\partial R_{4j}$ ($j=1,3$) consists of an edge and a diagonal of $\overline{D_{4j}}$; while $\partial R_{42}$ has a unique edge, namely $\partial U_4\cap \partial D_{42}$.

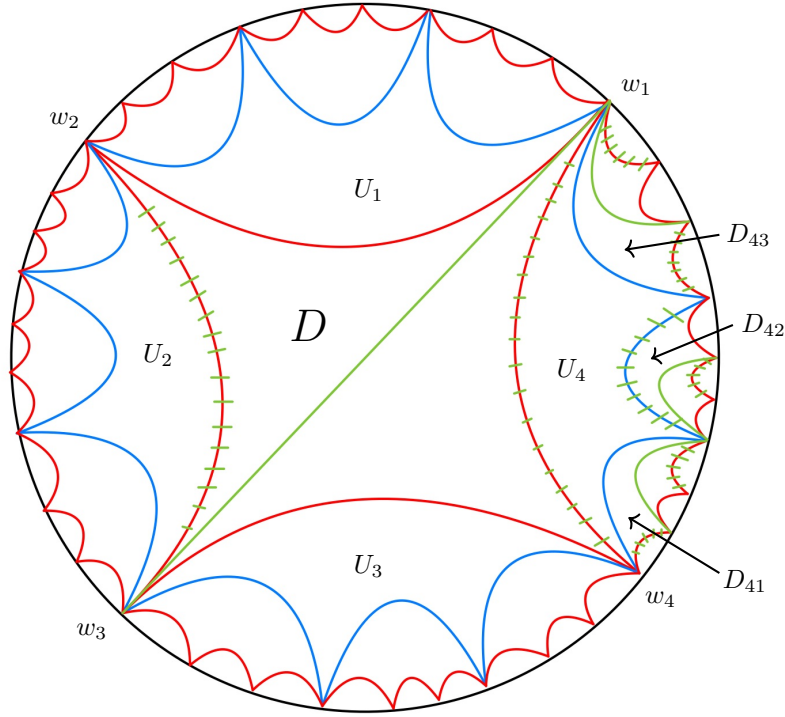
\begin{figure}[ht!]
	\begin{tikzpicture}
		\node[anchor=south west,inner sep=0] at (0,0) {\includegraphics[width=0.8\linewidth]{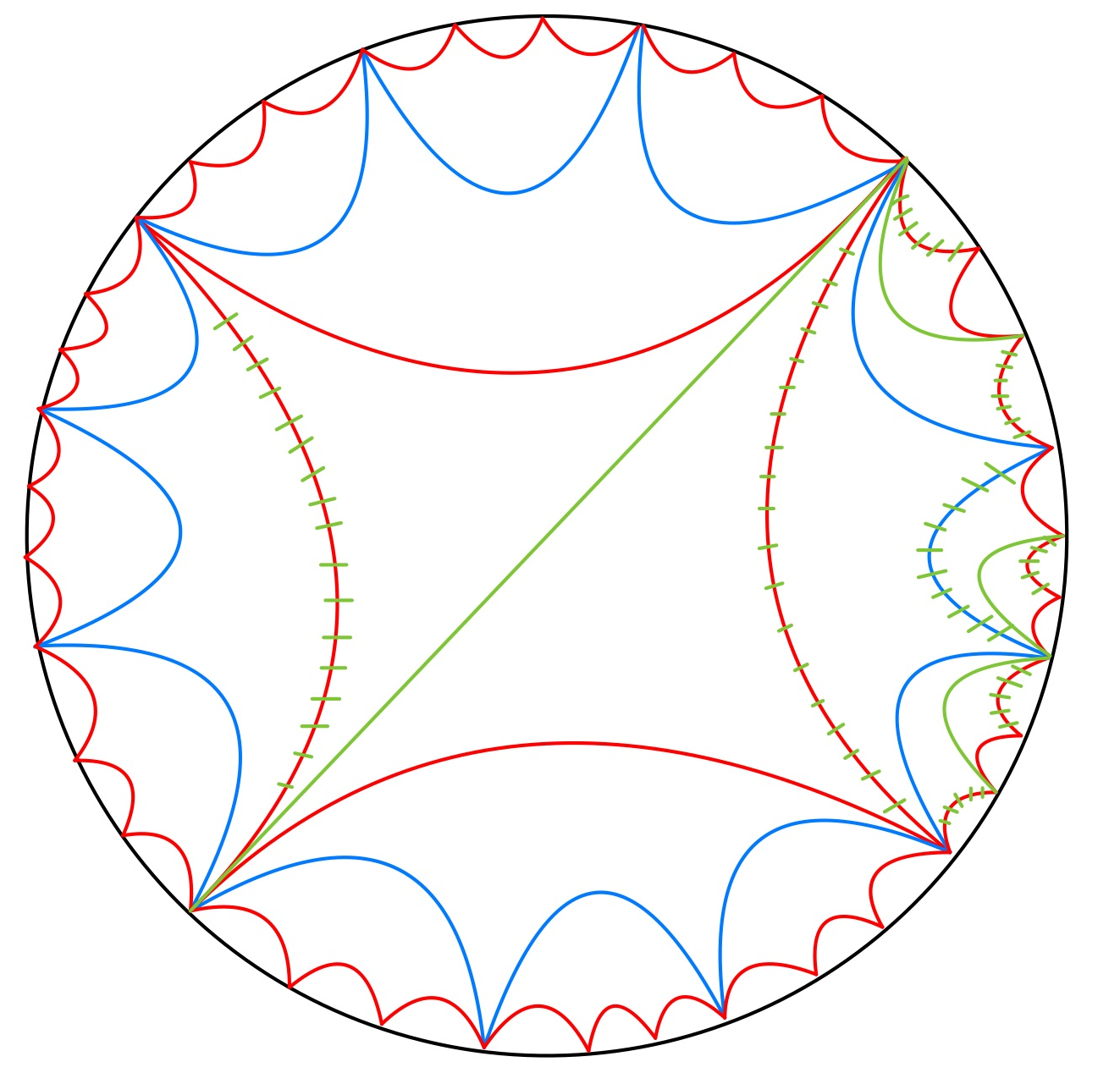}};
		\node at (4.2,5.4) {\begin{huge}$D$\end{huge}};
		\node at (5,7.2) {$U_1$};
		\node at (2.2,5) {$U_2$};
		\node at (5,2.2) {$U_3$};
		\node at (7.7,4.8) {$U_4$};
		\node at (8.56,8.56) {$w_1$};
		\node at (0.96,8.12) {$w_2$};
		\node at (1.32,1.32) {$w_3$};
		\node at (8.88,1.75) {$w_4$};
		\node at (10,2) {$D_{41}$};
		\draw[->,line width=0.8pt] (9.66,2.1)->(8.45,2.84);
		\node at (10.25,5.4) {$D_{42}$};
		\draw[->,line width=0.8pt] (9.84,5.4)->(8.75,4.95);
		\node at (10.05,6.6) {$D_{43}$};
	        \draw[->,line width=0.8pt] (9.66,6.6)->(8.4,6.4);
\end{tikzpicture}
\caption{Tiles $D, U_i$, and $D_{ij}$} 
\label{barrier_fig}
\end{figure}

2) First note that by Proposition~\ref{prop-goah-dichotomy}, the closure $\overline{D}$ of $D$ in $\D$ is contained in $R$. We claim that $\cup_{i=1}^k U_i$ is also contained in $R$. If this were not true, then some edge $\alpha$ of $R$ would intersect some $U_i$. Let $\delta$ be the leaf of $\LL_D$ on which $\alpha$ eventually falls under $\hA$. But then, $\delta\in\goah(\alpha)=\Gamma.\alpha$, which is impossible since there is no element of $\Gamma$ carrying a point of $U_i$ to a point of $\overline{D}$ (recall that $\Pi(\overline{D})$ and $\mathbf{U}=\Pi(U_i)$ are disjoint on the surface $\Sigma$). Hence, $\overline{D}\cup \cup_{i=1}^k U_i\subset R$. 

Now pick a leaf $\ell$ of $\LL_D$, and set $\ell_{ij}=\gamma_{ij}^{-1}.\ell\in\LL_{ij}$. Then, both $\ell$ and $\ell_{ij}$ lie in the $\Gamma$-orbit of some edge of $R$. By edge-orbit equivalence, $\ell$ and $\ell_{ij}$ lie in the same grand orbit under $\hA$. But since $\hA$ is not defined on $\ell$ (as $\ell\subset\overline{D}\subset  R$), $\ell_{ij}$ must be contained in the domain of definition of $\hA$; i.e., $\ell_{ij}\cap R=\emptyset$. Thus, $\LL_{ij}\cap R=\emptyset$. Combining this with the conclusion of the previous paragraph, we see that for $i\in\{1, \cdots, k\}, j\in\{1, \cdots, k-1\}$, there must be edges of $\partial R_{ij}$ (i.e., edges of $R$ contained in $\overline{D_{ij}}$) that act as barriers for $R$ to grow beyond $\LL_{ij}$. Note that since $\partial R_{ij} = \LL_{ij} \cap \partial R$, these edges must be some leaves of $\LL_{ij}$ themselves (see Figure~\ref{barrier_fig}). We call such edges of $\partial R_{ij}$ \emph{barrier} edges.

The properties of $\LL_D$ established in the first part of this proposition imply that every ideal vertex of $D_{ij}$ belongs to the closure $\overline{\LL_{ij}}$ of $\LL_{ij}$ (in $\C$), and $\overline{\LL_{ij}}$ is connected. Hence, for each $i, j$ as above, the closure (in $\C$) of the union of the barrier edges of $\partial R_{ij}$ must be homeomorphic to the closed interval $\left[0,1\right]$ with endpoints at the ideal vertices of $\partial U_i\cap \partial D_{ij}$ (see Figure~\ref{barrier_fig}). Since $\partial R$ is a polygon, it follows that every edge of $\partial R_{ij}$ is a barrier edge, and $R$ has no edge apart from these. In other words, $\partial R=\bigcup_{\substack{i=1, \cdots, k\\ j= 1, \cdots, k-1}}\partial R_{ij}$.
\end{proof}

\begin{rmk}\label{unique_element_rmk}
Since $\Gamma$ is torsion-free by our assumption and all break-points of $A$
are parabolics (and hence stabilizers in $\Gamma$ of all edges and diagonals of $R$ are trivial), if $\gamma. \LL_{ij} \cap \LL_D \neq \emptyset$,
then $\gamma=\gamma_{ij}$. Note that without the torsion-free assumption, the stabilizer of an edge/diagonal of $R$ may be an order two subgroup of $\Gamma$. This would add a technical expository complication that we want to avoid.
\end{rmk}

Note that there exists a unique edge of
$D_{ij}$ separating the interior of $D_{ij}$ from $D$. These edges are obtained as  the images of the edges of $\partial D$
under $\gamma_{ij}^{-1}$. 
Call these edges the \emph{near edges} of
$D_{ij}$. The vertices of a near edge of $D_{ij}$ will be called 
the \emph{near vertices} of $D_{ij}$.

\begin{rmk}\label{rmk-rij}
	A priori, we do not know from Proposition \ref{prop-goah-dichotomy} that near edges belong to $\LL$ (this is equivalent to saying that $\partial R_{ij}$ has a unique edge, and it is the near edge of $D_{ij}$). However, in light of the proof of Proposition \ref{prop-goah-dichotomy} and Lemma~\ref{barrier_edge_lem}, we do have the following dichotomy for any $\partial R_{ij}$:\\
\noindent $\bullet$ Either $\partial R_{ij}$ consists of a single edge, in which case the edge is a near edge.\\
\noindent $\bullet$ Else, $\partial R_{ij}$ consists of a chain of edges 
		$\beta_1, \cdots, \beta_r$, $2\leq r \leq k-1$ such that the initial vertex of $\beta_1$ and the final vertex of $\beta_r$
		coincide with the near vertices of $D_{ij}$. Further, the final vertex
		of $\beta_i$ equals the initial vertex of $\beta_{i+1}$ and
		the union of the $\beta_i$'s along with their ideal endpoints
		gives an embedding of the closed interval $[0,1]$ into
		$\overline{\D}$. In this case, we refer to the edges in $\partial R_{ij}$
		as the \emph{far edges} of $R$ in
		$\overline{D_{ij}}$. 
\end{rmk}

We shall refer to
diagonals of $R$ not contained in $\bbar D$ as \emph{bad diagonals}.
Proposition \ref{prop-goah-dichotomy} then shows that $\goah (\alpha)$ does not contain bad diagonals if $\alpha$ is an edge of $R$. Also, the union of the (closed) tiles $\bigcup_{j=1, \cdots, k-1} \overline{D_{ij}}$ will be called the $i$-th sector of $A$ and will be denoted
as $S_i$.

\begin{theorem}\label{thm-char-hdm-eoe}
	Suppose $A$ is a minimal
	mateable map without folding such that $\hA$ is edge-orbit equivalent to the Fuchsian group $\Gamma$ generated by its pieces, and $\hA(\partial R)\neq \partial R$. Then 
	$A$ is a higher Bowen-Series map.
\end{theorem}

\begin{proof}
	The proof is in two steps. 
	
	\noindent {\bf Step 1: Sector with least number of edges of $R$ maps to $\overline{D}$:} Without loss of generality (by renumbering the vertices if necessary), we can assume that $S_1$ is the sector of 
	$A$ with the smallest number of edges of $R$. Let $\alpha_1, \cdots, \alpha_s$ denote the edges of $R$ in $S_1$ ordered in a counterclockwise sense.
	We claim that 
	\begin{enumerate}
		\item for every edge $\alpha$ of $R$ in $S_1$, $\hA(\alpha)$ is an edge of $D$;

		\item further, each edge of $R$ in $S_1$ is a near edge.
	\end{enumerate}
	Suppose that the first claim is false. We continue with the notation introduced in Proposition \ref{prop-goah-dichotomy}. Assume that the ideal vertices $w_1, \cdots, w_k$ of $D$ are ordered in a counterclockwise sense, and $u_1,\cdots, u_{s-1}$ are the ideal vertices of $R$ (ordered in a counterclockwise sense) between $w_1$ and $w_2$ such that $\alpha_1=\overline{w_1 u_1}$ and $\alpha_s=\overline{u_{s-1} w_2}$. The proof of Proposition \ref{prop-goah-dichotomy} shows that $A(u_1)$ (respectively, $A(u_{s-1})$) is either an ideal vertex of $D$ or an ideal vertex of $R$ adjacent to $w_1$ (respectively, $w_2$) such that $\overline{w_1 A(u_1)}$ lies in $S_k$ (respectively, $\overline{w_2 A(u_{s-1})}$ lies in $S_2$).
	
\noindent \textbf{Case 1: At least one of $A(u_1)$ and $A(u_{s-1})$ is an ideal vertex of $D$.} By symmetry of the situation, we can assume without loss of generality that $A(u_1)$ is an ideal vertex of $D$; i.e., $\alpha_1$ maps to $\overline{D}$ under $\hA$. If some edge $\alpha$ of $R$ in $S_1$ does not map into $\bbar D$,
	then by Proposition \ref{prop-goah-dichotomy}, it must map to
	an edge of $R$. Let $S_i$ denote the sector of $R$ containing
	$\hA(\alpha)$. Since $A(w_2)=w_2$ and there are no bad diagonals in $\hA$-grand orbits of edges of $R$ by Proposition \ref{prop-goah-dichotomy} (as observed just before Theorem \ref{thm-char-hdm-eoe}), $S_i$ can contain at most $s-1$ edges of $R$, contradicting the assumption that $S_1$ is the sector of 
	$A$ with the smallest number of edges of $R$. Hence, every edge  of $R$
	in $S_1$ maps to $\bbar D$.

\noindent \textbf{Case 2: $A(u_1), A(u_{s-1})$ are ideal vertices of $R$ adjacent to $w_1, w_2$.} In this case, $\overline{w_1 A(u_1)}$ lies in $S_k$, and $\overline{w_2 A(u_{s-1})}$ lies in $S_2$. Once again, absence of bad diagonals in $\hA$-grand orbits of edges of $R$ and the fact that $A(w_m)=w_m$ ($m=1,2$) imply that each of the sectors $S_2,S_k$ can contain at most $s-2$ edges of $R$, contradicting the assumption that $S_1$ is the sector of 
	$A$ with the smallest number of edges of $R$. Hence, this case cannot arise. The proof of the first claim is now complete.
	
	It remains to prove that each $\overline{D_{1j}}$ contains exactly one edge of $R$. If this were not true, since all edges in $\partial R_{1j}$ map to $\bbar D$ under $\hA$ (and hence under some element of $\Gamma$), it would follow from Remark~\ref{unique_element_rmk} that $\hA$ must act as $\gamma_{1j}$ on all edges of $\partial R_{1j}$. But this would contradict minimality of $A$. This forces $\partial R_{1j}$ to consist of a single edge; namely, the near edge of $D_{1j}$. Further, $\hA$ acts as $\gamma_{1j}$ on the connected component of $\D\setminus R$ containing $D_{1j}$, and it maps the unique edge of $\partial R_{1j}$ (i.e., the near edge of $D_{1j}$) to a boundary edge of $D$. In particular, on $\partial R\cap S_1$, $\hA$ equals a higher degree map without folding.
	\smallskip
	
	\noindent {\bf Step 2: Treating the possible exceptional edge of $D$:} 
	By the proof of Step 1, the only edge of $D$ that might not belong to $\LL_D$ is $\overline{w_1 w_2}$. 
\smallskip

\noindent \textbf{Case 1: $\overline{w_1 w_2}\in \LL_D$.} Then, for each $i\in\{1, \cdots, k\}, j\in\{1, \cdots, k-1\}$, $\partial R_{ij}$ consists of a single edge; namely, the near edge of $D_{ij}$. Consequently, $R=D\cup\cup_{i=1}^k U_i$. Therefore, each sector of $A$ contains exactly $k-1$ edges of $R$. It now follows from the conclusion of Step 1 that $\hA$ maps every edge of $R$ to a boundary edge of $D$. In particular, $\LL_D$ consists precisely of the boundary edges of $D$. Thus, $A$ is a higher degree map without folding with fundamental domain $R=D\cup\cup_{i=1}^k U_i$ and inner domain $D$. Moreover, $A$ is injective on the union of the $k-1$ near edges of $R$ in each sector. Proposition \ref{higher_bs_char_prop} now shows that $A$ is a higher Bowen-Series map.   
\smallskip

\noindent \textbf{Case 2: $\overline{w_1 w_2}\notin \LL_D$.} We shall show that this case cannot arise. 

 In this case, $\overline{w_1 w_2}$ does not lie in the image of $\partial R$ under$\hA$ or its iterates. Then $R$ has far edges in $\overline{D_{ij}}$ if and only if the near edge of $D_{ij}$ is a $\Gamma$-translate of $\overline{w_1 w_2}$ (see Remark \ref{rmk-rij}). For the rest of the proof, we shall refer to such tiles $D_{ij}$ as \emph{far tiles}. For any such tile $D_{ij}$, $\partial R_{ij}$ can have at most $k-1$ edges. Since $\TT$ is a checkerboard tiling, $D_{21}\subset S_2$ is a far tile sharing the ideal vertex $w_2$ with the tile $D$. Let $\beta\subset\overline{D_{21}}$ be the edge of $R$ having an endpoint at $w_2$.
	
	By Proposition \ref{prop-goah-dichotomy}, $\hA(\beta)$ is either a diagonal of $R$ contained in $\bbar{D}$ or an edge of $R$ having an endpoint at $w_1$. We need to analyze these two cases separately.

\noindent \textbf{Subcase i: $\hA(\beta)\subset\overline{D}$.}
	 Since there are no bad diagonals in $\hA$-grand orbits of edges of $R$,
	this forces all the far edges of $R$ in $\overline{D_{21}}$ to map into $\bbar{D}$ under $\hA$. But then all the pieces of $A$ corresponding to the
	far edges of $R$ in $\overline{D_{21}}$ must coincide with the group element $\gamma_{21}$
	(see the discussion in the paragraph on `Checkerboard Tiling' and Remark~\ref{unique_element_rmk}). But this contradicts minimality of $A$.

\noindent \textbf{Subcase ii: $\hA(\beta)$ is an edge of $R$ having an endpoint at $w_2$.} Since the edges of $R$ in $S_1$ are precisely the near edges of $D_{1j}$ ($j\in\{1,\cdots,k-1\}$), absence of bad diagonals in $\hA$-grand orbits of edges of $R$ implies that $\partial R_{21}$ must consist of all $k-1$ edges of $D_{21}$ except the near edge. It now follows that $\LL_D$ does not contain any diagonal of $D$; otherwise, $\LL_{21}$ would contain a diagonal of $D_{21}$, forcing $\partial R_{21}$ to contain a diagonal of $D_{21}$ (see the proof of part (2) of Lemma~\ref{barrier_edge_lem}). Hence, $\LL_D$ consists of  all the edges of $D$ except $\overline{w_1 w_2}$. This yields the following description of $\partial R$:
\begin{itemize}
\item \textbf{type 1:} if the near edge of $D_{ij}$ is not a $\Gamma$-translate of $\overline{w_1 w_2}$, then this edge alone comprises $\partial R_{ij}$;

\item \textbf{type 2:} otherwise, $\partial R_{ij}$ consists of all $k-1$ edges of $D_{ij}$ except the near edge.
\end{itemize} 
Using a counting argument similar to the one used in Step 1 combined with the structure of the checkerboard tiling $\mathcal{T}$ and the fact that $\hA$ has no folding, one now sees that each type 1 edge of $R$ maps to a leaf of $\LL_D$ under $\hA$, while each type 2 edge maps to an edge of $R$ in $S_1$. 

Finally, in light of Remark~\ref{unique_element_rmk}, the checkerboard tiling $\mathcal{T}$ allows one to express all the group elements $\gamma_{ij}$ in terms of the elements $\gamma_{1j}$ ($i\in\{1,\cdots,k\}$, $j\in\{1,\cdots,k-1\}$). This, in turn, shows that $A$ agrees with the Bowen-Series map $A_{\Gamma, \mathrm{BS}}$ associated with the fundamental domain $\bbar{D}\cup U_1$ on the counterclockwise arc of $\bS^1$ from $w_1$ to $w_2$, and with  $A_{\Gamma, \mathrm{BS}}^{\circ 2}$ on the rest of $\bS^1$. According to Proposition~\ref{not_oe_prop}, the map $A$ is not orbit equivalent to $\Gamma$; i.e., $A$ is not a mateable map, contradicting our hypothesis. 
	We conclude that Case 2 is impossible, so $\overline{w_1 w_2}$ must be a leaf of $\LL_D$. The proof is now complete by Case 1.
\end{proof}

\begin{qn}\label{q-minimal} Does the 
	hypothesis that $A$ has  `no folding' in Proposition \ref{prop-goah-dichotomy}, Corollary \ref{cor-numberofpunct} and
	Theorem \ref{thm-char-hdm-eoe} follow from minimality of $A$?
	The only examples that we know of mateable maps with folding are non-minimal \cite[\S 4.1]{sullivan_survey}. Note however that for a general \pwfm map (not necessarily mateable), minimality does not preclude the existence of diagonal folds; see \cite[\S 4.4.2]{sullivan_survey}.
\end{qn}

\subsection{Moduli of matings}\label{sec-moduli} We shall now assemble the results of the previous sections to determine the moduli space of matings.

\noindent {\bf Normalization of gluing map:} Recall from Proposition~\ref{conformal_mating_general_prop} that mateable maps orbit equivalent to Fuchsian groups can be conformally mated with polynomials lying in principal hyperbolic components of appropriate degree. In fact, the proof of the proposition reveals that there are finitely many choices to glue the dynamics of a mateable map $A$ with that of a polynomial $P$ in a principal hyperbolic component. Since such a gluing map conjugates $A\vert_{\bS^1}$ to $P\vert_{\mathcal{J}(P)}$, it can be uniquely described by marking a fixed point of $A$ and $P$ each; more precisely, by stipulating that a marked fixed point of $A$ on $\bS^1$ is mapped under the gluing map to a marked fixed point of $P$ on $\mathcal{J}(P)$. We will now introduce a canonical marking of such fixed points. 

Let $\Gamma_0$ be a Fuchsian group, and $A_{\Gamma_0}$ be a mateable map orbit equivalent to $\Gamma_0$. Possibly after conjugating $\Gamma_0$ and $A_{\Gamma_0}$ by an element of $\mathrm{Aut}(\D)$, we can assume that $1$ is a fixed point of the mateable map $A_{\Gamma_0}$ (orbit equivalent to $\Gamma_0$ on $\bS^1$). For each $\Gamma\in\mathcal{B}(\Gamma_0)$, the associated mateable map $A_\Gamma$ compatible with $A_{\Gamma_0}$ is given by $\phi_\Gamma\circ A_{\Gamma_0}\circ\phi_\Gamma^{-1}$, where $\phi_\Gamma$ is a quasiconformal homeomorphism inducing the representation $\Gamma_0\to\Gamma$. This determines a marked fixed point $\phi_\Gamma(1)$ of $A_\Gamma$. 
 
On the other hand, the polynomial $P_0:z\mapsto z^d$ has a fixed point at $1$. As the hyperbolic component $\mathcal{H}_d$ is simply component (see \cite[\S 5]{milnor-hyp-comp}, for instance), we can use \cite[Theorem~8.1]{douady-julia-cont} to find a unique family of homeomorphisms 
$\lbrace\tau_P:\mathcal{J}(P_0)\to\mathcal{J}(P)\rbrace_{P\in\mathcal{H}_d}$
such that $\tau_{P_0}$ is the identity map, $(P,z)\mapsto \tau_P(z)$ is continuous, and $\tau_P$ conjugates $P_0$ to $P$. This yields a naturally marked fixed point $\tau_P(1)$ of $P$. 

Thus, Proposition~\ref{conformal_mating_general_prop} provides us with a canonical conformal mating between $A_\Gamma$ and $P$ once we prescribe that the gluing conjugacy between $A_\Gamma\vert_{\Lambda(\Gamma)}$ and $P\vert_{\mathcal{J}(P)}$ carries $\phi_\Gamma(1)$  to $\tau_P(1)$. We shall henceforth assume that
such a \emph{normalized gluing map} has been chosen.

\begin{defn}\label{defn-moduli}
	The \emph{moduli space of matings} between a topological surface $\Sigma$ and complex polynomials in principal hyperbolic components consists of equivalence classes of triples $(\Gamma, A_\Gamma, P)$, where
\begin{enumerate}
\item $\Gamma\in\mathcal{B}(\Gamma_0)$. Here $\Gamma_0$ is a Fuchsian group with $\D/\Gamma_0\cong\Sigma$,

\item $A_\Gamma$ is a mateable map associated to $\Gamma$ compatible with a minimal mateable map $A_{\Gamma_0}$ orbit equivalent to $\Gamma_0$, and

\item $P$ is a polynomial in a principal hyperbolic component with $\mathrm{deg}(P)=\mathrm{deg}(A_{\Gamma_0}:\bS^1\to\bS^1)$.
\end{enumerate}
Two such triples $(\Gamma_1, A_{\Gamma_1}, P_1)$ and $(\Gamma_2, A_{\Gamma_2}, P_2)$ are said to be \emph{equivalent} if the conformal matings of $A_{\Gamma_i}$ and $P_i$ ($i=1,2$) are M{\"o}bius conjugate.

	If, in addition, $A_{\Gamma_0}$ is backward edge-orbit equivalent to $\Gamma_0$, the corresponding
	space of triples $(\Gamma, A_\Gamma, P)$ is called the \emph{backward edge-orbit equivalence moduli space of matings}.
	
	If $A_{\Gamma_0}$ is  edge-orbit equivalent to $\Gamma_0$, the corresponding
	space of triples $(\Gamma, A_\Gamma, P)$ is called the \emph{edge-orbit equivalence moduli space of matings}. Further, if $A_{\Gamma_0}$ has no folding, then the corresponding space of triples $(\Gamma, A_\Gamma, P)$ is called the \emph{unfolded edge-orbit equivalence moduli space of matings}.
\end{defn}

\begin{theorem}\label{thm-moduli-eoe} Suppose $\Sigma$ has no orbifold points.
	The unfolded 
	edge-orbit equivalence moduli space of matings, and also the 
	backward edge-orbit equivalence moduli space of matings
	is non-empty if and only if $\Sigma$ is a sphere $S_{0,k}$ with 
	$k \geq 3$ punctures. The backward edge-orbit equivalence moduli space of matings has one connected component, while the unfolded 
	edge-orbit equivalence moduli space of matings has two.
\end{theorem}

\begin{proof}
	Theorem \ref{thm-aisbs} shows that every point in the backward edge-orbit equivalence moduli space of matings consists of triples $(\Gamma, A_\Gamma, P)$, where $\Gamma\in\mathcal{B}(\Gamma_0)$ such that $\Gamma_0$ is a Fuchsian group uniformizing a sphere $S_{0,k}$ (with $k \geq 3$ punctures), $A_\Gamma$ is the Bowen-Series map for $\Gamma$, and $P$ is a polynomial in the principal hyperbolic component of degree $1-2\chi (S_{0,k})$.

	Theorem \ref{thm-char-hdm-eoe} shows that  the unfolded  edge-orbit equivalence moduli space of matings consists of two components. One component coincides with that of the backward edge-orbit equivalence moduli space of matings. The other consists of triples $(\Gamma, A_\Gamma, P)$, where $\Gamma\in\mathcal{B}(\Gamma_0)$ such that $\Gamma_0$ is a Fuchsian group uniformizing a sphere $S_{0,k}$ (with $k \geq 3$ punctures), $A_\Gamma$ is the higher Bowen-Series map for $\Gamma$, and $P$ is a polynomial in the principal hyperbolic component of degree $(\chi (S_{0,k})-1)^2$ (by Equation \eqref{eq-degcfm}).
\end{proof}

We end this section by explicating the general questions that Theorem \ref{thm-moduli-eoe} addresses:

\begin{qn}\label{qn-moduli}
	1) Determine the (unconstrained) moduli space of matings between a topological surface $\Sigma$ and complex polynomials. \\
	2) It is tempting to call the numbers $1-2\chi (S_{0,k})$ and $(\chi (S_{0,k})-1)^2$ the \emph{Euler numbers} of the Bowen-Series map and the higher 
	Bowen-Series map respectively, in the expectation that they have a topological significance beyond the degree of $A_\Gamma$. However, a proper understanding
	of the spectrum of Euler numbers of mateable maps corresponding to $S_{0,k}$ requires a complete enumeration of the components of the (unconstrained) moduli space of matings and takes us back to the first part of this question.
\end{qn}

\section{Pinching laminations and mateability of Bers boundary groups}\label{limit_julia_homeo_sec}

We now turn our attention to groups on the Bers boundary of the Fuchsian group $G_d$ uniformizing a sphere with $d+1$ punctures (see Subsection~\ref{b_s_punc_sphere_subsec}). Our first goal is to answer the first half of Question~\ref{problem_1} for such groups; i.e., to isolate groups on $\partial\mathcal{B}(G_d)$ whose limit sets are homeomorphic to Julia sets of suitable complex polynomials in a dynamically natural way. To give a precise meaning to the word `dynamically natural', we first observe that if the limit set of some group $\Gamma\in\partial\mathcal{B}(G_d)$ is homeomorphic to the Julia set of some polynomial $P$, then such a homeomorphism would conjugate $P\vert_{\mathcal{J}(P)}$ to a continuous self-map of the limit set $\Lambda(\Gamma)$. We say that a homeomorphism $\pmb{\Phi}:\mathcal{J}(P)\rightarrow\Lambda(\Gamma)$ is \emph{dynamically natural} if $\pmb{\Phi}\circ P\circ\pmb{\Phi}^{-1}:\Lambda(\Gamma)\to\Lambda(\Gamma)$ is orbit equivalent to $\Gamma$. 
This raises the following:

\begin{qn}\label{problem_bers_boundary_map}
When does the limit set of a group $\Gamma$, lying on the boundary of $\mathcal{B}(G_d)$, admit a continuous self-map that is orbit equivalent to $\Gamma$?
\end{qn} 

While we do not know the answer to the above question in full generality, we will address the following special case of  Question~\ref{problem_bers_boundary_map}.

\begin{qn}\label{problem_bers_boundary_mateable_map}
Fix a mateable map $A_{G_d}:\bS^1\to\bS^1$ orbit equivalent to $G_d$ (e.g., $A_{G_d,\mathrm{BS}}, A_{G_d,\mathrm{hBS}}$). Which groups $\Gamma\in\partial\mathcal{B}(G_d)$ admit a mateable map $A_\Gamma:\Lambda(\Gamma)\rightarrow\Lambda(\Gamma)$, orbit equivalent to $\Gamma$ and compatible with $A_{G_d}$ (in the sense of Subsection~\ref{mateable_gen_subsec})?
\end{qn}

\subsection{Cannon-Thurston maps}\label{geod_lami_subsec}

 Pick a group $\Gamma\in\partial\mathcal{B}(G_d)$. According to Theorem~\ref{ctsurf}, there exists a continuous map from $\mathbb{S}^1$ onto the limit set of $\Gamma$, called the \emph{Cannon-Thurston map} after \cite{CTpub}, that semi-conjugates the action of $G_d$ to that of $\Gamma$.
The second part of Theorem 
\ref{ctsurf} shows that the data of the ending lamination can be recovered from the Cannon-Thurston map. (The existence of Cannon-Thurston maps
 for surface groups without accidental parabolics was proved in \cite{mahan-split} and their structure in terms of ending laminations in \cite{mahan-elct}. This was extended to the case with accidental parabolics in \cite{mahan-red}. The case of general Kleinian groups is treated in \cite{mahan-kl}.)
 
 \begin{figure}[ht]

	\includegraphics[height=3cm]{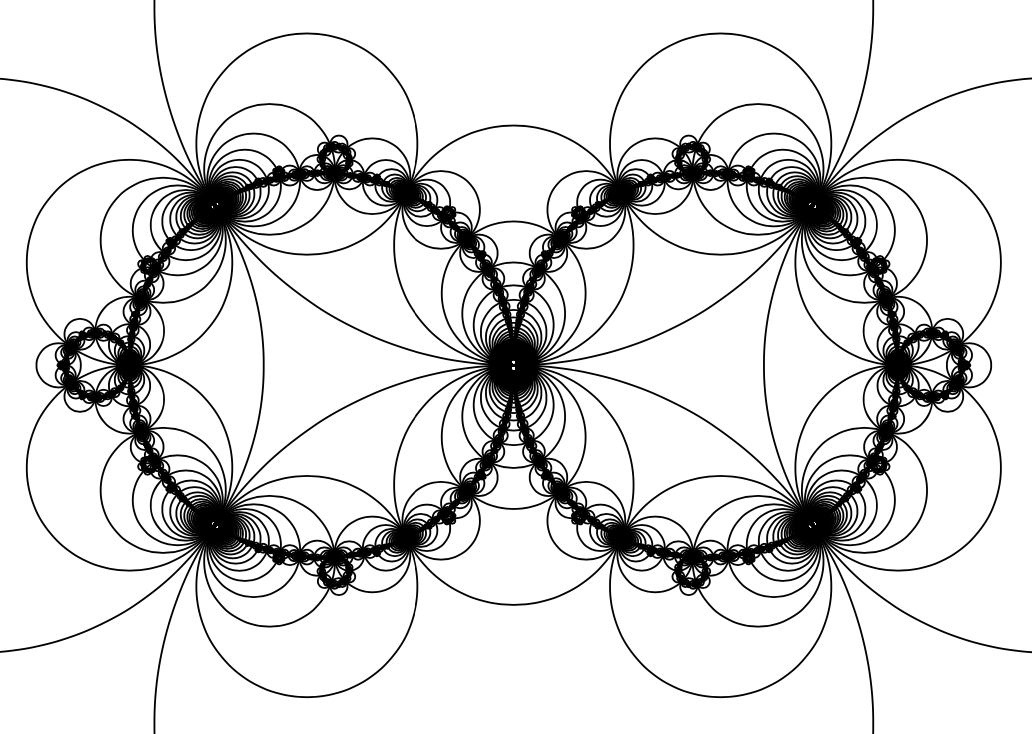}
	
\caption{Limit set of a Bers boundary group}
	\label{fig-bersbdy}
\end{figure}

  For $\Gamma \subset   \pslc$  a  Kleinian surface  group, and $M,\Sigma,i$ as in 
Theorem \ref{ctsurf}, $M \setminus i(\Sigma)$ consists of two components $M^+, M^-$ say
(assuming without loss of generality that $i$ is an embedding). Then each of 
$M^+, M^-$ furnishes end-invariants $\EE^+, \EE^-$ \cite[Chapter 9]{thurstonnotes},
consisting of conformal structures at infinity $\CC(M)$ and ending laminations $\EE\LL(M)$.
The celebrated Ending Lamination Theorem of Brock-Canary-Minsky \cite{minsky-elc2} establishes that a Kleinian surface group is determined uniquely by its end-invariants. The case $\EE\LL(M)=\emptyset$ corresponds precisely to quasi-Fuchsian groups. The Bers slice $\BB(\tau)$ in quasi-Fuchsian space corresponding to a conformal structure $\tau$ consists of the subset where $\EE^+=\tau$. More generally, the collection of\emph{ all} Kleinian surface groups with 
$\EE^+=\tau$ gives the closure $\bbar{\BB(\tau)}$ of the Bers slice \cite{minsky-elc2}. The case where  $\EE\LL(M)$ consists only of accidental parabolics  corresponds precisely to geometrically finite groups with $i$ weakly type-preserving, i.e.,  $i$ maps parabolics to parabolics. Since there can only be finitely many disjoint simple closed curves on a surface, geometrically finite Kleinian surface groups are precisely those where 
$\EE\LL(M)$ is finite.  The subcollection with $\EE^+=\tau$ gives the geometrically
finite groups on the boundary of the Bers slice $\BB(\tau)$ corresponding to a conformal structure $\tau$. Suppose $\sigma \subset \Sigma$ is a simple closed curve representing an accidental parabolic in $\Gamma\in \partial \BB(\tau)$. Then, by the Bers density theorem of \cite{minsky-elc2}, there exists a sequence of groups in $\BB(\tau)$ of the form $(\tau, \tau_n)$ converging to $\Gamma$ such that the length of $\sigma $ in the hyperbolic structure corresponding to $\tau_n$ converges to zero as $n\to \infty$.
We say informally that
the group $\Gamma$ is obtained by pinching the support  of a measured lamination on the surface $\Sigma$ (while keeping $\tau$ unchanged). When $\EE\LL(M)$ consists of infinitely many leaves, we obtain \emph{degenerate}
Kleinian surface groups \cite[Chapter 9]{thurstonnotes}: the surface
is thought of as \emph{degenerating} to a projectivized measured lamination on
the Thurston boundary of $\Teich(\Sigma)$. 
In all cases, Theorem \ref{ctsurf} says that the endpoints of the corresponding geodesic lamination $\mathcal{L}$ (on $\disk$) generate a $\pi_1(\Sigma)$-invariant equivalence relation on $\mathbb{S}^1$, and this equivalence relation agrees with the one defined by the fibers of the Cannon-Thurston map of $\Gamma$. 

\begin{defn}\label{b_s_invariance_def} 
We denote the equivalence relation on $\mathbb{S}^1$ generated by the endpoints of $\mathcal{L}$ by $\sim$. We say that $\mathcal{L}$ is \emph{invariant under a \pwfm map $A$} if $\sim$ is invariant under $A$; i.e., $x\sim y \implies A(x)\sim A(y)$ for $x, y\in\mathbb{S}^1$.
\end{defn}

\begin{lemma}\label{pwfm_invariant_lami_lem}
Let $A:\bS^1\to\bS^1$ be a \pwfm map orbit equivalent to $G_d$, and $\Gamma\in\partial\mathcal{B}(G_d)$. Then, the Cannon-Thurston map of $\Gamma$ semi-conjugates $A$ to a continuous map $A_{\Gamma}:\Lambda(\Gamma)\to\Lambda(\Gamma)$ that is orbit equivalent to $\Gamma$ if and only if the associated lamination $\mathcal{L}$ is invariant under $A$.
\end{lemma}

\begin{center}
\begin{tikzcd}
\mathbb{S}^1 \arrow{r}{A} \arrow[swap]{d}{\mathrm{C. T.}} & \mathbb{S}^1 \arrow{d}{\mathrm{C. T.}} \\
\Lambda(\Gamma) \arrow[swap]{r}{A_{\Gamma}}& \Lambda(\Gamma)
\end{tikzcd}
\end{center}

\begin{proof}
Since $\Lambda(\Gamma)\cong \mathbb{S}^1/\sim$, it is easy to see that invariance of $\mathcal{L}$ under the Bowen-Series map $A$ is a necessary and sufficient condition for the Cannon-Thurston map (of $\Gamma$) to semi-conjugate $A$ to a continuous self-map of $\Lambda(\Gamma)$. It remains to argue that if $\mathcal{L}$ is $A$-invariant, then the induced map $A_{\Gamma}$ is orbit equivalent to $\Gamma$ on $\Lambda(\Gamma)$.

To this end, let us suppose that $\mathcal{L}$ is $A$-invariant. Observe that since the Cannon-Thurston map of $\Gamma$ semi-conjugates $G_d$ to $\Gamma$, the map $A_{\Gamma}$ acts piecewise by elements of $\Gamma$. From this, it is easy to see that if $x, y\in\Lambda(\Gamma)$ lie in the same grand orbit of $A_{\Gamma}$, then there exists an element of $\Gamma$ that takes $x$ to $y$; i.e., $x$ and $y$ lie in the same $\Gamma$-orbit. Conversely, let $y=g(x)$, for some $g\in \Gamma$, and $x, y\in\Lambda(\Gamma)$. But this means that there exist $g'\in G_d$ and $x', y'\in\mathbb{S}^1$ such that the Cannon-Thurston map of $\Gamma$ sends $x', y'$ to $x, y$ (respectively), and $y'=g'(x')$. By orbit equivalence of $G_d$ and $A$, we have that $A^{\circ n}(x')=A^{\circ m}(y')$, for some $n, m\geq 0$. Since the Cannon-Thurston map of $\Gamma$ semi-conjugates $A$ to $A_{\Gamma}$, it now follows that $A_{\Gamma}^{\circ n}(x)=A_{\Gamma}^{\circ m}(y)$. This completes the proof.
\end{proof}

\begin{rmk}\label{rmk-lts}
The Bers density conjecture ensures that groups in the Bers boundary are strong limits of groups in the Bers slice. This conjecture is a consequence of the Ending Lamination Theorem \cite{minsky-elc2} of Brock-Canary-Minsky. Next, let $\Gamma_n=\phi_n\circ G_d\circ\phi_n^{-1}$ be a sequence of groups in $\BB(G_d)$ (where $\phi_n$ is a quasiconformal homeomorphism inducing the representation $G_d\to\Gamma_n$) converging strongly to
$\Gamma \in \partial \BB(G_d)$.
By work
of the first author and Series \cite{mahan-series1,mahan-series2}, Cannon-Thurston maps of $\Gamma_n$ converge uniformly to the 
Cannon-Thurston map of $\Gamma$. Thus, the map $A_{\Gamma}$ in Lemma \ref{pwfm_invariant_lami_lem} can, in fact, be regarded as a limit of the piecewise M{\"o}bius Markov maps $A_{\Gamma_n}=\phi_n\circ A\circ \phi_n^{-1}$ orbit equivalent to the quasi-Fuchsian groups $\Gamma_n$ on their limit sets.
\end{rmk}

\subsection{Invariant laminations under Bowen-Series and higher Bowen-Series maps}\label{lami_inv_subsec}

We will now carry out the task of classifying supports of measured laminations $\mathcal{L}^*$ (associated to groups on the Bers boundary $\partial\mathcal{B}(G_d)$) for which the $G_d$-lift $\mathcal{L}$ (to $\D$) is invariant under the Bowen-Series/ higher Bowen-Series map of $G_d$.

We start with a general statement that puts a severe constraint on such laminations. To this end, let $A:\bS^1\to\bS^1$ be a \pwfm map. The pieces of $A$ are $h_1,\cdots, h_k$ such that $A\vert_{I_j}=h_j$, where $\{I_1,\cdots, I_k\}$ is a partition of $\bS^1$ by closed arcs. The canonical domain of definition of $\widehat{A}$ in $\overline{\D}$ is denoted by $\mathcal{D}$, and the fundamental domain of $\widehat{A}$ is $R=\disk\setminus\mathcal{D}$. Throughout this subsection, $\mathcal{L}$ will stand for the $G_d$-lift (to $\D$) of the support of a measured lamination $\mathcal{L}^*$ on $\D/G_d$. 

\begin{lemma}\label{inv_lami_pwfm}
Let $A:\bS^1\to\bS^1$ be a \pwfm map such that $\overline{R}$ contains a fundamental domain of $G_d$. Then, $\mathcal{L}$ is invariant under $A$ only if $\mathcal{L}^*$ is a union of mutually disjoint, simple, closed, non-peripheral geodesics on $\disk/G_d$ that are represented by group elements in the finite set $\{h_j^{-1}\circ h_i: i,j\in\{1,\cdots, k\}, i\neq j\}$. In particular, any group $\Gamma\in\partial\mathcal{B}(G_d)$ obtained by pinching such a lamination $\mathcal{L}^*$ on the surface $\D/G_d$ is geometrically finite.
\end{lemma}
\begin{proof}
Pick $\gamma\in\mathcal{L}^*$, lift it to a fundamental domain of $G_d$ contained in $\overline{R}$, and choose a leaf $\ell$ of $\mathcal{L}$ that contains    a connected component of the lift. By construction, the ideal endpoints $x$ and $y$ of $\ell$ are contained in $I_i$, $I_j$ with $i, j\in\{1,\cdots, k\}$, $i\neq j$.

Since $x\in I_i$ and $y\in I_j$, we have that $A(x)=h_i(x)$ and $A(y)=h_j(y)$. Now, $A$-invariance of $\mathcal{L}$ implies that $h_i(x)\sim h_j(y)$, and $G_d$-invariance of $\mathcal{L}$ yields that $y\sim h_j^{-1}\circ h_i(x)$. Since $x \sim y$, group invariance  of $\mathcal{L}$ also implies that $h_j^{-1}\circ h_i(x)\sim h_j^{-1}\circ h_i(y)$. Hence, $y\sim h_j^{-1}\circ h_i(y)$.

Let $\rho: G_d \to \pslc$ be the discrete faithful representation with $\rho(G_d) =\Gamma$. Also, let $\phi_\Gamma: \bS^1 \to \Lambda(\Gamma)$ be the surjective Cannon-Thurston map in
Theorem \ref{ctsurf} identifying the ideal end-points of leaves of $\LL$. Note that $\Lambda(\Gamma)$ is an equivariant quotient of $\bS^1$ obtained  precisely from this identification by Theorem \ref{ctsurf}.

Since $y\sim h_j^{-1}\circ h_i(y)$, 
\begin{enumerate}
\item $\phi_\Gamma (y) = \phi_\Gamma \big(h_j^{-1}\circ h_i(y)\big) \in \Lambda(\Gamma)$,
\item $\phi_\Gamma (y) = \rho (h_j^{-1}\circ h_i) \cdot \phi_\Gamma (y)$,
\end{enumerate}
where the first equality follows from the structure of Cannon-Thurston maps, and the
second from equivariance of Cannon-Thurston maps (Theorem \ref{ctsurf}).

Hence, $\rho (h_j^{-1}\circ h_i)$ fixes  $\phi_\Gamma(x)=\phi_\Gamma ( y)$. As $x$ and $y$ are two different points on $\mathbb{S}^1$, we conclude that the element $\rho(h_j^{-1}\circ h_i)$ of $\Gamma$ must be an accidental parabolic, and $x$ and $y$ are fixed points of $h_j^{-1}\circ h_i$. Hence, $\ell$ is the axis of $h_j^{-1}\circ h_i$.

Since $\gamma$ was an arbitrary member of $\mathcal{L}^*$, the first part of the lemma follows. The second statement is now an immediate consequence of the fact that geometrically finite groups on the Bers boundary correspond precisely to laminations $\mathcal{L}^*$ given by the union of a finite collection of mutually disjoint simple, closed, non-peripheral geodesics on the surface (see the discussion in Subsection~\ref{geod_lami_subsec}).
\end{proof}

We now proceed to classify all laminations $\mathcal{L}$ (associated with groups on the Bers boundary of $G_d$) that are invariant under the Bowen-Series map $A_{G_d, \mathrm{BS}}$. The following subset of $G_d$ (see Subsection~\ref{b_s_punc_sphere_subsec} for the definition of $g_i$s) will play a special role in this description:
$$
\mathbf{S}_d:=\{g_2,\cdots,g_{d-1}\}\cup\{g_i^{-1}\circ g_j: i, j\in\{1,\cdots, d\},\ \vert i-j\vert>1\}.
$$

\begin{prop}\label{b_s_inv_simp_closed_geod_lem}
1) $\mathcal{L}$ is invariant under $A_{G_d, \mathrm{BS}}$ if and only if $\mathcal{L}^*$ is a union of mutually disjoint, simple, closed, non-peripheral geodesics on $\disk/G_d$ that are represented by group elements in the finite set $\mathbf{S}_d$.

2) Let $\Gamma\in\partial\mathcal{B}(G_d)$ be obtained by pinching a measured lamination $\mathcal{L}^*$ on the surface $\D/G_d$. Then, the Cannon-Thurston map of $\Gamma$ semi-conjugates $A_{G_d, \mathrm{BS}}$ to a self-map of $\Lambda(\Gamma)$ that is orbit equivalent to $\Gamma$ if and only if $\mathcal{L}^*$ is as in part (1) of the proposition. In particular, there are only finitely many quasiconformal conjugacy classes of groups $\Gamma\in\partial\mathcal{B}(G_d)$ for which the Cannon-Thurston map of $\Gamma$ semi-conjugates $A_{G_d, \mathrm{BS}}$ to a self-map of $\Lambda(\Gamma)$ that is orbit equivalent to $\Gamma$. Moreover, all such groups $\Gamma$ are geometrically finite. 
\end{prop}

\begin{proof}
1) Applying Lemma~\ref{inv_lami_pwfm} to the Bowen-Series map $A_{G_d, \mathrm{BS}}$, one sees that for any $A_{G_d, \mathrm{BS}}$-invariant lamination $\mathcal{L}$, the geodesic lamination $\mathcal{L}^*$ on the surface is a union of mutually disjoint, simple, closed, non-peripheral geodesics on $\disk/G_d$ that are represented by group elements in the finite set $\mathbf{S}_d$ (the other group elements give rise to peripheral/non-simple geodesics).

It remains to show that if $\mathcal{L}^*$ is the union of finitely many mutually disjoint, simple, closed, non-peripheral geodesics on $\disk/G_d$ represented by group elements in the finite set $\mathbf{S}_d$, then $\mathcal{L}$ is $A_{G_d, \mathrm{BS}}$-invariant. 

To this end, let us first assume that $\gamma\in\mathcal{L}^*$ is represented by $g_j$, for some $j\in\{2,\cdots, d-1\}$. Then, the lift of $\gamma$ to the fundamental domain of $G_d$ (which is also the fundamental domain of $A_{G_d, \mathrm{BS}}$) consists of a single arc that is contained in the axis $\ell$ of $g_j$. In particular, $\ell$ is a leaf of $\mathcal{L}$ that projects to $\gamma$ and connects  the two fixed points of $g_j$. Hence, $\ell$ has its endpoints at
$$
\lim_{n\to+\infty} g_j^{-n}(0)\in I_j\quad \textrm{and}\quad \lim_{n\to+\infty} g_j^n(0)\in I_{-j}.
$$
It follows that $A_{G_d, \mathrm{BS}}$ preserves the endpoints of the leaf $\ell$. It is also easy to see that the endpoints of any leaf of $\mathcal{L}$ of the form $g\cdot\ell$ (for $g\neq 1$) is contained in some $I_{\pm k}$. Hence, every iterate of $A_{G_d, \mathrm{BS}}$ carries the endpoints of such a leaf to the endpoints of some leaf $g'\cdot\ell$ (where $g'\in G_d$) of $\mathcal{L}$.

Now suppose that $\gamma\in\mathcal{L}^*$ is represented by $g_i^{-1}\circ g_j$, for some $i, j\in\{1,\cdots, d\}$ with $\vert i-j\vert>1$. Then, the lift of $\gamma$ to the fundamental domain of $G_d$ consists of two geodesic arcs; one of them is contained in the axis $\ell_1$ of $g_i^{-1}\circ g_j$, and the other is contained in the axis $\ell_2$ of $g_i\circ g_j^{-1}$. Thus, $\ell_1$ (respectively, $\ell_2$) is a leaf of $\mathcal{L}$ that projects to $\gamma$ and connects  the two fixed points of $g_i^{-1}\circ g_j$ (respectively, of $g_i\circ g_j^{-1}$). It follows that $\ell_1$ has its endpoints at
$$
\lim_{n\to+\infty} (g_i^{-1}\circ g_j)^n(0)\in I_i\quad \textrm{and}\quad \lim_{n\to+\infty} (g_j^{-1}\circ g_i)^n(0)\in I_j,
$$
while 
$\ell_2$ has its endpoints at
$$
\lim_{n\to+\infty} (g_i\circ g_j^{-1})^n(0)\in I_{-i}\quad \textrm{and}\quad \lim_{n\to+\infty} (g_j\circ g_i^{-1})^n(0)\in I_{-j}.
$$
It is now straightforward to check that $A_{G_d, \mathrm{BS}}$ maps the endpoints of $\ell_1$ to those of $\ell_2$, and vice versa.
Moreover, the mapping properties of the generators $g_1,\cdots, g_d$ imply that any leaf of $\mathcal{L}$ of the form $g\cdot\ell_i$ (for $g\in G_d$, $i\in\{1,2\}$) either coincides with $\ell_{i'}$ ($i'\in\{1,2\}$) or has both its endpoints in some $I_{\pm k}$. Hence, every iterate of $A_{G_d, \mathrm{BS}}$ carries the endpoints of such a leaf to the endpoints of some leaf $g'\cdot\ell_i$ (where $g'\in G_d$, and $i\in\{1,2\}$) of $\mathcal{L}$.

Since $\gamma\in\mathcal{L}^*$ was arbitrarily chosen, we now conclude that the $G_d$-invariant geodesic lamination $\mathcal{L}$ obtained by lifting $\mathcal{L}^*$ to the universal cover is $A_{G_d, \mathrm{BS}}$-invariant.

2) This follows from the first part and Lemma~\ref{pwfm_invariant_lami_lem}.
\end{proof}

\begin{defn}\label{b_s_boundary_group_def}
For $\Gamma\in\partial\mathcal{B}(G_d)$ with an $A_{G_d, \mathrm{BS}}$-invariant lamination $\mathcal{L}$, the continuous self-map $A_{\Gamma, \mathrm{BS}}$ of $\Lambda(\Gamma)$ provided by Proposition~\ref{b_s_inv_simp_closed_geod_lem} is a mateable map associated to $\Gamma$ that is compatible with  $A_{G_d, \mathrm{BS}}$. We call the map $A_{\Gamma, \mathrm{BS}}$ the \emph{Bowen-Series} map of $\Gamma$.
\end{defn}

We now turn our attention to higher Bowen-Series maps. In fact, giving a complete description of $A_{G_d, \mathrm{hBS}}$-invariant laminations $\mathcal{L}$ is a tedious and combinatorially involved task. To avoid this, we content ourselves with the following result which states that there is a non-empty, finite collection of (supports of) measured laminations $\mathcal{L}^*$ for which $\mathcal{L}$ is $A_{G_d, \mathrm{hBS}}$-invariant. In the following proposition, we will use the notation of Section~\ref{sec-fold}.

\begin{prop}\label{higher_b_s_inv_simp_closed_geod_lem}
1) The set of laminations $\mathcal{L}$ that are invariant under $A_{G_d, \mathrm{hBS}}$ is non-empty and finite. All such laminations correspond to geometrically finite groups on $\partial\mathcal{B}(G_d)$.

2) There are only finitely many quasiconformal conjugacy classes of groups $\Gamma\in\partial\mathcal{B}(G_d)$ for which the Cannon-Thurston map of $\Gamma$ semi-conjugates $A_{G_d, \mathrm{hBS}}$ to a self-map of $\Lambda(\Gamma)$ that is orbit equivalent to $\Gamma$.
\end{prop}
\begin{proof}
1) The finiteness assertion is a consequence of Lemma~\ref{inv_lami_pwfm}. 

Examples of $A_{G_d, \mathrm{hBS}}$-invariant laminations $\mathcal{L}^*$ are given by (simple, closed, non-peripheral) geodesics on $\disk/G_d$ represented by $g_2,\cdots, g_{d-1}$. To see this, assume without loss of generality that $\mathcal{L}^*$ consists of a single curve on $\disk/G_d$ represented by $g_j$, for some $j\in\{2,\cdots, d-1\}$. Then $\mathcal{L}$ (which is the $G_d$-lift of this curve to the universal cover) intersects the fundamental domain of $A_{G_d, \mathrm{hBS}}$ in $d$ geodesic arcs. Call the collection of these $d$ bi-infinite geodesics $\mathcal{L}'$. One of these bi-infinite geodesics has the fixed points of $g_j$ as its endpoints; we denote this geodesic by $\ell$. With this notation, we have that $\mathcal{L}'=\{g_1(\ell),\cdots, g_j(\ell)=\ell,\cdots, g_d(\ell)\}$. Note that each leaf of $\mathcal{L}\setminus\mathcal{L}'$ has its endpoints in a single
piece of $A_{G_d, \mathrm{hBS}}$. Hence, it suffices to argue that if $p, q$ are the endpoints of some leaf in $\mathcal{L}'$, then $A_{G_d, \mathrm{hBS}}(p), A_{G_d, \mathrm{hBS}}(q)$ are also the endpoints of some leaf in $\mathcal{L}'$. Clearly, this property is satisfied by the endpoints of $\ell$. Now pick $i\in\{1,\cdots, d\},\ i\neq j$. Then one endpoint of $g_i(\ell)$ lies in the sub-arc of $\bS^1$ where $A_{G_d, \mathrm{hBS}}\equiv g_i^{-1}$, and the other endpoint lies in the sub-arc where $A_{G_d, \mathrm{hBS}}\equiv g_j\circ g_i^{-1}$ (see Figure~\ref{cfmgen}). Since the endpoints of $\ell$ are fixed by $g_j$, it is now easy to see that $A_{G_d, \mathrm{hBS}}$ maps the endpoints of $g_i(\ell)$ to those of $\ell$. This completes the proof of the fact that $\mathcal{L}$ is $A_{G_d, \mathrm{hBS}}$-invariant.

2) Follows from the first part and Lemma~\ref{pwfm_invariant_lami_lem}.
\end{proof}

\begin{defn}\label{higher_b_s_boundary_group_def}
For $\Gamma\in\partial\mathcal{B}(G_d)$ with an $A_{G_d, \mathrm{hBS}}$-invariant lamination $\mathcal{L}$, the continuous self-map $A_{\Gamma, \mathrm{hBS}}$ of $\Lambda(\Gamma)$ provided by Proposition~\ref{higher_b_s_inv_simp_closed_geod_lem} is a mateable map associated to $\Gamma$ that is compatible with  $A_{G_d, \mathrm{hBS}}$. We call the map $A_{\Gamma, \mathrm{hBS}}$ the \emph{higher Bowen-Series} map of $\Gamma$.
\end{defn}

\subsection{Dynamics of Bowen-Series maps for Bers boundary groups}\label{boundary_group_b_s_subsec}

 Recall from Theorem \ref{ctsurf}) that for $\Gamma\in\partial\mathcal{B}(G_d)$, the limit set $\Lambda_\Gamma$ is obtained topologically as a quotient space of the circle $\bS^1$ by identifying ideal end-points of lifts of leaves of a non-trivial ending lamination (see Figure \ref{fig-bersbdy}). We recall the notation $\Omega(\Gamma), \Omega_\infty(\Gamma), K(\Gamma)$ from Section \ref{mateable_gen_subsec} to deal with Bers boundary groups.

\subsubsection{Explicit description of Bowen-Series maps for Bers boundary groups}
Let $\Gamma\in\partial\mathcal{B}(G_d)$ be a group that admits a Bowen-Series map $A_{\Gamma,\mathrm{BS}}$ (see Definition~\ref{b_s_boundary_group_def}). We will now see how $A_{\Gamma, \mathrm{BS}}$ can be explicitly written in terms of suitable elements of $\Gamma\leq\pslc$.

Following Subsection~\ref{b_s_punc_sphere_subsec}, we denote by $I_{\pm j}$ the arc of $\bS^1$ where $A_{G_d, \mathrm{BS}}\equiv g_j^{\pm 1}$, $j\in\{1,\cdots, d\}$. Since $\Gamma\in\partial\mathcal{B}(G_d)$, there is a (weakly type-preserving) group isomorphism $\rho: G_d\to\Gamma$ and a conformal isomorphism $\phi_\Gamma: \widehat{\C}\setminus\overline{\D}\to\Omega_\infty(\Gamma)$ such that $\phi_\Gamma$ conjugates the $g_j$-action on $\widehat{\C}\setminus\overline{\D}$ to the action of $\widehat{g}_j:=\rho(g_j)\in\pslc$ on $\Omega_\infty(\Gamma)$. In fact, the Cannon-Thurston map of $\Gamma$ is the continuous extension of $\phi_\Gamma$ to $\bS^1$. Abusing notation, we will denote the Cannon-Thurston map of $\Gamma$ by $\phi_\Gamma:\bS^1\to\Lambda(\Gamma)$. It now follows from the definition of $A_{\Gamma, \mathrm{BS}}$ (for instance, see the commutative diagram in Lemma~\ref{pwfm_invariant_lami_lem}) that 
$$
A_{\Gamma, \mathrm{BS}}\equiv \widehat{g}_j^{\pm 1},\ \mathrm{on}\ \phi_\Gamma(I_{\pm j}),\ j\in\{1,\cdots, d\}.
$$
We will see in the next subsection that using the M{\"o}bius maps $\widehat{g}_j$, the Bowen-Series map $A_{\Gamma, \mathrm{BS}}\vert_{\Lambda(\Gamma)}$ can be extended to a subset of $K(\Gamma)$ as a continuous piecewise complex-analytic map.

\subsubsection{Canonical extension of Bowen-Series maps for Bers boundary groups}\label{bs_extension_bdry_subsec}
We will now define a canonical extension of $A_{\Gamma, \mathrm{BS}}$. To this end, choose a lamination $\mathcal{L}$ such that $\mathcal{L}^*$ is a union of mutually disjoint, simple, closed, non-peripheral geodesics on $\disk/G_d$ that are represented by group elements in the finite set $\mathbf{S}_d$ (recall from Proposition~\ref{b_s_inv_simp_closed_geod_lem} that these are precisely the $A_{G_d, \mathrm{BS}}$-invariant laminations). The $G_d$-lift of each leaf of $\mathcal{L}^*$ corresponding to a group element in $\{g_2,\cdots,g_{d-1}\}$ intersects the fundamental domain $R_{G_d}$ of $A_{G_d, \mathrm{BS}}$ (whose closure is a closed fundamental domain of $G_d$) in exactly one geodesic arc. This geodesic is symmetric with respect to the real line. We call such an arc a \emph{1-arc}. On the other hand, the $G_d$-lift of each leaf of $\mathcal{L}^*$ corresponding to a group element in $\{g_i^{-1}\circ g_j: i, j\in\{1,\cdots, d\},\ \vert i-j\vert>1\}$ intersects the fundamental domain $R_{G_d}$ in exactly two geodesic arcs, that are mapped to each other by the complex conjugation map. We call such a pair of arcs a \emph{2-arc}. Thus, every $A_{G_d, \mathrm{BS}}$-invariant lamination $\mathcal{L}$ intersects $R_{G_d}$ in a disjoint collection of $1$-arcs and $2$-arcs. We denote the union of such arcs by $\mathcal{A}_\mathcal{L}$.

It follows from the above discussion that for an $A_{G_d, \mathrm{BS}}$-invariant lamination $\mathcal{L}$, the associated collection of $1$ and $2$-arcs $\mathcal{A}_\mathcal{L}$ cuts $R_{G_d}$ into finitely many components. We call such a component \emph{equatorial} if it intersects the real line (i.e., the set $\mathrm{Im}(z)=0$ in $\D$). Otherwise, it is called a \emph{polar} component. (See Figure~\ref{induced_bs_fig}, where a certain $G_6$-invariant lamination gives rise to two equatorial and two polar components.) 

If a geometrically finite group $\Gamma$ is obtained by pinching $\mathcal{L}^*$, then $R_{G_d}$ naturally determines a \emph{pinched fundamental domain} for the $\Gamma$-action on $\Omega(\Gamma)\setminus \Omega_\infty(\Gamma)$. The closure of this pinched fundamental domain can be topologically realized by pinching the polygon $\overline{R_{G_d}}$ along $\mathcal{A}_\mathcal{L}$. In fact, the components of $R_{G_d}\setminus\mathcal{A}_\mathcal{L}$ bijectively correspond to the components of the pinched fundamental domain. We call a component of the pinched fundamental domain \emph{equatorial} (respectively, \emph{polar}) if it corresponds to an equatorial (respectively, polar) component of $R_{G_d}\setminus \mathcal{A}_\mathcal{L}$ (see Figure~\ref{induced_bs_fig}). We denote the interior of this fundamental domain (for the $\Gamma$-action on $\Omega(\Gamma)\setminus \Omega_\infty(\Gamma)$) by $R_\Gamma$.

As $A_\Gamma$ acts by M{\"o}bius maps (in $\pslc$), it admits a natural extension to $K(\Gamma)\setminus R_\Gamma$ as follows. Fix $j\in\{\ 1,\cdots, d\}$. There is a unique arc $c_j$ (respectively, $c_{-j}$) in $\partial R_\Gamma$ connecting the endpoints of $\phi_\Gamma(I_ j)$ (respectively, the endpoints of $\phi_\Gamma(I_{- j})$) such that $c_j$ (respectively, $c_{-j}$) does not contain any non-accidental parabolic fixed point of $\Gamma$ other than its endpoints (i.e., the arcs $c_j, c_{-j}$ are allowed to pass through accidental parabolics, but not through non-accidental parabolics). On the (closed) subset of $K(\Gamma)$ bounded by $c_j$ and $\phi_\Gamma(I_ j)$ (respectively, bounded by $c_{-j}$ and $\phi_\Gamma(I_{-j})$), we extend $A_{\Gamma, \mathrm{BS}}$ as $\widehat{g}_j$ (respectively, $\widehat{g}_j^{-1}$).

We call this extended map the \emph{canonical extension} of $A_{\Gamma, \mathrm{BS}}$, and denote it by $\widehat{A}_{\Gamma, \mathrm{BS}}: K(\Gamma)\setminus R_\Gamma\to K(\Gamma)$.

\subsubsection{Dynamics of $\widehat{A}_{\Gamma, \mathrm{BS}}$}\label{sec-dynaAhbs}

We call a component $U$ of $\Omega(\Gamma)$ a \emph{principal component} if $U$ contains a component of $R_\Gamma$. Further, we call such a component $U$ an \emph{equatorial/polar principal component} if the component of $R_\Gamma$ contained in $U$ is equatorial/polar.

Note that since the restriction of $\widehat{A}_{G_d, \mathrm{BS}}$ on $\partial R_{G_d}$ is a homeomorphism of order two, the same is true for the restriction of $\widehat{A}_{\Gamma, \mathrm{BS}}$ on $\partial R_\Gamma$. More precisely, $\widehat{A}_{\Gamma, \mathrm{BS}}$ maps $c_{\pm j}$ onto $c_{\mp j}$, for $j\in\{1,\cdots, d\}$. The next lemma follows from this observation and the definitions of equatorial/polar principal components.

\begin{lemma}\label{period_one_two_lem}
\noindent
\begin{enumerate}\upshape
\item Each equatorial principal component of $\Omega(\Gamma)$ is preserved under $\widehat{A}_{\Gamma, \mathrm{BS}}$, while each polar principal component forms a $2$-cycle under $\widehat{A}_{\Gamma, \mathrm{BS}}$.

\item Each non-principal component of $\Omega(\Gamma)\setminus\Omega_\infty(\Gamma)$ is eventually mapped by $\widehat{A}_{\Gamma, \mathrm{BS}}$ to a principal component.

\end{enumerate}
\end{lemma}

\begin{figure}[h!]
\begin{tikzpicture}
\node[anchor=south west,inner sep=0] at (0,0) {\includegraphics[width=0.42\textwidth]{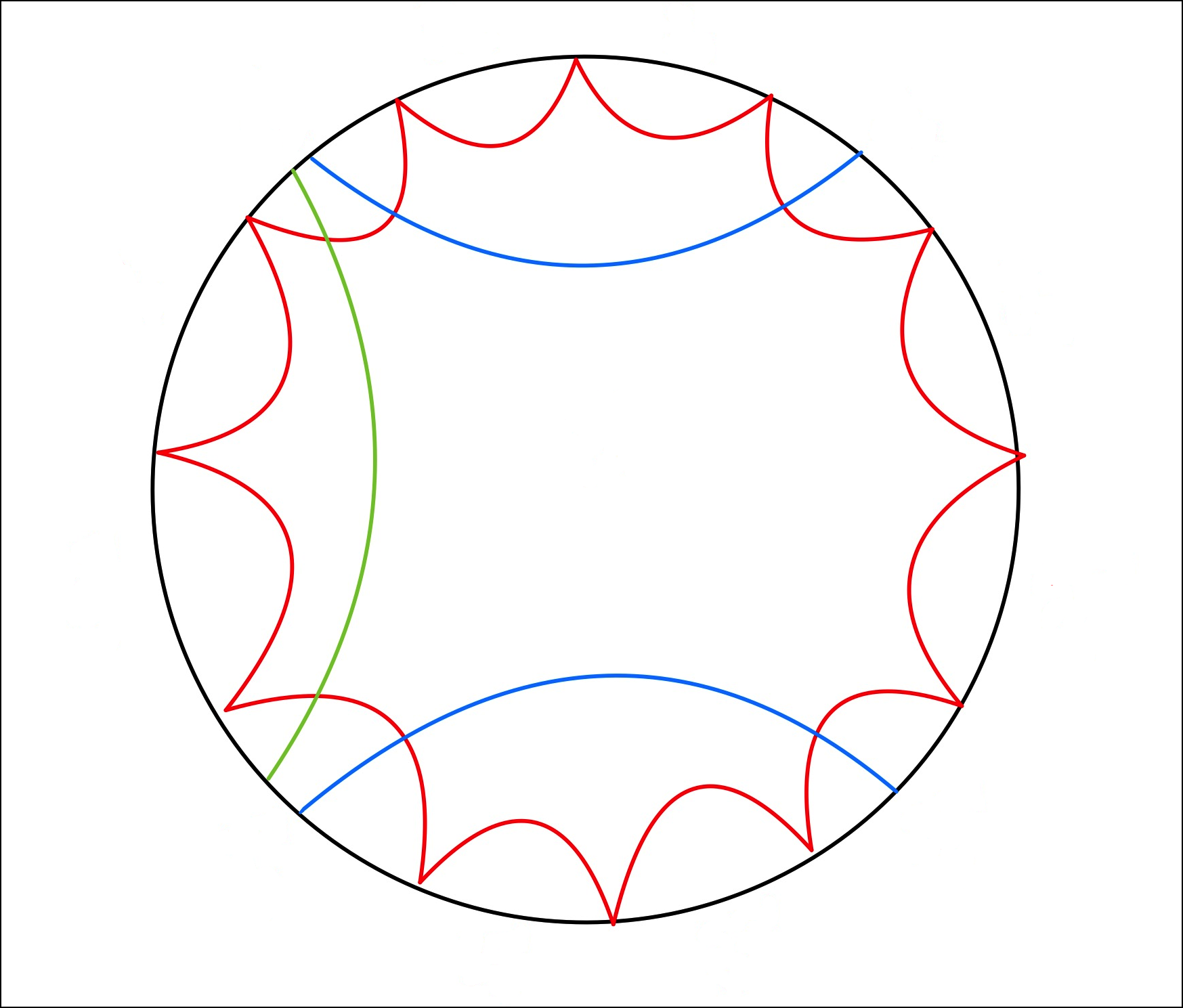}};
\node[anchor=south west,inner sep=0] at (5.4,0) {\includegraphics[width=0.54\textwidth]{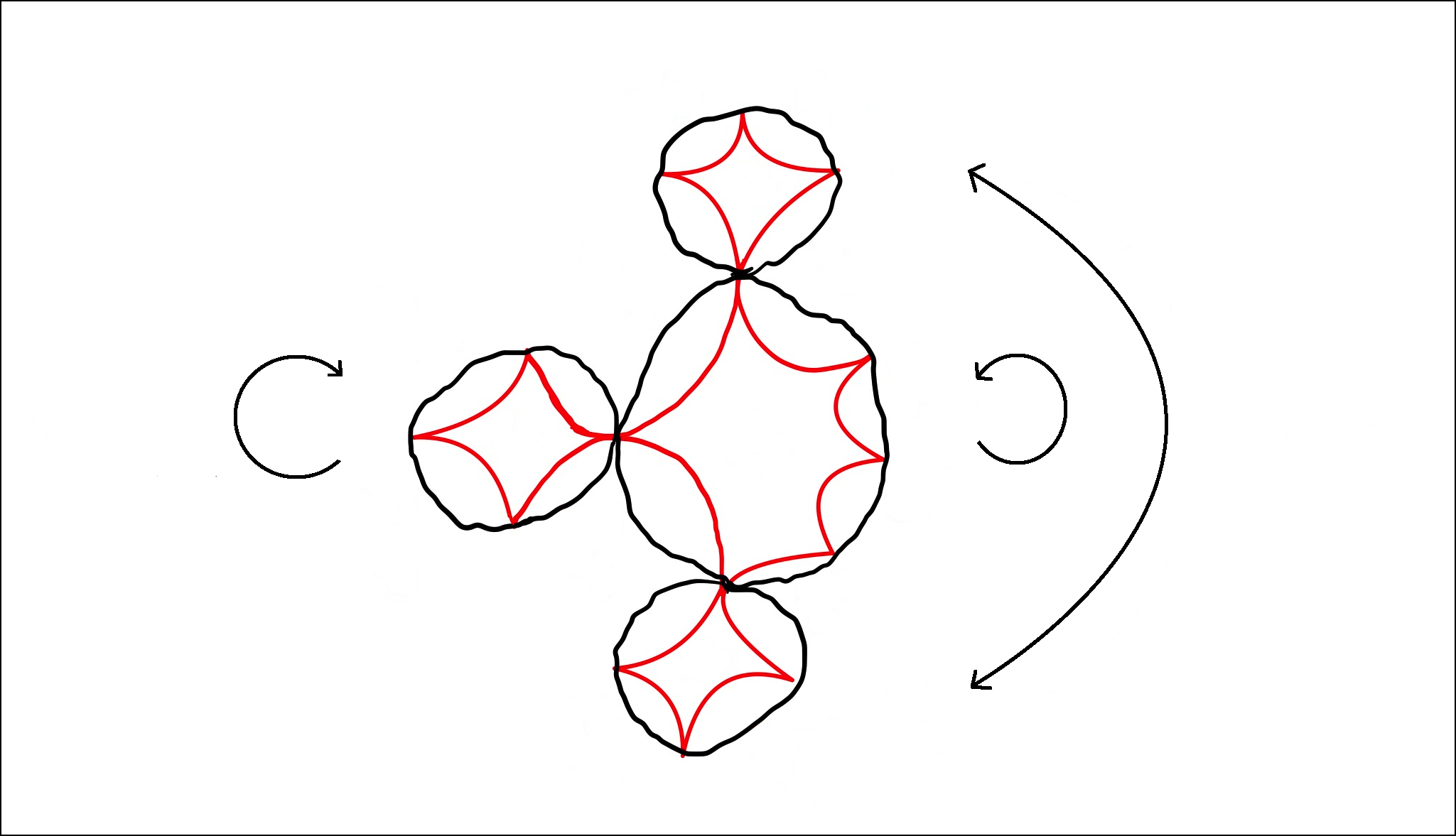}};
\node at (0.4,3.1) {$g_6$};
\node at (0.36,2) {$g_6^{-1}$};
\node at (1.2,4.1) {$g_5$};
\node at (1,0.6) {$g_5^{-1}$};
\node at (4.8,3.1) {$g_1$};
\node at (5,1.8) {$g_1^{-1}$};
\node at (4.1,4.1) {$g_2$};
\node at (4.4,0.8) {$g_2^{-1}$};
\node at (3.2,4.4) {$g_3$};
\node at (3.28,0.32) {$g_3^{-1}$};
\node at (2.2,4.4) {$g_4$};
\node at (2,0.22) {$g_4^{-1}$};
\node at (2.7,2.4) {\begin{Large}$R_{G_d}$\end{Large}};
\node at (9.7,3.16) {\begin{small}$U^+$\end{small}};
\node at (9.66,0.7) {\begin{small}$U^-$\end{small}};
\node at (9,1.8) {$R_\Gamma$};
\node at (7.9,1.84) {\begin{tiny}$R_\Gamma$\end{tiny}};
\node at (8.92,3.06) {\begin{tiny}$R_\Gamma$\end{tiny}};
\node at (8.75,0.84) {\begin{tiny}$R_\Gamma$\end{tiny}};
\node at (9.75,2.1) {\begin{tiny}$\widehat{g}_1$\end{tiny}};
\node at (9.8,1.6) {\begin{tiny}$\widehat{g}_1^{-1}$\end{tiny}};
\node at (9.36,2.7) {\begin{tiny}$\widehat{g}_2$\end{tiny}};
\node at (9.4,1.1) {\begin{tiny}$\widehat{g}_2^{-1}$\end{tiny}};
\node at (9.2,3.6) {\begin{tiny}$\widehat{g}_3$\end{tiny}};
\node at (9.2,0.4) {\begin{tiny}$\widehat{g}_3^{-1}$\end{tiny}};
\node at (8.6,3.6) {\begin{tiny}$\widehat{g}_4$\end{tiny}};
\node at (8.2,0.4) {\begin{tiny}$\widehat{g}_4^{-1}$\end{tiny}};
\node at (8.4,2.7) {\begin{tiny}$\widehat{g}_5$\end{tiny}};
\node at (8.3,1.2) {\begin{tiny}$\widehat{g}_5^{-1}$\end{tiny}};
\node at (7.4,2.25) {\begin{tiny}$\widehat{g}_6$\end{tiny}};
\node at (7.24,1.4) {\begin{tiny}$\widehat{g}_6^{-1}$\end{tiny}};
\node at (6.4,2.5) {\begin{small}$\widehat{A}_{\Gamma, \mathrm{BS}}$\end{small}};
\node at (10.2,2.5) {\begin{small}$\widehat{A}_{\Gamma, \mathrm{BS}}$\end{small}};
\node at (11.4,2) {\begin{small}$\widehat{A}_{\Gamma, \mathrm{BS}}$\end{small}};
\end{tikzpicture}
\caption{Left: $R_{G_6}$ is (the interior of) a fundamental domain of $G_6$.
Right: A cartoon of the limit set of the group $\Gamma$, which is obtained by pinching $\mathcal{L}^*$.}
\label{induced_bs_fig}
\end{figure}
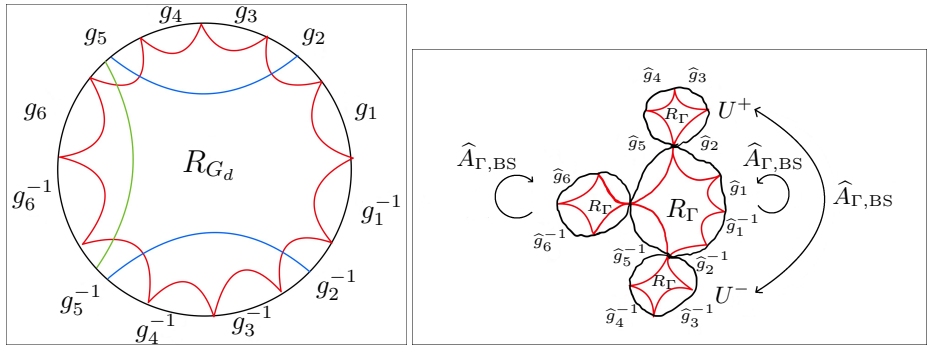

In the left picture of figure \ref{induced_bs_fig}, the blue and green geodesics comprise $\mathcal{A}_\mathcal{L}$, where $\mathcal{L}^*$ consists of two simple, closed curves on $\D/G_6$ corresponding to the elements $g_5, g_5^{-1}\circ g_2\in G_6$. Thus, there are two equatorial and two polar components of $R_{G_6}\setminus\mathcal{A}_\mathcal{L}$. The M{\"o}bius maps defining $\widehat{A}_{G_d, \mathrm{BS}}$ are marked. In  the  picture on the right, the $\Gamma$-action on $\Omega(\Gamma)\setminus \Omega_\infty(\Gamma)$ admits a pinched fundamental domain whose interior is $R_\Gamma$ and whose closure is a finite tree of polygons such that two adjacent polygons share a vertex at an accidental parabolic fixed point of $\Gamma$. $\Omega(\Gamma)$ has four principal components; two of them (the equatorial ones) are preserved by $\widehat{A}_{\Gamma, \mathrm{BS}}$, while the other two (the polar ones) are exchanged. The M{\"o}bius maps defining $\widehat{A}_{\Gamma, \mathrm{BS}}$ are also marked.

The following theorem provides us with conformal models for the first return maps of $\widehat{A}_{\Gamma, \mathrm{BS}}$ on the principal components of $\Omega(\Gamma)$. The explicit description of first return maps of Bowen-Series maps will play a key role in the David surgery step in the proof of existence of conformal matings between $\widehat{A}_{\Gamma, \mathrm{BS}}$ and polynomials in principal hyperbolic components (see Subsection~\ref{mating_boundary_groups_subsec}). It is worth noting that the conformal models of first return maps of Bowen-Series maps naturally involve higher Bowen-Series maps.

\begin{theorem}\label{first_return_conf_model_thm}
Let $U$ be a principal component of $\Omega(\Gamma)$. 
\begin{enumerate}\upshape
\item If $U$ is equatorial, then $\widehat{A}_{\Gamma, \mathrm{BS}}: \overline{U}\setminus R_\Gamma\to\overline{U}$ is conformally conjugate to (the canonical extension of) the Bowen-Series map of a punctured sphere Fuchsian group. 

\item If $U$ is polar, then the dynamical system 
$$
\overline{U}\ \xleftrightarrows[\quad \widehat{A}_{\Gamma, \mathrm{BS}}\quad ]{\quad \widehat{A}_{\Gamma, \mathrm{BS}}\quad }\ 
\widehat{A}_{\Gamma, \mathrm{BS}}(\overline{U})
$$
is conformally conjugate to a fiberwise dynamical system 
$$
\mathbf{A}:\left(\overline{\D}\setminus R_+\right)\times\{+\}\bigsqcup\left(\overline{\D}\setminus R_-\right)\times\{-\}\to\overline{\D}\times\{+,-\},\ (z,\pm)\mapsto (\widehat{A}_\pm(z), \mp),
$$
where $A_\pm$ are \pwfm maps satisfying the conditions of Corollary~\ref{hbs_first_return_cor}. Consequently,
the first return map of $\widehat{A}_{\Gamma, \mathrm{BS}}$ on $\overline{U}$ is conformally conjugate to (the canonical extension of) the higher Bowen-Series map of a punctured sphere Fuchsian group.

\end{enumerate}
\end{theorem}
\begin{proof}
1) Suppose that $U$ is an equatorial principal component of $\Omega(\Gamma)$. Let $W$ be the component of $R_\Gamma$ contained in $U$. Then, $\partial W$ is formed by parts of $c_{\pm i_1},\cdots, c_{\pm i_r}$, for some $i_1,\cdots, i_r\in\{1,\cdots, d\}$ (compare Figure~\ref{induced_bs_fig}). Further, setting $c_{\pm i_s}^U:= c_{\pm i_s}\cap U$, we see that $\widehat{A}_{\Gamma, \mathrm{BS}}$ acts on $c_{\pm i_s}^U$ as $\widehat{g}_{i_s}^{\pm 1}$, and $\widehat{g}_{i_s}$ carries $c_{i_s}^U$ onto $c_{-i_s}^U$.

Let $\psi_{U}:\overline{\D}\to\overline{U}$ be (the homeomorphic extension of) a Riemann uniformization. It now readily follows that 
$$
\psi_{U}^{-1}\circ\widehat{A}_{\Gamma, \mathrm{BS}}\vert_{\overline{U}}\circ\psi_{U}: \overline{\D}\setminus \psi_{U}^{-1}(\Int{W})\to\overline{\D}
$$
is the canonical extension of a \pwfm map that carries the boundary of its fundamental domain $\psi_{U}^{-1}(\Int{W})$ onto itself. By Proposition~\ref{bs_char_1}, $\psi_{U}^{-1}\circ\widehat{A}_{\Gamma, \mathrm{BS}}\vert_{\overline{U}}\circ\psi_{U}$ is the canonical extension of the Bowen-Series map of a punctured sphere Fuchsian group.

2) We now assume that $U=U^+$ is a polar principal component of $\Omega(\Gamma)$. By Lemma~\ref{period_one_two_lem}, $U^+$ forms a $2$-cycle under $\widehat{A}_{\Gamma, \mathrm{BS}}$. Let $U^-$ be the image polar principal component; i.e., $\partial U^-= \widehat{A}_{\Gamma, \mathrm{BS}}(\partial U^+)$. Denote the component of $R_\Gamma$ contained in $U^+$ (respectively, $U^-$) by $D^+$ (respectively, $D^-$). Without loss of generality, we can assume that $\partial D^+$ is formed by parts of $c_{i_1},\cdots, c_{i_r}$, for some $i_1,\cdots, i_r\in\{1,\cdots, d\}$. Then, $\partial D^-$ is formed by parts of $c_{-i_1},\cdots, c_{-i_r}$ (compare Figure~\ref{induced_bs_fig}). As in the previous case, we set $c_{i_s}^{U^+}:= c_{i_s}\cap U^+$, and $c_{-i_s}^{U^-}:= c_{-i_s}\cap U^-$, for $s\in\{1,\cdots, r\}$. Note that $\widehat{A}_{\Gamma, \mathrm{BS}}$ acts on $c_{i_s}^{U^+}$ as $\widehat{g}_{i_s}$, and on $c_{-i_s}^{U^-}$ as $\widehat{g}_{i_s}^{-1}$. Moreover, $\widehat{g}_{i_s}\left(c_{i_s}^{U^+}\right)=c_{- i_s}^{U^-}$.

Let $\psi_{\pm}:\overline{\D}\to\overline{U^\pm}$ be (the homeomorphic extensions of) Riemann uniformizations. Then, for each $s\in\{1,\cdots, r\}$, $\left(\psi_{\mp}\right)^{-1}\circ\widehat{g}_{i_s}^{\pm 1}\circ\psi_{\pm}$ is a conformal automorphism of $\D$. Hence, $\left(\psi_{-}\right)^{-1}\circ\widehat{A}_{\Gamma, \mathrm{BS}}\vert_{\overline{U^+}}\circ\psi_{+}$ and $\left(\psi_{+}\right)^{-1}\circ\widehat{A}_{\Gamma, \mathrm{BS}}\vert_{\overline{U^-}}\circ\psi_{-}$ are the canonical extensions of two \pwfm maps, which we denote by $A_+$ and $A_-$ (respectively). By construction, 
$$
\overline{U}\ \xleftrightarrows[\quad \widehat{A}_{\Gamma, \mathrm{BS}}\quad ]{\quad \widehat{A}_{\Gamma, \mathrm{BS}}\quad }\ \widehat{A}_{\Gamma, \mathrm{BS}}(\overline{U})
$$
is conjugate via the pair of conformal maps $\psi_\pm$ to the fiberwise dynamical system $\mathbf{A}$ defined by $A_+$ and $A_-$. The description of the action of $\widehat{A}_{\Gamma, \mathrm{BS}}$ on $\partial D^\pm$ (given in the previous paragraph) now readily implies that the \pwfm maps $A_+, A_-$ satisfy the conditions of Corollary~\ref{hbs_first_return_cor}. Hence, the first return map $\widehat{A}_{\Gamma, \mathrm{BS}}^{\circ 2}$ on $\overline{U^+}$ is conformally conjugate (via $\psi_+$) to the canonical extension of the higher Bowen-Series map of a punctured sphere Fuchsian group.
\end{proof}

\subsection{From geodesic laminations to polynomial laminations} For a monic, centered, complex polynomial $P$ of degree $n$ with connected filled Julia set $\mathcal{K}(P)$, there exists a conformal isomorphism 
$\phi_P:\widehat{\C}\setminus\overline{\disk}\to\mathcal{B}_\infty(P):=\widehat{\C}\setminus\mathcal{K}(P)$ 
that conjugates $z^n$ to $P$ \cite[Theorem~9.5]{milnor-book}. The map $\phi_P$ is unique up to precomposition with multiplication by an $(n-1)$-th root of unity. We normalize $\phi_P$ so that $\phi_P(z)/z\to 1$ as $z\to\infty$, and call $\phi_P$ the \emph{B{\"o}ttcher coordinate} for $P$. If $\mathcal{J}(P)=\partial\mathcal{K}(P)$ is locally connected, then the conformal map $\phi_P$ extends continuously to a semi-conjugacy $\phi_P$ between $z^n\vert_{\mathbb{S}^1}$ and $P\vert_{\mathcal{J}(P)}$, and the fibers of $\phi_P$ define a $z^n$-invariant closed equivalence relation on $\mathbb{S}^1$. This equivalence relation, which is called the \emph{lamination} of $P$ (denoted by $\lambda(P)$), can be used to topologically model the dynamics of $P$ on its Julia set $\mathcal{J}(P)$ (see \cite{kiwi1} for a comprehensive account of polynomial laminations).

\begin{defn}\label{poly_lami_def}
 (1) Let $\lambda$ be an equivalence relation on $\mathbb{S}^1\cong \reals/\integers$ satisfying the following conditions, where the map $m_n:\reals/\integers\to\reals/\integers$ is given by $\theta\mapsto n\theta$.\\
\noindent {$a)$} $\lambda$ is closed in $\reals/\integers\times\reals/\integers$.\\
\noindent {$b)$} Each equivalence class $X$ of $\lambda$ is a finite subset of $\reals/\integers$.\\
\noindent {$c)$} If $X$ is a $\lambda$-equivalence class, then $m_n(X)$ is also a $\lambda$-equivalence class.\\
\noindent {$d)$} If $X$ is a $\lambda$-equivalence class, then $X\mapsto m_n(X)$ is cyclic order preserving.\\
\noindent {$e)$} $\lambda$-equivalence classes are pairwise \emph{unlinked}; i.e., if $X$ and $Y$ are two distinct equivalence classes of $\lambda$, then there exist disjoint intervals $I_X, I_Y\subset\reals/\integers$ such that $X\subset I_X$ and $Y\subset I_Y$.\\
Then, $\lambda$ is called a \emph{$\reals$eal lamination}.

\smallskip

\noindent (2) A $\reals$eal lamination $\lambda$ is said to have \emph{no rotation curves} if for every periodic simple curve $\gamma\subset\faktor{(\reals/\integers)}{\lambda}$ (periodic under the self-map of $\faktor{(\reals/\integers)}{\lambda}$ induced by $m_n$), the corresponding return map is not a homeomorphism.

\smallskip

\noindent (3) An equivalence class $X$ of a $\reals$eal lamination $\lambda$ is called a \emph{Julia critical element} if the degree of the map $m_n:X\to m_n(X)$ is greater than one. Finally, a $\reals$eal lamination $\lambda$ is called \emph{postcritically finite} if every Julia critical element of $\lambda$ is contained in $\ratls/\integers$. 
\end{defn}

Using this terminology, we can now state a realization theorem that will play a crucial role in relating group-invariant geodesic laminations to polynomial laminations (see \cite{kiwi1} for a more general statement).

\begin{theorem}\cite[Theorem~9.6]{LMMN}\label{poly_laminations_realized_thm}
For a postcritically finite $\reals$eal lamination $\lambda$ (with no rotation curves), there exists a monic, centered, postcritically finite polynomial $P$ with $\lambda(P)=\lambda$. In particular, $\faktor{(\reals/\integers)}{\lambda}\cong\mathcal{J}(P)$, and $\phi_P$ descends to a topological conjugacy between 
$m_n:\faktor{(\reals/\integers)}{\lambda}\to\faktor{(\reals/\integers)}{\lambda}$ and $P: \mathcal{J}(P)\to\mathcal{J}(P).$
\end{theorem}

The following result gives a positive answer to Problem~\ref{problem_1} for a special class of Bers boundary groups.

\begin{theorem}\label{julia_limit_dyn_equiv_thm}
Let $\Gamma\in\partial\mathcal{B}(G_d)$ be a group obtained by pinching a measured lamination $\mathcal{L}^*$ (on $\D/G_d$) such that $\mathcal{L}$ is $A_{G_d, \mathrm{BS}}$-invariant (respectively, $A_{G_d, \mathrm{hBS}}$-invariant). Then, there exists a (monic, centered) postcritically finite complex polynomial $P_\Gamma$ of degree $2d-1$ (respectively, $d^2$) such that 
$$
A_{\Gamma, \mathrm{BS}}:\Lambda(\Gamma)\to\Lambda(\Gamma)\ \qquad \textrm{and} \qquad P_\Gamma:\mathcal{J}(P_\Gamma)\to\mathcal{J}(P_\Gamma)
$$
$$
\left(\textrm{respectively}, \qquad A_{\Gamma, \mathrm{hBS}}:\Lambda(\Gamma)\to\Lambda(\Gamma)\ \qquad \textrm{and} \qquad P_\Gamma:\mathcal{J}(P_\Gamma)\to\mathcal{J}(P_\Gamma)\ \right)
$$
are topologically conjugate. 
\end{theorem}
\begin{proof}
We write a proof in the case of Bowen-Series maps. Exactly the same proof can be carried through in the higher Bowen-Series case.

We set $n=2d-1$.

Recall that the equivalence relation on $\mathbb{S}^1$ generated by the endpoints of $\mathcal{L}$ is denoted by $\sim$. Using the topological conjugacy between $A_{G_d, \mathrm{BS}}$ and $z^n$, we can push forward the closed equivalence relation $\sim$ to a closed equivalence relation $\lambda$ on $\reals/\integers$.
As each equivalence class of $\sim$ has cardinality at most two by Lemma \ref{inv_lami_pwfm}, the same is true for $\lambda$. 
Since $A_{G_d, \mathrm{BS}}$ maps each equivalence class of $\sim$ bijectively onto its image class, the same holds for the $\lambda$-equivalence classes under $m_n$.
The fact that every class $X'$ of $\sim$ has at most two members implies that $X'\mapsto A_{G_d, \mathrm{BS}}(X')$ is trivially cyclic order preserving. Therefore, $X\mapsto m_n(X)$ is cyclic order preserving for every $\lambda$-equivalence class $X$. 
Since no two leaves of $\mathcal{L}$ can cross, $\sim$-equivalence classes are pairwise unlinked. So the same is true for $\lambda$.  
That $A_{G_d, \mathrm{BS}}$ carries each equivalence class of $\sim$ bijectively onto its image class translates to the fact that $\lambda$ is postcritically finite. By Lemma~\ref{period_one_two_lem} and Theorem~\ref{first_return_conf_model_thm}, $\lambda$ has no rotation curves. 

Thanks to the above properties, we can invoke Theorem~\ref{poly_laminations_realized_thm}, which provides us with a degree $n$ polynomial $P_\Gamma$ such that 
$$
m_n:\faktor{(\reals/\integers)}{\lambda}\to\faktor{(\reals/\integers)}{\lambda}\qquad \textrm{and}\qquad P_\Gamma: \mathcal{J}(P_\Gamma)\to\mathcal{J}(P_\Gamma)
$$
are topologically conjugate.

On the other hand, the Cannon-Thurston map of $\Gamma$ induces a topological conjugacy between 
$$A_{G_d, \mathrm{BS}}:\faktor{\mathbb{S}^1}{\sim}\to\faktor{\mathbb{S}^1}{\sim}\qquad \textrm{and}\qquad  A_{\Gamma, \mathrm{BS}}: \Lambda(\Gamma)\to\Lambda(\Gamma).
$$

Finally, since the topological conjugacy between $A_{G_d, \mathrm{BS}}\vert_{\mathbb{S}^1}$ and $m_n\vert_{\reals/\integers}$ pushes forward the relation $\sim$ to $\lambda$, it descends to a topological conjugacy between
$$A_{G_d, \mathrm{BS}}:\faktor{\mathbb{S}^1}{\sim}\to\faktor{\mathbb{S}^1}{\sim} \qquad \textrm{and} \qquad m_n:\faktor{(\reals/\integers)}{\lambda}\to\faktor{(\reals/\integers)}{\lambda}.
$$  
Composing the above topological conjugacies, we get a homeomorphism $\pmb{\Phi}:\mathcal{J}(P_\Gamma)\to\Lambda(\Gamma)$ that conjugates $P_\Gamma$ to $A_{\Gamma, \mathrm{BS}}$. Thus, $P_\Gamma$ is our desired polynomial.
\end{proof}

\begin{rmk}\label{antiholo_dictionary_rmk}
In the anti-holomorphic world, analogues of Theorem~\ref{julia_limit_dyn_equiv_thm} were proved in \cite{LLM20} (also compare \cite[Theorem~B]{LMM2}), where dynamically natural homeomorphisms between limit sets of kissing reflection groups and Julia sets of critically fixed anti-rational maps were constructed.
\end{rmk}

We will conclude this subsection with an analysis of some dynamical properties of the polynomial $P_\Gamma$ obtained in Theorem~\ref{julia_limit_dyn_equiv_thm} (in the Bowen-Series case). These properties will play an important role in the next subsection.

Note that the topological conjugacy $\pmb{\Phi}$ between $P_\Gamma\vert_{\mathcal{J}(P_\Gamma)}$ and $A_{\Gamma, \mathrm{BS}}\vert_{\Lambda(\Gamma)}$ maps the boundaries of the bounded Fatou components of $P_\Gamma$ to the boundaries of the components of $\Omega(\Gamma)\setminus\Omega_\infty(\Gamma)$. In light of Lemma~\ref{period_one_two_lem}, the boundaries of the periodic (respectively, strictly pre-periodic) bounded Fatou components of $P_\Gamma$ are sent (by $\pmb{\Phi}$) to the boundaries of the principal (respectively, non-principal) components of $\Omega(\Gamma)$ (see Subsection~\ref{sec-dynaAhbs} for definitions). 
 
 We will call a periodic bounded Fatou component of $P_\Gamma$ \emph{equatorial} (respectively, \emph{polar}) if its boundary is mapped by $\pmb{\Phi}$ to the boundary of an equatorial (respectively, polar) principal component of $\Omega(\Gamma)$. Clearly, each equatorial (respectively, polar) Fatou component of $P_\Gamma$ is invariant under $P_\Gamma$ (respectively, forms a $2$-cycle of Fatou components).

\begin{prop}\label{dyn_p_gamma_prop}
Suppose that $\Gamma\in\partial\mathcal{B}(G_d)$ admits a Bowen-Series map, and $P_\Gamma$ is the postcritically finite polynomial obtained in Theorem~\ref{julia_limit_dyn_equiv_thm} such that $P_\Gamma\vert_{\mathcal{J}(P_\Gamma)}$ and $A_{\Gamma, \mathrm{BS}}\vert_{\Lambda(\Gamma)}$ are topologically conjugate. Then the following hold.
\begin{enumerate}
\item Each finite critical point of $P_\Gamma$ lies in a periodic bounded Fatou component.

\item If $\mathcal{U}$ is an equatorial Fatou component of $P_\Gamma$, then $P_\Gamma\vert_{\overline{\mathcal{U}}}$ is conformally conjugate to $z^k\vert_{\overline{\D}}$, for some $k\geq 2$.

\item If $\mathcal{U}$ is a polar Fatou component of $P_\Gamma$, then 
$$
\overline{\mathcal{U}}\ \xleftrightarrows[\quad P_\Gamma\quad ]{\quad P_\Gamma \quad}\  P_\Gamma(\overline{\mathcal{U}})
$$
is conformally conjugate to the fiberwise dynamical system
$$
\mathbf{B}: \overline{\D}\times\{+,-\}\to\overline{\D}\times\{+,-\},\ (z,\pm)\mapsto (z^k, \mp),
$$
for some $k\geq 2$
\end{enumerate}
\end{prop}
\begin{proof}
1) The fact that $m_{2d-1}$ maps each equivalence class of $\lambda$ bijectively onto its image class translates to the fact that $P_\Gamma$ has no critical point on its Julia set. Also recall that $A_{\Gamma, \mathrm{BS}}$ acts as a single group element on the boundary of each non-principal component of $\Omega(\Gamma)$. Thus, the boundary of each non-principal component of $\Omega(\Gamma)$ is mapped homeomorphically onto the boundary of a principal component by some positive iterate of $A_{\Gamma, \mathrm{BS}}$. It follows that the boundary of each strictly pre-periodic bounded Fatou component of $P_\Gamma$ is mapped homeomorphically onto the boundary of a periodic Fatou component by some iterate of $P_\Gamma$. This shows that no strictly pre-periodic bounded Fatou component of $P_\Gamma$ contains a critical point. Hence, every finite critical point of the postcritically finite polynomial $P_\Gamma$ lies in a periodic bounded Fatou component. 

2) Let $\mathcal{U}$ be an equatorial Fatou component of $P_\Gamma$. Then $\pmb{\Phi}(\partial\mathcal{U})$ is the boundary of an equatorial principal component of $\Omega(\Gamma)$, and hence $A_{\Gamma, \mathrm{BS}}:\pmb{\Phi}(\partial\mathcal{U})\to\pmb{\Phi}(\partial\mathcal{U})$ is a covering of some degree $k\geq 2$ (by part (1) of Theorem~\ref{first_return_conf_model_thm}). Therefore, $P_\Gamma:\partial\mathcal{U}\to\partial\mathcal{U}$ is a covering map of degree $k\geq 2$, so $P_\Gamma:\mathcal{U}\to\mathcal{U}$ is a proper branched covering of degree $k$. Choose (the homeomorphic extension of) a Riemann uniformization $\phi_{\mathcal{U}}:\overline{\D}\to\overline{\mathcal{U}}$ that sends $0$ to a critical point of $P_\Gamma$. Then, $\phi_{\mathcal{U}}^{-1}\circ P_\Gamma\circ \phi_{\mathcal{U}}$ is a degree $k$ proper branched covering of $\D$, and hence a Blaschke product $B$ of degree $k$. Moreover, postcritical finiteness of $P_\Gamma$ implies postcritical finiteness of $B$. Hence, possibly after pre-composing $\phi_{\mathcal{U}}$ with an element of $\mathrm{Aut}(\D)$, we can assume that the postcritically finite degree $k$ Blaschke product $B$ is of the form $z\mapsto z^k$ (for instance, see \cite[Lemma~4.17]{milnor-hyp-comp}). 

3) Let $\mathcal{U}^+=\mathcal{U}$ be a polar Fatou component of $P_\Gamma$. Then $\pmb{\Phi}(\partial\mathcal{U}^+)$ is the boundary of a polar principal component $U^+$ of $\Omega(\Gamma)$. Suppose that 
$\partial U^-=A_{\Gamma, \mathrm{BS}}(\partial U^+)$; equivalently,  $U^-$ is the image polar principal component under $\widehat{A}_{\Gamma, \mathrm{BS}}$. By part (2) of Theorem~\ref{first_return_conf_model_thm}, the maps $A_{\Gamma, \mathrm{BS}}:\partial U^\pm\to\partial U^\mp$ are coverings of (a common) degree $k\geq 2$. 

Let $\mathcal{U}^-$ be the image of $\mathcal{U}^+$ under $P_\Gamma$, so $\pmb{\Phi}(\partial\mathcal{U}^-)=\partial U^-$. It now follows from the previous paragraph that $P_\Gamma:\partial \mathcal{U}^\pm\to\partial \mathcal{U}^\mp$ are coverings of (a common) degree $k\geq 2$. Hence, $P_\Gamma:\mathcal{U}^\pm\to\mathcal{U}^\mp$ are proper branched coverings of degree $k$. Choose (homeomorphic extensions of) Riemann uniformizations $\phi_\pm:\overline{\D}\to\overline{\mathcal{U}^\pm}$ that send $0$ to critical points of $P_\Gamma$ (in $\mathcal{U}^\pm$). Then, the maps $\phi_{\mp}^{-1}\circ P_\Gamma\vert_{\mathcal{U}^\pm}\circ \phi_{\pm}$ are degree $k$ proper branched coverings of $\D$, and hence are Blaschke products of degree $k$ with a critical point at the origin. We call these Blaschke products $B_+, B_-$ (respectively). By construction,
$$
\overline{\mathcal{U}^+}\ \xleftrightarrows[\quad P_\Gamma\quad ]{\quad P_\Gamma\quad }\ \overline{\mathcal{U}^-}
$$
is conjugate via the pair of conformal maps $\phi_\pm$ to the fiberwise Blaschke product dynamical system
$$
\mathbf{B}: \overline{\D}\times\{+,-\}\to\overline{\D}\times\{+,-\},\ (z,\pm)\mapsto (B_\pm(z), \mp).
$$
Finally, postcritical finiteness of $P_\Gamma$ implies postcritical finiteness of the above fiberwise Blaschke product dynamical system $\mathbf{B}$. By \cite[Lemma~4.17]{milnor-hyp-comp}, possibly after pre-composing $\phi_{\pm}$ with suitable elements of $\mathrm{Aut}(\D)$, both the degree $k$ Blaschke products $B_+, B_-$ can be chosen to be $z\mapsto z^k$. The result follows.
\end{proof}

\subsection{Mateability of Bowen-Series maps of Bers boundary groups}\label{mating_boundary_groups_subsec}

In Subsection~\ref{boundary_group_b_s_subsec}, we showed that if $\Gamma\in\partial\mathcal{B}(G_d)$ is obtained by pinching a measured lamination $\mathcal{L}^*$ on the surface $\D/G_d$ such that $\mathcal{L}$ is $A_{G_d, \mathrm{BS}}$-invariant, then the Bowen-Series map $A_{\Gamma, \mathrm{BS}}:\Lambda(\Gamma)\to\Lambda(\Gamma)$ (given by Proposition~\ref{b_s_inv_simp_closed_geod_lem}) admits a continuous, piecewise (complex) M{\"o}bius extension
$\widehat{A}_{\Gamma, \mathrm{BS}}: K(\Gamma)\setminus R_\Gamma\to K(\Gamma)$
that is orbit equivalent to $\Gamma$ on $\Lambda(\Gamma)$. 

The purpose of the current subsection is to show that the canonical extensions of these Bowen-Series maps can be conformally mated with all polynomials in the principal hyperbolic components of a suitable degree (in the sense of Subsection~\ref{mat_def_subsec}). The main theorem of this section is the following:

\begin{theorem}\label{bers_bdry_mating_thm}
Let $\Gamma\in\partial\mathcal{B}(G_d)$ admit a Bowen-Series map $A_{\Gamma,\mathrm{BS}}$. 
Further, let $P\in\mathcal{H}_{2d-1}$. Then, $\widehat{A}_{\Gamma, \mathrm{BS}}$ can be conformally mated with $P$.
\end{theorem}

\noindent {\bf Notation for the proof of Theorem~\ref{bers_bdry_mating_thm}:}\\
\noindent$\bullet$  We denote (as in the previous subsection) the topological conjugacy between $P_\Gamma\vert_{\mathcal{J}(P_\Gamma)}$ and $A_{\Gamma, \mathrm{BS}}\vert_{\Lambda(\Gamma)}$ by $\pmb{\Phi}$.

\noindent$\bullet$ Denote the equatorial components of $\Omega(\Gamma)$ by $V_1,\cdots, V_m$, and the polar components by $U^\pm_1,\cdots, U^\pm_n$ such that $A_{\Gamma, \mathrm{BS}}(\partial U^\pm_j)=\partial U^\mp_j$, for $j\in\{1,\cdots, n\}$.

\noindent$\bullet$ Let $\mathcal{V}_1,\cdots, \mathcal{V}_m$ be the equatorial Fatou components of $P_\Gamma$ such that $\pmb{\Phi}(\partial\mathcal{V}_i)=\partial V_i$ ($i\in\{1,\cdots, m\}$), and $\mathcal{U}^\pm_1,\cdots, \mathcal{U}^\pm_n$ be the polar Fatou components such that $\pmb{\Phi}(\partial\mathcal{U}^\pm_j)=\partial U_j^\pm$ ($j\in\{1,\cdots, n\}$). As $\pmb{\Phi}$ is a conjugacy, we have that $P_{\Gamma}(\mathcal{U}^\pm_j)=\mathcal{U}^\mp_j$, for $j\in\{1,\cdots, n\}$.

\noindent$\bullet$ By Theorem~\ref{first_return_conf_model_thm}, for each $i\in\{1,\cdots, m\}$, there exists a conformal map $\psi_i:\overline{\D}\to\overline{V_i}$ that conjugates (the canonical extension of) the Bowen-Series map $A_i$ of a punctured sphere Fuchsian group to $\widehat{A}_{\Gamma, \mathrm{BS}}\vert_{\overline{V_i}}$. We assume that $A_i:\bS^1\to\bS^1$ is a covering of degree $k_i$. Moreover, the same theorem asserts that for each $j\in\{1,\cdots, n\}$, there exists a pair of conformal maps $\psi_{j, \pm}:\overline{\D}\to\overline{U^\pm_j}$ that conjugates the fiberwise dynamical system 
$$
\mathbf{A}_j:\left(\overline{\D}\setminus R_{j,+}\right)\times\{+\}\bigsqcup\left(\overline{\D}\setminus R_{j,-}\right)\times\{-\}\to\overline{\D}\times\{+,-\},\ (z,\pm)\mapsto (\widehat{A}_{j,\pm}(z), \mp)
$$
 to
$$
\overline{U^+_j}\ \xleftrightarrows[\quad \widehat{A}_{\Gamma, \mathrm{BS}}\quad ]{\quad \widehat{A}_{\Gamma, \mathrm{BS}}\quad }\ \overline{U^-_j},
$$
(where $A_{j,\pm}$ are \pwfm maps satisfying the conditions of Corollary~\ref{hbs_first_return_cor}, and $R_{j,\pm}$ are the fundamental domains of $A_{j,\pm}$)
We can assume that both $A_{j,+}, A_{j,-}$ are circle coverings of (a common) degree $d_j$. By Corollary~\ref{hbs_first_return_cor}, $A_{j,-}\circ A_{j,+}$ and $A_{j,+}\circ A_{j,-}$ are higher Bowen-Series maps of punctured sphere Fuchsian groups.

\noindent$\bullet $ By Proposition~\ref{dyn_p_gamma_prop}, for each $i\in\{1,\cdots, m\}$, there exists a conformal map $\phi_i:\overline{\D}\to\overline{\mathcal{V}_i}$ that conjugates $z^{k_i}\vert_{\overline{\D}}$ to $P_{\Gamma}\vert_{\overline{\mathcal{V}_i}}$. Moreover, the same proposition states that for each $j\in\{1,\cdots, n\}$, there exists a pair of conformal maps $\phi_{j, \pm}:\overline{\D}\to\overline{\mathcal{U}^\pm_j}$ that conjugates the fiberwise dynamical system 
$$
\mathbf{B}_j:\overline{\D}\times\{+,-\} \to\overline{\D}\times\{+,-\},\ (z,\pm)\mapsto (z^{d_j}, \mp)
$$
to
$$
\overline{\mathcal{U}^+_j}\ \xleftrightarrows[\quad P_{\Gamma}\quad ]{\quad P_{\Gamma}\quad }\ \overline{\mathcal{U}^-_j}.
$$

\noindent$\bullet$ The filled Julia set $\mathcal{K}(P)$ of the polynomial $P$ in the principal hyperbolic component of degree $(2d-1)$ polynomials is a closed Jordan disk. Hence, there exists (a homeomorphic extension of) a conformal isomorphism $\kappa:\widehat{\C}\setminus\D\to\mathcal{K}(P)$ that conjugates a degree $(2d-1)$ Blaschke product $B$ (with an attracting fixed point in $\widehat{\C}\setminus\D$) to $P$. We choose a quasisymmetric homeomorphism $\eta:\bS^1\to\bS^1$ that conjugates $z^{2d-1}$ to $B$, and extend $\eta$ to a quasiconformal self-homeomorphism of $\widehat{\C}\setminus\D$, also called $\eta$.

\noindent$\bullet$ The B{\"o}ttcher coordinate of $P_\Gamma$ is denoted by $\phi_{P_\Gamma}:\widehat{\C}\setminus\overline{\D}\to\mathcal{B}_\infty(P_\Gamma)$. As the Julia set $\mathcal{J}(P_\Gamma)=\partial\mathcal{B}_\infty(P_\Gamma)$ is locally connected (recall that $P_\Gamma$ is postcritically finite), the map $\phi_{P_\Gamma}$ extends to a continuous semi-conjugacy between $z^{2d-1}\vert_{\bS^1}$ and $P_\Gamma\vert_{\mathcal{J}(P_\Gamma)}$.

\begin{equation}
\begin{tikzcd}
\mathcal{K}(P)\arrow[swap]{d}{P} & \widehat{\C}\setminus\D \arrow[swap]{l}{\kappa}  \arrow[swap]{d}{B} &  \widehat{\C}\setminus\D  \arrow{r}{\phi_{P_\Gamma}} \arrow[swap]{d}{z^{2d-1}} \arrow[swap]{l}{\eta} & \overline{\mathcal{B}_\infty(P_\Gamma)} \arrow{d}{P_\Gamma}\\
\mathcal{K}(P) & \widehat{\C}\setminus\D \arrow{l}{\kappa} & \ \widehat{\C}\setminus\D \arrow{l}{\eta} \arrow[swap]{r}{\phi_{P_\Gamma}} & \overline{\mathcal{B}_\infty(P_\Gamma)}.
\end{tikzcd}
\label{comm_diag_blaschke}
\end{equation}

\noindent$\bullet$  Diagram~\eqref{comm_diag_blaschke} shows the homeomorphisms $\kappa, \eta$, and $\phi_{P_\Gamma}$. Here, $\eta$ is a topological conjugacy between $z^{2d-1}$ and $B$ only on $\bS^1$, the map $\kappa$ is a topological conjugacy between $B\vert_{\widehat{\C}\setminus\D}$ and $P\vert_{\mathcal{K}(P)}$ that is conformal on the interior, and the map $\phi_{P_\Gamma}$ is a topological semi-conjugacy between $z^{2d-1}\vert_{\widehat{\C}\setminus\D}$ and $P_\Gamma\vert_{\overline{\mathcal{B}_\infty(P_\Gamma)}}$ that is conformal on the interior.

\begin{lemma}\label{david_existence_lem}
\noindent
\begin{enumerate}\upshape
\item For $i\in\{1,\cdots, m\}$, there exist homeomorphisms $H_i:\bS^1\to\bS^1$ conjugating $z^{k_i}\vert_{\bS^1}$ to $A_i$. Moreover, $H_i$ admits a continuous extension to a David homeomorphisms of $\D$.

\item For $j\in\{1,\cdots, n\}$, there exists a pair of homeomorphisms $H_j^\pm:\bS^1\to\bS^1$ conjugating 
$$
\mathbf{B}_j:\bS^1\times\{+,-\} \to\bS^1\times\{+,-\},\ (z,\pm)\mapsto (z^{d_j}, \mp)
$$
to$$
\mathbf{A}_j:\bS^1\times\{+,-\}\to\bS^1\times\{+,-\},\ (z,\pm)\mapsto (A_{j,\pm}(z), \mp).
$$
Moreover, $H_j^\pm$ continuously extend as David homeomorphisms of $\D$.
\end{enumerate}
\end{lemma}
\begin{proof}
1) Fix $i\in\{1,\cdots,m\}$, and set $H_i:=\psi_i^{-1}\circ\pmb{\Phi}\circ \phi_i:\bS^1\to\bS^1$. By construction, $H_i$ conjugates $z^{k_i}$ to $\widehat{A}_i$. As $A_i$ is the Bowen-Series map of a punctured sphere Fuchsian group, it is a \pwfm map without asymmetrically hyperbolic periodic break-points. The conclusion now follows from Proposition~\ref{extension_david_normal}.

2) Fix $j\in\{1,\cdots, n\}$, and set $H_j^\pm:=\psi_{j,\pm}^{-1}\circ\pmb{\Phi}\circ \phi_{j,\pm}:\bS^1\to\bS^1$. By construction, $\bS^1\times\{+,-\} \to\bS^1\times\{+,-\},\ (z,\pm)\mapsto (H_j^\pm(z), \pm)$ conjugates $\mathbf{B}_j$ to $\mathbf{A}_j$. Moreover, $A_{j,-}\circ A_{j,+}$ (respectively, $A_{j,+}\circ A_{j,-}$) is a higher Bowen-Series map of a punctured sphere Fuchsian group, and hence is 
a \pwfm map with no asymmetrically hyperbolic periodic break-point. Since $H_j^+$ (respectively, $H_j^-$) conjugates $z^{d_j^2}\vert_{\bS^1}$ to $A_{j,-}\circ A_{j,+}$ (respectively, $A_{j,+}\circ A_{j,-}$), Proposition~\ref{extension_david_normal} implies that $H_j^\pm$ can be continuously extended to David homeomorphisms of $\disk$.
\end{proof}

\noindent Abusing notation slightly, we will denote the David extensions of $H_j^\pm$ to $\D$ also by  $H_j^\pm$.

\noindent$\bullet$ The commutative diagram~\eqref{comm_diag_per_1} shows the conjugacies $\pmb{\Phi}$, $\psi_i, \phi_i$ and $H_i$ (associated with equatorial components). The map $\psi_i$ (respectively, $\phi_i$) conformally conjugates $\widehat{A}_i\vert_{\overline{\D}}$ to $\widehat{A}_{\Gamma, \mathrm{BS}}\vert_{\overline{V_i}}$ (respectively, $z^{k_i}\vert_{\overline{\D}}$ to $P_\Gamma\vert_{\overline{\mathcal{V}_i}}$); while $\pmb{\Phi}^{-1}:\partial V_i\to\partial\mathcal{V}_i$ conjugates $\widehat{A}_{\Gamma, \mathrm{BS}}$ to $P_\Gamma$, and $H_i=\psi_i^{-1}\circ\pmb{\Phi}\circ \phi_i$ is a topological conjugacy between $z^{k_i}$ and $\widehat{A}_i$ on $\bS^1$.

\begin{equation}
\begin{tikzcd}
\overline{V_i} \arrow[swap]{d}{\widehat{A}_{\Gamma, \mathrm{BS}}} & \overline{\D} \arrow[swap]{l}{\psi_i}  \arrow[swap]{d}{\widehat{A}_i} & \overline{\D}  \arrow[swap]{l}{H_i} \arrow[swap]{d}{z^{k_i}} \arrow{r}{\phi_i} & \overline{\mathcal{V}_i} \arrow[bend right=36,swap]{lll}{\pmb{\Phi}} \arrow{d}{P_\Gamma}\\
\overline{V_i} & \overline{\D} \arrow{l}{\psi_i} & \overline{\D} \arrow{l}{H_i} \arrow[swap]{r}{\phi_i} & \overline{\mathcal{V}_i} \arrow[bend left=36]{lll}{\pmb{\Phi}} .
\end{tikzcd}
\label{comm_diag_per_1}
\end{equation}

\noindent$\bullet$ The commutative diagram~\eqref{comm_diag_per_2} shows the conjugacies $\pmb{\Phi}, \psi_{j,\pm}, \phi_{j,\pm}$, and $H_j^\pm$ (associated with polar components). The maps $\psi_{j,\pm}$ (respectively, $\phi_{j, \pm}$) conformally conjugate $\overline{\D} \xleftrightarrows[\widehat{A}_{j,+}]{\widehat{A}_{j,-}} \overline{\D}$ to $\overline{U_j^+} \xleftrightarrows[\widehat{A}_{\Gamma, \mathrm{BS}}]{\widehat{A}_{\Gamma, \mathrm{BS}}} \overline{U_j^-}$ (respectively, $\overline{\D} \xleftrightarrows[z^{d_j}]{z^{d_j}}\overline{\D}$ to $\overline{\mathcal{U}_j^+} \xleftrightarrows[P_\Gamma]{P_\Gamma} \overline{\mathcal{U}_j^-}$). On the other hand, $H_j^\pm=\psi_{j,\pm}^{-1}\circ\pmb{\Phi}\circ \phi_{j,\pm}$ induce a topological conjugacy between $\bS^1 \xleftrightarrows[z^{d_j}]{z^{d_j}}\bS^1$ and $\bS^1 \xleftrightarrows[\widehat{A}_{j,+}]{\widehat{A}_{j,-}} \bS^1$.

\begin{equation}
\begin{tikzcd}
\overline{U_j^+} \arrow[swap]{d}{\widehat{A}_{\Gamma, \mathrm{BS}}} & \overline{\D} \arrow[swap]{l}{\psi_{j,+}}  \arrow[swap]{d}{\widehat{A}_{j,+}} & \overline{\D}  \arrow[swap]{l}{H_j^+} \arrow[swap]{d}{z^{d_j}} \arrow{r}{\phi_{j,+}} & \overline{\mathcal{U}_j^+} \arrow[bend right=36,swap]{lll}{\pmb{\Phi}} \arrow{d}{P_\Gamma}\\
\overline{U_j^-}  \arrow[swap]{u}{} & \overline{\D} \arrow[swap]{u}{\widehat{A}_{j,-}} \arrow{l}{\psi_{j,-}} & \overline{\D} \arrow{l}{H_j^-} \arrow[swap]{u}{z^{d_j}} \arrow[swap]{r}{\phi_{j,-}} & \overline{\mathcal{U}_j^-} \arrow[bend left=36]{lll}{\pmb{\Phi}} \arrow[swap]{u}{}.
\end{tikzcd}
\label{comm_diag_per_2}
\end{equation}

The strategy of the proof of Theorem~\ref{bers_bdry_mating_thm} can now be summarized as follows. We will start with the polynomial $P_\Gamma$, and topologically modify it to match the dynamics of $\widehat{A}_{\Gamma, \mathrm{BS}}$ and $P$. More precisely, we will replace the action of $P_\Gamma$ on its basin of infinity by the action of $P$ on its filled Julia set, and replace the $P_\Gamma$-action on each equatorial (respectively, polar) bounded Fatou component by $\widehat{A}_i$ (respectively, by $\widehat{A}_{j, \pm}$). This would produce a continuous map $\widetilde{F}$ defined on a subset of the topological $2$-sphere: the topological mating of $\widehat{A}_{\Gamma,\mathrm{BS}}$ and $P$. We will then equip this $2$-sphere with an $\widetilde{F}$-invariant almost complex stricture that satisfies the David condition~\eqref{david_cond}. This will allow us to uniformize the aforementioned almost complex structure by a David homeomorphism, and produce a complex-analytic map $F$ (conjugating $\widetilde{F}$ by the David homeomorphism) defined on a subset of $\widehat{\C}$. Finally, the construction of $\widetilde{F}$ will reveal that $F$ is a conformal mating of $\widehat{A}_{\Gamma, \mathrm{BS}}:K(\Gamma)\setminus R_\Gamma\to K(\Gamma)$ and $P:\mathcal{K}(P)\to\mathcal{K}(P)$.

\begin{proof}[Proof of Theorem~\ref{bers_bdry_mating_thm}]
In the interest of clarity, we will split the proof into various steps.

\noindent\textbf{The topological mating $\widetilde{F}$.}
We define a (partially defined) continuous map
$$
\widetilde{F}:=
\begin{cases}
\left(\phi_{P_\Gamma}\circ\eta^{-1}\right)\circ B\circ\left(\eta\circ\phi_{P_\Gamma}^{-1}\right),\ {\rm on\ } \mathcal{B}_\infty(P_\Gamma), \\
\left(\phi_i\circ H_i^{-1}\right)\circ \widehat{A}_i\circ\left(H_i\circ \phi_i^{-1}\right),\ {\rm on\ } \mathcal{V}_i\setminus \left(\phi_i\circ H_i^{-1}\right)(R_i), \\
\left(\phi_{j,\mp}\circ (H_j^\mp)^{-1}\right)\circ \widehat{A}_{j,\pm}\circ \left(H_j^\pm \circ (\phi_{j,\pm})^{-1}\right),\ {\rm on}\ \mathcal{U}_j^\pm\setminus \left(\phi_{j,\pm}\circ (H_j^\pm)^{-1}\right)(R_{j,\pm}),\\
P_\Gamma, \quad {\rm on\ } \mathcal{K}(P_\Gamma)\setminus\left(\bigcup_{i=1}^m \mathcal{V}_i\cup\bigcup_{j=1}^n\mathcal{U}_j^\pm\right),
\end{cases}
$$
where $i\in\{1,\cdots, m\}, j\in\{1,2,\dots, n\}$, and $R_i, R_{j,\pm}$ are the fundamental domains of $A_i, A_{j,\pm}$, respectively (see the commutative diagrams~\eqref{comm_diag_per_1},~\eqref{comm_diag_per_2}, and~\eqref{comm_diag_blaschke}). We will denote the domain of definition of $\widetilde{F}$ by $\mathrm{Dom}(\widetilde{F})$.

\noindent\textbf{An $\widetilde{F}$-invariant almost complex structure $\mu$.} 
Let $\mu\vert_{\mathcal{B}_\infty(P_\Gamma)}$ be the pullback to $\mathcal{B}_\infty(P_\Gamma)$ of the standard complex structure on $\widehat{\C}\setminus\overline{\D}$ under the map $\eta\circ\phi_{P_\Gamma}^{-1}$. As $B$ is complex-analytic (i.e., it preserves the standard complex structure), it follows that $\mu\vert_{\mathcal{B}_\infty(P_\Gamma)}$ is $\widetilde{F}$-invariant.

Next, for $i\in\{1,\cdots, m\}$, we set $\mu\vert_{\mathcal{V}_i}$ to be the pullback to $\mathcal{V}_i$ of the standard complex structure on $\D$ under $H_i\circ \phi_i^{-1}$. As in the previous paragraph, since $\widehat{A}_i$ is complex-analytic, we have that $\left(\widetilde{F}\vert_{\mathcal{V}_i}\right)^\ast(\mu\vert_{\mathcal{V}_i})=\mu\vert_{\mathcal{V}_i}$. 
On the other hand, for $j\in\{1,2,\dots, n\}$, we set $\mu\vert_{\mathcal{U}_j^\pm}$ to be the pullback to $\mathcal{U}_j^\pm$ of the standard complex structure on $\D$ under $H_j^\pm \circ (\phi_{j,\pm})^{-1}$. The fact that $\widehat{A}_{j,\pm}$ preserve the standard complex structure now implies that $\left(\widetilde{F}\vert_{\mathcal{U}_j^+}\right)^\ast(\mu\vert_{\mathcal{U}_j^-})=\mu\vert_{\mathcal{U}_j^+}$ and $\left(\widetilde{F}\vert_{\mathcal{U}_j^-}\right)^\ast(\mu\vert_{\mathcal{U}_j^+})=\mu\vert_{\mathcal{U}_j^-}$. 

We now use the iterates of $P_\Gamma$ to pull back the complex structure $\mu$ to all the strictly pre-periodic Fatou components of $P_\Gamma$. As the Julia set of a postcritically finite polynomial has zero area, this procedure defines an $\widetilde{F}$-invariant measurable complex structure $\mu$ on $\widehat{\C}$.

\noindent\textbf{$\mu$ is a David coefficient.}
We will now argue that $\mu$ is a David coefficient on $\widehat{\C}$; i.e., it satisfies condition~\eqref{david_cond} of Definition \ref{def-david}.
Since $P_\Gamma$ is postcritically finite with a connected Julia set, each Fatou component $\mathcal{V}_i, \mathcal{U}_j^\pm$ is a John domain \cite[\S 7, Theorem~3.1]{CG1}. By \cite[Proposition~2.5 (part iv)]{LMMN}, the map $H_{i}\circ\phi_{i}^{-1}:\mathcal{V}_i\to\D$ (respectively, $H_j^\pm \circ (\phi_{j,\pm})^{-1}:\mathcal{U}_j^\pm\to\D$) is a David homeomorphism, and hence, $\mu$ is a David coefficient on $\bigcup_{i=1}^m\mathcal{V}_i\cup\bigcup_{j=1}^n \mathcal{U}_j^\pm$. Moreover, since $\eta\circ\phi_{P_\Gamma}^{-1}:\mathcal{B}_\infty(P_\Gamma)\to\widehat{\C}\setminus\overline{\D}$ is a quasiconformal homeomorphism, we have that 
$\vert\vert\mu\vert_{\mathcal{B}_\infty(P_\Gamma)}\vert\vert_\infty <1.$ 
Therefore, there exist constants $C,\alpha, \varepsilon_0>0$ such that
\begin{align}
\sigma\left(\lbrace z\in \bigcup_{i=1}^m\mathcal{V}_i\cup\bigcup_{j=1}^n \mathcal{U}_j^\pm\colon \vert\mu(z)\vert\geq 1-\varepsilon\rbrace\right) \leq Ce^{-\alpha/\varepsilon}, \quad \varepsilon\leq \varepsilon_0,
\label{david_on_bounded_periodic}
\end{align}
and 
\begin{align}
\sigma\left(\lbrace z\in \mathcal{B}_\infty(P_\Gamma)\colon \vert\mu(z)\vert\geq 1-\varepsilon\rbrace\right) = 0, \quad \varepsilon\leq \varepsilon_0,
\label{david_on_basin_infty}
\end{align}
where $\sigma$ is the spherical measure.

It remains to check the David condition on the union of the strictly pre-periodic Fatou components of $P_\Gamma$. By Proposition~\ref{dyn_p_gamma_prop}, there is a neighborhood of the closure of the union of the strictly pre-periodic Fatou components of $P_\Gamma$ that is disjoint from the critical points of $P_\Gamma$. Hence, if a strictly preperiodic Fatou component $\mathcal{U}'$ lands on $\mathcal{V}_i$ or $\mathcal{U}_j^\pm$ under $P_\Gamma^{\circ r}$ (where $r$ is the smallest positive integer with this property), then by the Koebe distortion theorem, $P_\Gamma^{\circ r}\circ \lambda_{\mathcal{U}'}$ is an $L$-bi-Lipschitz map between $\frac{1}{\diam{\mathcal{U}'}}\mathcal{U}'$ and $P_\Gamma^{\circ r}(\mathcal{U}')$, for some absolute constant $L\geq1$, where $\lambda_{\mathcal{U}'}(z)=\diam{\mathcal{U}'}\cdot z$ is a scaling map. This implies that, given any $0<\varepsilon\leq\varepsilon_0$,  
\begin{align}
\sigma\left(\{z\in \mathcal{U}'\colon|\mu(z)|\geq 1-\varepsilon\}\right)\leq L^2(\diam{\mathcal{U}'})^2\sigma\left(\{z\in P_\Gamma^{\circ r}(\mathcal{U}')\colon|\mu(z)|\geq 1-\varepsilon\}\right).
\label{koebe_control}
\end{align}
Moreover, since all the Fatou components $\mathcal{U}'$ are uniform John domains \cite[Proposition~10]{Mihalache}, it follows from \cite[p. 444]{Dimitrios-semi-hyperbolic} that there exists a constant $C_1>0$ such that
\begin{align}
(\diam{\mathcal{U}'})^2\leq C_1\sigma\left(\mathcal{U}'\right),
\label{diam_area_inequality}
\end{align}
for all strictly pre-periodic Fatou components $\mathcal{U}'$ of $P_\Gamma$.

Putting inequalities~\eqref{david_on_bounded_periodic}, \eqref{david_on_basin_infty},~\eqref{koebe_control}, and~\eqref{diam_area_inequality} together, we obtain our desired exponential decay criterion:
\begin{align*}
\begin{split}
&\sigma\left(\{z\in\widehat{\mathbb C}\colon|\mu(z)|\geq 1-\varepsilon\}\right)\\
&=\sigma\left(\lbrace z\in \mathcal{F}_{\mathrm{per}}\colon \vert\mu(z)\vert\geq 1-\varepsilon\rbrace\right)+ 
\sum_{\mathcal{U}'\ \textrm{strictly\ preperiodic}}\sigma\left(\{z\in{\mathcal{U}'}\colon|\mu(z)|\geq 1-\varepsilon\}\right)\\
&\leq \left(L^2 C_1\left(\sum_{\mathcal{U}'\ \textrm{strictly\ preperiodic}}\sigma\left(\mathcal{U}'\right)\right)+1\right)\cdot
\left(\sigma\left(\lbrace z\in \mathcal{F}_{\mathrm{per}}\colon \vert\mu(z)\vert\geq 1-\varepsilon\rbrace\right)\right)\\
&\leq \left(L^2C_1\sigma\left(\widehat{\mathbb C}\right)+1\right)\cdot Ce^{-\alpha/\varepsilon},\quad \varepsilon\leq\varepsilon_0,
\end{split}
\end{align*}
where $\mathcal{F}_{\mathrm{per}}:=\bigcup_{i=1}^m\mathcal{V}_i\cup\bigcup_{j=1}^n \mathcal{U}_j^\pm$ is the union of the bounded periodic Fatou components of $P_\Gamma$.

\noindent\textbf{Straightening $\mu$.}
The David Integrability Theorem \cite{David88} \cite[Theorem~20.6.2, p.~578]{AIM09} now provides us with a David homeomorphism $\xi:\widehat{\C}\to\widehat{\C}$ such that the pullback of the standard complex structure under $\xi$ is equal to $\mu$. Conjugating $\widetilde{F}$ by $\xi$, we obtain the map 
$F:=\xi\circ\widetilde{F}\circ\xi^{-1}:\xi(\mathrm{Dom}(\widetilde{F}))\to\widehat{\C}.$
We set $\mathrm{Dom}(F):= \xi(\mathrm{Dom}(\widetilde{F}))$.

\noindent\textbf{Complex-analyticity of $F$.} 
Note that since $\mathcal{B}_\infty(P_\Gamma)$ is a John domain, its boundary $\mathcal{J}(P_\Gamma)$ is removable for $W^{1,1}$ functions \cite[Theorem~4]{Jones-Smirnov}. By \cite[Theorem~2.7]{LMMN}, $\xi(\mathcal{J}(P_\Gamma))$ is locally conformally removable. Hence, it suffices to show that $F$ is analytic on the interior of $\mathrm{Dom}(F)\setminus\xi(\mathcal{J}(P_\Gamma))$.

To this end, first observe that both the maps $\eta\circ\phi_{P_\Gamma}^{-1}$ and $\xi$ are David homeomorphisms on $\mathcal{B}_\infty(P_\Gamma)$ straightening $\mu\vert_{\mathcal{B}_\infty(P_\Gamma)}$ (the former map is, in fact, quasiconformal). By \cite[Theorem~20.4.19, p.~565]{AIM09}, $\eta\circ\phi_{P_\Gamma}^{-1}\circ\xi^{-1}$ is conformal. It now follows from the definitions of $\widetilde{F}$ and $F$ that $F$ is analytic on $\xi(\mathcal{B}_\infty(P_\Gamma))$. 

Again, for each $i\in\{1,\cdots, m\}$, both $H_i\circ \phi_i^{-1}$ and $\xi$ are David homeomorphisms on $\mathcal{V}_i$ straightening $\mu\vert_{\mathcal{V}_i}$. Hence by \cite[Theorem~20.4.19]{AIM09}, $H_i\circ\phi_i^{-1}\circ\xi^{-1}$ is conformal. Therefore, $F$ is analytic on $\xi(\bigcup_{i=1}^m\mathcal{V}_i)\cap\mathrm{Dom}(F)$. One similarly proves that $F$ is analytic on $\xi(\bigcup_{j=1}^n\mathcal{U}_j^\pm)\cap\mathrm{Dom}(F)$.

It remains to check that $F$ is analytic on $\xi(\mathcal{U}')$, where $\mathcal{U}'$ is a strictly pre-periodic Fatou component of $P_\Gamma$. Note that on such a component $\mathcal{U}'$, the map $\widetilde{F}$ is conformal. Hence, $\xi$ and $\xi\circ\widetilde{F}$ are David homeomorphisms on $\mathcal{U}'$ (that $\left(\xi\circ\widetilde{F}\right)\vert_{\mathcal{U}'}$ is David follows from \cite[Proposition~2.5, part (iv)]{LMMN}), and by $\widetilde{F}$-invariance of $\mu$, they both straighten $\mu$. As in the previous cases, this implies that $F\equiv \xi\circ\widetilde{F}\circ \xi^{-1}$ is analytic on $\xi(\mathcal{U}')$.

This completes the proof of the fact that $F$ is analytic on the interior of $\mathrm{Dom}(F)$.

\noindent\textbf{$F$ is a mating of $\widehat{A}_{\Gamma, \mathrm{BS}}$ and $P$.} 
The dynamical plane of $F$ can be written as the union of two invariant subsets
$\widehat{\C}=\xi\left(\overline{\mathcal{B}_\infty(P_\Gamma)}\right)\bigcup \xi\left(\mathcal{K}(P_\Gamma)\right).$
We will show that $P\vert_{\mathcal{K}(P)}$ is topologically semi-conjugate to $F\vert_{\xi\left(\overline{\mathcal{B}_\infty(P_\Gamma)}\right)}$, and $\widehat{A}_{\Gamma, \mathrm{BS}}: K(\Gamma)\setminus R_\Gamma\to K(\Gamma)$ is topologically conjugate to $F:\xi\left(\mathcal{K}(P_\Gamma)\right)\cap\mathrm{Dom}(F)\to\xi\left(\mathcal{K}(P_\Gamma)\right)$ such that the conjugacies are conformal on $\Int{\mathcal{K}(P)}$ and $\Int{K(\Gamma)}$ (respectively).

The arguments used in the previous step yield that $\xi\circ\phi_{P_\Gamma}\circ\eta^{-1}$ conformally conjugates $B\vert_{\widehat{\C}\setminus\overline{\D}}$ to $F\vert_{\xi\left(\mathcal{B}_\infty(P_\Gamma)\right)}$. Thus, 
$\mathfrak{X}_P:=\xi\circ\phi_{P_\Gamma}\circ\eta^{-1}\circ\kappa^{-1}:\Int{\mathcal{K}(P)}\to\xi\left(\mathcal{B}_\infty(P_\Gamma)\right)$ 
is a conformal conjugacy between $P$ and $F$ (see  commutative diagram~\eqref{comm_diag_blaschke}). Clearly, $\mathfrak{X}_P$ extends to a topological semi-conjugacy between $P\vert_{\mathcal{J}(P)}$ and $F\vert_{\xi\left(\mathcal{J}(P_\Gamma)\right)}$.

Recall that $\pmb{\Phi}^{-1}$ conjugates $\widehat{A}_{\Gamma, \mathrm{BS}}\vert_{\Lambda(\Gamma)}$ to $\widetilde{F}\vert_{\mathcal{J}(P_\Gamma)}$ (as $\widetilde{F}\equiv P_\Gamma$ on $\mathcal{J}(P_\Gamma)$). Hence, $\mathfrak{X}_\Gamma:=\xi\circ\pmb{\Phi}^{-1}:\Lambda(\Gamma)\to\xi\left(\mathcal{J}(P_\Gamma)\right)$ conjugates $\widehat{A}_{\Gamma, \mathrm{BS}}$ to $F$. We will complete the proof by arguing that $\mathfrak{X}_\Gamma$ continuously extends to $\Int{K(\Gamma)}$ as a conformal conjugacy between $\widehat{A}_{\Gamma, \mathrm{BS}}$ and $F$. 

To do so, let us first consider an equatorial component $V_i$, $i\in\{1,\cdots, m\}$. According to the commutative diagram~\ref{comm_diag_per_1}, $\pmb{\Phi}^{-1}\equiv\phi_i\circ H_i^{-1}\circ \psi_i^{-1}$ on $\partial V_i$. Also recall from the previous step that $\xi\circ\phi_i\circ H_i^{-1}:\D\to\xi(\mathcal{V}_i)$ is a conformal conjugacy between $\widehat{A}_i$ and $F$. Hence, the conformal map $\xi\circ\phi_i\circ H_i^{-1}\circ\psi_i^{-1}:V_i\to\xi(\mathcal{V}_i)$ continuously extends $\mathfrak{X}_\Gamma:\partial V_i\to\xi(\partial\mathcal{V}_i)$, and conjugates $\widehat{A}_{\Gamma, \mathrm{BS}}$ to $F$. We set $\mathfrak{X}_\Gamma\vert_{V_i}\equiv \xi\circ\phi_i\circ H_i^{-1}\circ\psi_i^{-1}$.

Similarly, the conformal map $\xi\circ\phi_{j,\pm}\circ \left(H_j^\pm\right)^{-1}\circ\psi_{j,\pm}^{-1}:U_j^\pm\to\xi(\mathcal{U}_j^\pm)$ continuously extends $\mathfrak{X}_\Gamma:\partial U_j^\pm\to\xi(\partial\mathcal{U}_j^\pm)$, and conjugates $\widehat{A}_{\Gamma, \mathrm{BS}}$ to $F$. We set $\mathfrak{X}_\Gamma\vert_{U_j^\pm}\equiv \xi\circ\phi_{j,\pm}\circ \left(H_j^\pm\right)^{-1}\circ\psi_{j,\pm}^{-1}$.

At this point, we have continuously extended $\mathfrak{X}_\Gamma:\Lambda(\Gamma)\to\xi\left(\mathcal{J}(P_\Gamma)\right)$ to the principal components of $\Omega(\Gamma)$ as a conformal conjugacy between $\widehat{A}_{\Gamma, \mathrm{BS}}$ and $F$. Let us now consider a non-principal component $U'$ of $\Int{K(\Gamma)}$. Then, there exists $p\in\mathbb{N}$ such that $\widehat{A}_{\Gamma, \mathrm{BS}}^{\circ p}$ maps $\overline{U'}$ homeomorphically onto the closure $\overline{U}$ of a principal component $U$ of $\Int{K(\Gamma)}$. Let $\mathcal{U}'$ (respectively, $\mathcal{U}$) be the strictly pre-periodic (respectively, periodic) Fatou component of $P_\Gamma$ such that $\pmb{\Phi}^{-1}(\partial U')=\partial\mathcal{U}'$ (respectively, $\pmb{\Phi}^{-1}(\partial U)=\partial\mathcal{U}$). Clearly, $F^{\circ p}$ maps $\xi(\overline{\mathcal{U}'})$ homeomorphically onto $\xi(\overline{\mathcal{U}})$. Let $^{\xi(\mathcal{U}')}F^{-p}$ be the (well-defined) inverse branch of $F^{\circ p}$ that maps $\xi(\overline{\mathcal{U}})$ onto $\xi(\overline{\mathcal{U}'})$. Since $\mathfrak{X}_\Gamma$ is a conjugacy on the whole limit set $\Lambda(\Gamma)$, we have 
$\mathfrak{X}_\Gamma\vert_{\partial U'}=\ ^{\xi(\mathcal{U'})}F^{-p}\circ\mathfrak{X}_\Gamma\vert_{\partial U}\circ \widehat{A}_{\Gamma, \mathrm{BS}}^{\circ p}:\partial U'\to\partial\xi(\mathcal{U}').$
We now (continuously) extend $\mathfrak{X}_\Gamma$ to $U'$ as the conformal map $^{\xi(\mathcal{U}')}F^{-p}\circ\mathfrak{X}_\Gamma\vert_{U}\circ \widehat{A}_{\Gamma, \mathrm{BS}}^{\circ p}: U'\to\xi(\mathcal{U}').$

Thus, we have extended the conjugacy $\mathfrak{X}_\Gamma$ (which was originally defined only on $\Lambda(\Gamma)$) conformally and equivariantly to all of $\Int{K_\Gamma}$. Since $\xi(\mathcal{J}(P_\Gamma))$ is locally connected, the diameters of the components of $\Int{\xi(\mathcal{K}(P_\Gamma))}$ tend to $0$. It is now easy to verify that the extension $\mathfrak{X}_\Gamma$ is a homeomorphism on $K(\Gamma)$. Consequently, $\mathfrak{X}_\Gamma$ is the desired topological conjugacy between $\widehat{A}_{\Gamma, \mathrm{BS}}: K(\Gamma)\setminus R_\Gamma\to K(\Gamma)$ and $F:\xi\left(\mathcal{K}(P_\Gamma)\right)\cap\mathrm{Dom}(F)\to\xi\left(\mathcal{K}(P_\Gamma)\right)$ such that the conjugacy is conformal on $\Int{\mathcal{K}(P)}$.
\end{proof}

Recall from Proposition~\ref{higher_b_s_inv_simp_closed_geod_lem} that if $\mathcal{L}^*$ consists of (simple, closed, non-peripheral) geodesics on $\disk/G_d$ represented by $g_2,\cdots, g_{d-1}$, then a group $\Gamma\in\partial\mathcal{B}(G_d)$ obtained by pinching $\mathcal{L}^*$ admits a higher Bowen-Series map that is orbit equivalent to $\Gamma$. For such a group $\Gamma$, one can construct a canonical extension of $A_{\Gamma, \mathrm{hBS}}$ to a suitable subset of $K(\Gamma)$, and describe conformal models of its dynamics on periodic components of $\Int{K(\Gamma)}$ as higher Bowen-Series maps of Fuchsian groups uniformizing spheres with a smaller number of punctures (in the spirit of Subsection~\ref{boundary_group_b_s_subsec}). Moreover, the proof of Theorem~\ref{bers_bdry_mating_thm} remains valid, mutatis mutandis, in the case of higher Bowen-Series maps producing conformal matings of higher Bowen-Series maps of Bers boundary groups and polynomials in principal hyperbolic components.

\section{Topological orbit equivalence rigidity}\label{sec-coe}
\begin{defn}\label{defn-toe} (see \cite{fisherwhyte_gafa} for instance)
	Let $\Gamma_1, \Gamma_2$ be   groups acting continuously on compact Hausdorff spaces $X_1, X_2$.  We say that $\Gamma_1, \Gamma_2$
	are \emph{topologically orbit equivalent}  if there exists a homeomorphism $\phi: X_1 \to X_2$ such that
	for  every $x\in X_1$, $\phi(\Gamma_1.x) = \Gamma_2.\phi(x)$. The homeomorphism $\phi$ is called a \emph{topological orbit equivalence} between $\Gamma_1$ and $\Gamma_2$.
\end{defn}

 If $\psi: \Gamma_1\to \Gamma_2$ is an isomorphism, and $\phi ( \gamma_1 . \phi^{-1}(y))= \psi (\gamma_1).y$ for all $\gamma_1 \in \Gamma_1$, we say that $\phi$ is a  \emph{topological conjugacy} between  the $\Gamma_1$ and $\Gamma_2$ actions. Note that a topological conjugacy is necessarily a topological orbit equivalence.
\begin{defn}\label{def-toe-rig}
	Let $\Gamma_1$ be   a group acting continuously on a
	compact Hausdorff space $X$. The action is said to be \emph{topological orbit equivalence rigid}, if any topological orbit equivalence $\phi$ between the action of $\Gamma_1$ on $X$ and 
	the 	action
	of a group $\Gamma_2$ on $Y$  is a conjugacy.
\end{defn}
Note that it follows from  Definition \ref{defn-toe}  that $\phi^{-1}(\Gamma_2.x) = \Gamma_1.\phi^{-1}(x)$. 
The maps $a: \Gamma_1 \times X \to \Gamma_2$ and $b: \Gamma_2 \times Y \to \Gamma_1$ are called \emph{cocycles associated to $\phi$} if they satisfy
\begin{equation}\label{eq-a}
	\phi(\gamma_1.x) = a(\gamma_1,x).\phi(x),\ \forall\ x \in X,\  \gamma_1 \in \Gamma_1,
\end{equation}
and
\begin{equation}\label{eq-b}
	\phi^{-1}(\gamma_2.y) = b(\gamma_2,y).\phi^{-1}(y),\ \forall\ y \in Y,\  \gamma_2 \in \Gamma_2.
\end{equation}

\begin{defn}\label{defn-coe} (see \cite{li-coe})
	Let $\Gamma_1, \Gamma_2, \phi$ be as in Definition \ref{defn-toe}. 
	If, in addition, the associated cocycles $a,b$ given by Equations \ref{eq-a}
	and \ref{eq-b}
	are continuous (where $\Gamma_i$ are equipped with the discrete topology), we say that $\Gamma_1, \Gamma_2$
	are \emph{continuously orbit equivalent} and $\phi$ is called a {continuous orbit equivalence}.
\end{defn}

For $\Gamma_1, \Gamma_2$ Fuchsian, the actions of $\Gamma_1, \Gamma_2$ on $X=Y=\bS^1$ are topologically free. It follows that 
\begin{enumerate}
	\item The maps $a,b$ are uniquely determined by Equations \ref{eq-a}
	and \ref{eq-b}
	\cite[Remark 2.7]{li-coe}.
	\item $a(\gamma\gamma', x) = a(\gamma, \gamma'.x)a(\gamma', x)$ 
	for all $\gamma, \gamma' \in \Gamma_1$ 
	by \cite[Lemma 2.8]{li-coe}.
\end{enumerate}

We note that when $\Gamma_1, \Gamma_2$ are Fuchsian, continuity of $a, b$
and density of $\Gamma_1, \Gamma_2$ orbits in $\bS^1$ immediately implies
that $a, b$ are constant on the $\bS^1$ factor. Hence,  $a(\gamma\gamma', x) = a(\gamma, \gamma'.x)a(\gamma', x)= a(\gamma, x)a(\gamma', x)$. Since $a$ is constant on the $\bS^1$ factor, we may write $a(\gamma, x)=a(\gamma)$,
for all $\gamma, \gamma' \in \Gamma_1$.  Thus,  
$a(\gamma\gamma') = a(\gamma)a(\gamma')$, i.e.,  $a$ is a homomorphism. Similarly, $b$ is a homomorphism. Equations \ref{eq-a}
and \ref{eq-b} now imply that	if $\Gamma_1, \Gamma_2$ are Fuchsian and $\phi$ is a continuous orbit equivalence, then  $\phi$ is a conjugacy. In particular, $\Gamma_1, \Gamma_2$ are isomorphic.

\subsection{Bi-orbit equivalence rigidity}
	Let $\alpha, \beta$ be  the induced actions  of $\Gamma_1, \Gamma_2$ on unordered pairs of points in $X_1, X_2$ (where 
	$X_1, X_2$ are as in Definition \ref{defn-toe}). 
	A topological orbit equivalence $\phi$ between $\Gamma_1, \Gamma_2$ as in Definition \ref{defn-toe} is said to be a \emph{topological bi-orbit equivalence} if the map $(\phi \times \phi)\big/\!\sim\,: (X_1 \times X_1)\big/\!\sim\,
	\longrightarrow (X_2 \times X_2)\big/\!\sim$ induced by $\phi$ is a topological orbit equivalence between $\alpha, \beta$. (Here $\sim$ denotes the equivalence relation that exchanges factors.) The actions of $\Gamma_1, \Gamma_2$ are said to be 
	topologically bi-orbit equivalent.

\begin{lemma}\label{lemma-pop} Let $\Gamma_1, \Gamma_2$ be Fuchsian lattices acting on $\bS^1$ and let $\phi$ be a topological orbit equivalence as in Definition \ref{defn-toe}.
	If $\phi$ is also a topological bi-orbit equivalence between the actions of $\Gamma_1, \Gamma_2$ (in the above sense), then $\phi$ is a conjugacy between the $\Gamma_i$ actions on $\bS^1$. Equivalently, the action of $\Gamma_1$ on  $(\bS^1 \times \bS^1)/\sim$ is 
	topological orbit equivalence rigid (in the sense of Definition \ref{def-toe-rig}).
\end{lemma}

\begin{proof}
	Note that for any $\gamma \in \Gamma_i$, $i=1,2$, the fixed set $F(\gamma)=\{p \in (\bS^1 \times \bS^1)/\sim : \gamma(p) = p \}$ consists of at most three points.
	In particular, $\{F(\gamma):\gamma\in\Gamma_i\}$ is of codimension two in $ (\bS^1 \times \bS^1)/\sim$, for $i=1,2$. By \cite[Theorem 3.4]{fisherwhyte_gafa}, $(\phi \times \phi)/\sim$
	conjugates $\alpha$ to $\beta$. Since $(\phi \times \phi)/\sim$ preserves the diagonal $\Delta:= \{ \{x,x\} : x \in \bS^1 \}$, it follows that $\phi$ is a conjugacy between the $\Gamma_i$ actions on $\bS^1$.
\end{proof}

Using the same argument as in the proof of Lemma \ref{lemma-pop}, we obtain:

\begin{lemma}\label{lemma-pophyp} Let $G_1, G_2$ be one-ended hyperbolic
	groups.  Let $\phi$ be a topological bi-orbit equivalence between
	the $G_1, G_2$ actions on the Gromov boundaries
	$\partial G_1$, $\partial G_2$.
	Then $\phi$ is a conjugacy. The same conclusion holds if we only assume
	a topological orbit equivalence between the $G_i$-actions
	on the set $\partial^2G_i = (\partial G_i \times \partial G_i \setminus \Delta)/\sim$ of \emph{distinct} unordered pairs on $\partial G_i$ (here $\Delta$ denotes the diagonal).
\end{lemma}

\begin{proof}
	We only need to observe that, since $G_i$ are one-ended, $\partial G_i$
	are connected ($i=1,2$). By \cite{bes-mess}, $\partial G_i$ are  locally connected ($i=1,2$) and has (topological) dimension at least one. The conclusion now follows by \cite[Theorem 3.4]{fisherwhyte_gafa}.
\end{proof}

\subsection{Failure of topological orbit equivalence rigidity for Fuchsian groups}\label{sec-toe-fail}
Let $\Gamma_1, \Gamma_2$ be groups acting on $X_1, X_2$ as in Definition
\ref{defn-toe}.
In contrast to \cite{fisherwhyte_gafa}, we shall see now that topological orbit equivalence rigidity fails for Fuchsian groups. 
The following Lemma was kindly provided to us by an anonymous referee. It has allowed us to simplify
the proof of  Theorem \ref{thm-goe-fuchsian} and also strengthen the conclusion 
from an earlier version.
Identify the circle $\bS^1$ with $\R/\Z$, so that for any $d \geq 1$ the map $z \to z^d$ on the unit circle is conjugate to the map $m_d(x) = dx\, \mod\, 1$.

\begin{lemma}\label{lemma-referee}
 For any $d_1, d_2 \geq 2$, the maps $m_{d_1}$ and $m_{d_2}$ are topologically orbit equivalent. Hence $z\to z^{d_1}$ and $z\to z^{d_2}$
 are topologically orbit equivalent on $\bS^1$.
\end{lemma}
\begin{proof}
Fix $d\geq 2$. It suffices, by induction, to prove that $m_d$ is 
topologically orbit equivalent to $m_{d+1}$. Let $p = \frac{1}{d(d-1)}.$
Then $m_d(p) =m_d^2(p)$, and the restriction of $m_d$ to $[0, p] $ is an expanding homeomorphism onto its image. Set
\begin{equation*}
f(x) \ = \ \left\{\begin{array}{ll}
	m_d^2(x), & x \in [0,p], \\
	m_d(x)	, & x \in [p,1].
	\end{array}\right. 
\end{equation*}
Since $m_d(p) =m_d^2(p)$, $f$ is 
continuous. Further, since $f$ is an  expanding circle map of topological degree $d + 1$, $f$ is topologically conjugate to
$z \to z^{d+1}$. To prove that
$m_d$ is 
topologically orbit equivalent to $m_{d+1}$,  it then suffices to show that for any $x\in [0,1]$ we have equality of grand orbits
$\mathrm{GO}_{m_d} (x) = \mathrm{GO}_f (x).$
 By construction of $f$, $\mathrm{GO}_f (x) \subset \mathrm{GO}_{m_d} (x)$. To prove the reverse inclusion, it is enough to show that
 for any $x \in [0, 1)$, the points $x$ and $m_d(x)$ lie in $\mathrm{GO}_f (x)$. Consider the two sets $$A_0 = \{a\geq  0\, :\, \exists k \geq 0; f^k(x) = {m_d}^a(x)\}, \ {\rm and}$$
$$A_1 = \{a\geq  0\, :\, \exists k \geq  0; f^k(m_d(x)) = {m_d}^a(x)\} .$$
  By definition, $0 \in  A_0$ and $1 \in A_1$. If $A_0 \cap A_1 \neq \emptyset$, then $x$ and  $m_d (x)$ lie in $\mathrm{GO}_f(x)$, and we are done. Note that, for any natural number $n$, the set $\{n, n+1\}$  intersects both $A_0$ and $A_1$. Hence, $A_0 \cap A_1 =\emptyset$  only  if $A_0$ is the set of even numbers and $A_1$ is the set of odd numbers. However, this is possible only if $m_d^k(x) \in [0, p]$ for all $k \geq 0$. Since $f$ is an expanding homeomorphism onto its image on $[0, p]$, this forces $x= 0$.  Since $f(0) = m_d(0)$, the lemma follows.
\end{proof}

\begin{theorem}\label{thm-goe-fuchsian}
	Let $\Gamma_1,\Gamma_2$ be  punctured sphere Fuchsian groups, where the number of punctures are
	$k_1, k_2$  respectively. 
	Then the actions of   $\Gamma_1,\Gamma_2$ on $\bS^1$ are topologically orbit equivalent.
	Hence topological orbit equivalence rigidity fails for Fuchsian groups.
\end{theorem}
\begin{proof}
	Let $\Sigma_i = \D/\Gamma_i$, $i=1,2$. By Proposition~\ref{b_s_poly_conjugate_prop_1}, $\Gamma_i$ is orbit equivalent to its Bowen-Series map $A_{\Gamma_i}$, which  in turn is  topologically conjugate to $z^{d_i}\vert_{\bS^1}$ for some $d_i \geq 3$.
	Hence, by Lemma \ref{lemma-referee}, the actions of $A_{\Gamma_1}$ and $A_{\Gamma_2}$ on $\bS^1$
	are topologically orbit equivalent to each other. Hence, the actions of
	${\Gamma_1}$ and ${\Gamma_2}$ on $\bS^1$
	are topologically orbit equivalent to each other.
	If $k_1 \neq k_2$, then 
		$\Gamma_1,\Gamma_2$ are not isomorphic. In particular, the actions of
	$\Gamma_1,\Gamma_2$ on $\bS^1$ are not conjugate and  topological orbit equivalence rigidity fails.
\end{proof}

We conclude with the following question that remains to be resolved:
 Classify actions of Fuchsian lattices on $\bS^1$ up to topological orbit equivalence.

\newcommand{\etalchar}[1]{$^{#1}$}

\end{document}